\documentclass[a4paper]{article}

\usepackage{fullpage}

\usepackage{amsmath,amsthm}
\usepackage{amsfonts,amssymb}
\usepackage{enumitem}

\usepackage{graphicx}
\usepackage{xcolor}
\definecolor{cb2blue}{RGB}{55,126,184}
\definecolor{cb2green}{RGB}{77,175,74}
\definecolor{cb2red}{RGB}{228,26,28}

\usepackage[colorlinks,
	citecolor=cb2green,
	linkcolor=cb2blue,
	urlcolor=cb2red]{hyperref}
\usepackage[ruled,vlined,linesnumbered,algosection,algo2e]{algorithm2e} 
\usepackage[capitalize,noabbrev,nameinlink]{cleveref} 

\newcommand{\emailLink}[1]{\href{mailto:#1}{#1}}
\newcommand{\orcidLink}[1]{\href{https://orcid.org/#1}{#1}}
\newcommand{\amsmscLink}[1]{\href{http://www.ams.org/mathscinet/msc/msc2020.html?t=#1}{#1}}

\usepackage{mathtools}
\usepackage{nicefrac}

\usepackage{booktabs}  
\usepackage{array}     

\theoremstyle{plain}
\newtheorem{theorem}{Theorem}[section]
\newtheorem{corollary}[theorem]{Corollary}

\newtheorem{lemma}[theorem]{Lemma}
\newtheorem{proposition}[theorem]{Proposition}
\newtheorem{assumption}[theorem]{Assumption}
\theoremstyle{definition}
\newtheorem{definition}[theorem]{Definition}
\theoremstyle{remark}
\newtheorem{remark}[theorem]{Remark}

\AddToHook{env/corollary/begin}{\crefalias{theorem}{corollary}}
\AddToHook{env/example/begin}{\crefalias{theorem}{example}}
\AddToHook{env/lemma/begin}{\crefalias{theorem}{lemma}}
\AddToHook{env/proposition/begin}{\crefalias{theorem}{proposition}}
\AddToHook{env/assumption/begin}{\crefalias{theorem}{assumption}}
\AddToHook{env/definition/begin}{\crefalias{theorem}{definition}}
\AddToHook{env/remark/begin}{\crefalias{theorem}{remark}}

\AtBeginEnvironment{remark}{%
	\pushQED{\qed}%
}
\AtEndEnvironment{remark}{\popQED}

\usepackage{tikz}
\usetikzlibrary{positioning,arrows.meta}

\DeclareMathOperator{\dom}{dom}
\DeclareMathOperator{\prox}{prox}

\DeclareMathOperator*{\minimize}{minimize}
\DeclareMathOperator{\dist}{dist}
\DeclareMathOperator{\interior}{int}

\DeclareMathOperator{\stt}{subject~to}
\newcommand{\indicator}{\iota}

\newcommand{\N}{\mathbb{N}}
\newcommand{\R}{\mathbb{R}}

\newcommand{\innerprod}[2]{\langle#1,#2\rangle}
\newcommand{\bigO}{\mathcal{O}}
\newcommand{\emphdef}[1]{\emph{#1}}
\renewcommand{\emptyset}{\varnothing}

\newcommand{\Hs}{H}
\newcommand{\Hsobs}{\Hs_{\mathrm{obs}}}
\newcommand{\Us}{U}
\newcommand{\Usad}{\Us_{\mathrm{ad}}}
\newcommand{\barUsad}{\bar{\Us}_{\mathrm{ad}}}

\newcommand{\meas}{{\mathfrak{m}}}

\newcommand{\cost}{\mathcal{J}} 			
\newcommand{\costr}{\mathcal{F}} 	
\newcommand{\costp}{\mathcal{R}} 	
\newcommand{\costs}{\mathcal{Q}} 			
\newcommand{\costb}{\cost_{\nu}} 	
\newcommand{\costrD}{{\nabla \costr}} 	

\newcommand{\barrierfun}{\mathcal{B}}

\newcommand{\gmap}{{\mathcal{G}}}	
\newcommand{\ssize}{{\alpha}} 						
\newcommand{\ssizek}{\alpha_j} 					
\newcommand{\ssizeinit}{{\alpha_{\mathrm{init},j}}}	
\newcommand{\ssizelb}{\ssize_{\min}}				
\newcommand{\ssizeub}{\ssize_{\max}}				
\newcommand{\mmax}{{m_{\mathrm{nm}}}}
\newcommand{\Lsmooth}{L_{\mathrm{smooth}}}

\newcommand{\yd}{{y_d}}
\newcommand{\Sctl}{S_{\mathrm{ctl}}}

\newcommand{\coloneqq}{:=}

\newcommand{\Niterout}{N_{\rm out}}
\newcommand{\Niterin}{N_{\rm in}}
\newcommand{\Nitertot}{N_{\rm tot}}

\newcommand{\stringlog}{\texttt{log}}
\newcommand{\stringinverse}{\texttt{inverse}}
\newcommand{\stringloglike}{\texttt{loglike}}
\newcommand{\unitsphere}{\mathbb{B}_1}
\newcommand{\mucoh}{\mu_{\text{coh}}}

\newcommand{\TheAuthorADM}{Alberto De~Marchi}
\newcommand{\TheEmailADM}{alberto.demarchi@unibw.de}
\newcommand{\TheOrcidADM}{0000-0002-3545-6898}
\newcommand{\TheAffiliationADM}{%
	Institute of Applied Mathematics and Scientific Computing,
	Department of Aerospace Engineering,
	University of the Bundeswehr Munich,
	85577 Neubiberg, Germany%
}
\newcommand{\TheAuthorBA}{Behzad Azmi}
\newcommand{\TheEmailBA}{behzad.azmi@uni-konstanz.de}
\newcommand{\TheOrcidBA}{0000-0002-6303-254X}
\newcommand{\TheAffiliationBA}{%
    Department of Mathematics and Statistics,
    University of Konstanz,
    78457 Konstanz, Germany%
}

\newcommand{\TheTitle}{Interior-Point Proximal Methods for Nonsmooth Optimization in Hilbert Spaces with Cone-Ordered Constraints}

\newcommand{\TheAbstract}{%
We study an inexact interior-point method for nonsmooth, possibly nonconvex optimization in a Hilbert space with inequality constraints ordered by a cone in a Banach lattice, with particular emphasis on infinite-dimensional state-constrained optimal control.
The objective function is given by the sum of a smooth, possibly nonconvex term and a convex, possibly nonsmooth term with a computable proximal mapping.
The constraints are formulated by means of an order cone in a Banach lattice.
This setting covers finite-dimensional nonsmooth nonlinear problems with componentwise constraints as well as infinite-dimensional PDE-constrained optimization problems with pointwise state constraints.
The method is based on barrier-regularized subproblems, which are solved inexactly by a proximal-gradient method.
We consider logarithmic and power-type barriers and derive the differentiability and curvature estimates needed for the convergence and complexity analysis.
Under suitable constraint qualifications and compactness assumptions, we establish approximate KKT conditions for the original problem and convergence of the inexact interior-point sequence.
For logarithmic and power-type barriers, we derive complementarity estimates and outer iteration bounds; the corresponding power-barrier rates and total inner–outer complexity bounds are stated under explicit barrier-path and uniform smoothness assumptions.
For convex problems, we obtain stronger convergence results.
We apply the framework to state-constrained semilinear elliptic optimal control and sparse dictionary learning with nonlinear side constraints.
Numerical experiments illustrate the proposed method.
}

\newcommand{\TheKeywords}{Proximal methods, interior-point methods, PDE-constrained optimization, worst-case complexity.}

\newcommand{\TheMSCClass}{\amsmscLink{65K05}, 
	\amsmscLink{90C51}, 
	\amsmscLink{65Y20}, 
	\amsmscLink{49M41}. 
}
	
\begin{document}
	
\title{\bfseries \TheTitle}
\author{\TheAuthorBA\thanks{%
		\TheAffiliationBA,
		\textsc{email} \emailLink{\TheEmailBA},
		\textsc{orcid} \orcidLink{\TheOrcidBA}.%
	}\and%
	\TheAuthorADM\thanks{%
		\TheAffiliationADM,
		\textsc{email} \emailLink{\TheEmailADM},
		\textsc{orcid} \orcidLink{\TheOrcidADM}.%
	}%
}%
\date{}

\maketitle

\begin{abstract}
    \TheAbstract


    \paragraph*{AMS MSC} \TheMSCClass
\end{abstract}

\section{Introduction}\label{sec:generic}

Interior-point methods are a central paradigm in continuous optimization, in which the original problem is replaced by a sequence of barrier-regularized subproblems
\cite{fiacco1964sequential,Forsgren2002interior,Nesterov1994Interior,Wright1997Primal}.
Interior-point methods have also been extended to infinite-dimensional and PDE-constrained optimization problems
\cite{Ulbrich2009Primal,Schiela2009barrier}.
Most of the available theory concerns problems with smooth objectives, for which the barrier subproblems can be treated by smooth optimization methods.
Much less is known when the objective also contains a nonsmooth term with a computable proximal mapping.

In this paper, we study an inexact interior-point method for nonsmooth optimization problems subject to conic inequality constraints.
More precisely, we consider problems of the form
\begin{equation}\label{eq:Reduced}
    \minimize_{u \in \Us}\;
    \cost(u)\coloneqq \costr(u)+\costp(u)
    \quad\stt\quad
    H(u)\le_Z 0 \ \text{in } Z 
\end{equation}
where \(\Us\) is a Hilbert space and the objective \(\cost\) consists of a smooth, possibly nonconvex term \(\costr\) and a convex, possibly nonsmooth term \(\costp\) with computable proximal mapping.
Inequality \(H(u)\le_Z 0\) is understood with respect to an order structure in \(Z\) and the constraint mapping \(H\) may be nonlinear, so the feasible set need not be convex.

The formulation \eqref{eq:Reduced} covers both finite- and infinite-dimensional constrained optimization problems.
In finite dimensions, taking $Z=\R^q$ with its componentwise order yields nonsmooth nonlinear problems with finitely many inequality constraints.
This includes sparse and regularized models from structured learning, signal processing, and imaging, such as basis-pursuit denoising, sparse nonnegative matrix factorization, and sparse dictionary learning.

The same framework also applies to infinite-dimensional problems whose constraints are naturally expressed through an order structure.
A relevant class studied in this paper is PDE-constrained optimization with state constraints.
Using the reduced formulation, pointwise or functional state constraints give rise to constraints of the form $H(u)\leq_Z 0$ in a suitable Banach lattice $Z$.
For pointwise state constraints, a typical choice is $Z=C(\overline{K})$, where $K\subset\R^n$ is a bounded open domain.

Interior-point methods are well-suited to \eqref{eq:Reduced} because the conic constraint can be incorporated through a barrier functional.
For a barrier parameter \(\nu>0\), this leads to subproblems of the form
\begin{equation}\label{eq:Reduced_barrier_problem}
    \minimize_{u \in \Us}\;
    \costb(u)\coloneqq \cost(u) + \nu \barrierfun (H(u)),
\end{equation}
where \(\barrierfun \) is defined on \(\interior  (Z_-)\), the interior of the negative cone.
Since the objective $\cost\coloneqq \costr+\costp$ contains the nonsmooth term \(\costp\), the barrier subproblems are not purely smooth.
Throughout, we regard \(\costb\) as an extended-real-valued functional on \(U\), with
$\costb(u)\coloneqq +\infty$ whenever $H(u)\not\in \interior Z_-$.
We therefore use a proximal-gradient method to compute approximate stationary points, without requiring exact subproblem minimizers.

The analysis of this scheme involves several difficulties.
The barrier term induces curvature that may grow as the barrier parameter decreases and as the iterates approach the boundary of the cone.
For integral barriers, blow-up of the scalar kernel does not imply blow-up of the functional at every boundary point of the cone.
Moreover, the inexact solution of the subproblems \eqref{eq:Reduced_barrier_problem} has to be related to the approximate stationarity of the original constrained problem \eqref{eq:Reduced}.
In function spaces, passing to the limit requires suitable compactness properties, since primal iterates and measure-valued multipliers may only converge weakly and weakly-$\ast$, respectively.

\subsection{Related work}
For nonlinear and conic optimization, interior-point and barrier methods have been studied extensively; see, for instance, \cite{Forsgren2002interior,gondzio2025interior}.
Classical analyses are mainly developed for finite-dimensional problems with smooth objective and constraint functions and are often based on Newton, trust-region, or primal-dual Newton-type steps.

More recently, barrier methods for nonconvex constrained optimization have been studied from the viewpoint of iteration complexity.
First- and second-order Hessian barrier methods for finite-dimensional nonconvex conic optimization were analyzed in \cite{Dvurechensky2025Hessian}, where complexity bounds for reaching approximate first- and second-order KKT points were derived.
First-order interior-point methods for linearly constrained smooth optimization were considered in \cite{Tseng2011afirst}.
More broadly, worst-case complexity theory for first-order methods in nonconvex composite optimization has been developed independently of interior-point structure \cite{cartis2010complexity,ghadimi2016accelerated}, underlying the $\bigO(\varepsilon^{-2})$-type bounds used in \cref{sec:inner_complexity}.
These results, however, do not treat the present combination of nonsmooth proximal terms, nonlinear cone-valued constraints, and infinite-dimensional spaces.

Another related direction concerns regularized optimization and proximal-gradient methods.
Proximal-gradient methods are standard for composite objectives with convex nonsmooth terms \cite{themelis2018forward,kanzow2022convergence,demarchi2023monotony,azmi2025nonmonotone,tseng2009coordinate}, and have also been combined with penalty, augmented-Lagrangian, and interior-point ideas \cite{demarchi2023constrained,hallak2023adaptive,demarchi2025penalty,monteiro2013iteration,chouzenoux2020proximal,demarchi2024interior}.
The closest predecessor is \cite{demarchi2024interior}, which combines interior-point and proximal-gradient methods for nonsmooth nonconvex optimization with smooth inequality constraints in finite dimensions.
The present work extends this approach to Hilbert-space controls and Banach-lattice-valued constraints, including integral barriers, measure-valued multipliers, weak/weak-\(\ast\) limit passages, and collective compactness of the adjoint derivative family.

Function-space interior-point and state-constrained optimal-control methods have also been studied, including affine-scaling, primal-dual, barrier, and path-following approaches \cite{Ulbrich2000super,Ulbrich2009Primal,weiser2005Interior,Schiela2009barrier,Prufert2008Convergence,Schiela2011anInterior,kruse2015self}.
The state-constrained PDE applications considered there are typically governed by linear equations, the corresponding optimization problems are largely convex, and the methods are generally based on smooth problem structures and Newton-type steps.
Moreau–Yosida regularization with semismooth Newton methods provides an alternative for state constraints \cite{ito2003semismooth,hintermueller2009moreau}.
These works mainly consider smooth or convex settings and do not combine the nonsmooth proximal inner solver, Banach-lattice constraint structure, and complexity analysis developed here.

The present paper combines several features that are not covered together in the works discussed above.

\paragraph*{Contributions}
The main contributions are as follows.
\begin{itemize}
    \item We study logarithmic and power-type integral barriers for cone-ordered constraints in Banach lattices and establish differentiability, curvature, and complementarity estimates.
    \item We analyze the inexact solution of nonsmooth barrier subproblems by a nonmonotone proximal-gradient method, proving finite termination of backtracking and deriving inner iteration bounds under suitable assumptions.
    \item Under compactness and constraint-qualification assumptions, we prove KKT-type convergence of the inexact interior-point sequence and derive barrier-dependent outer iteration complexity bounds.
    Under additional assumptions on the inner iterations, we also obtain total proximal-gradient complexity bounds.
    For both barrier types, we establish stronger optimality and convergence results in the convex and strongly convex settings.
    \item We investigate the applicability of the framework to state-constrained semilinear elliptic control problems, verifying the relevant regularity and compactness assumptions, and present numerical examples in PDE-constrained and finite-dimensional optimization.
\end{itemize}

\paragraph*{Organization}
The structure of the paper and the relations between its main results are summarized in \cref{fig:roadmap}.
\Cref{sec:problem_examples} introduces the problem class and establishes existence and first-order optimality conditions.
\Cref{sec:barrier_inner} develops the barrier and inner-loop analysis, \cref{sec:outer_scheme} the interior-point convergence and complexity results, \cref{sec:pde_verification} the PDE verification, and \cref{sec:numerics} the numerical experiments.

\begin{figure}[tbh]
	\centering%
	\begin{tikzpicture}[
	font=\small,
	node distance=8mm and 8mm,
	box/.style={
		rectangle, rounded corners=0pt, draw, very thin,
		align=center, inner sep=2mm, line width=0.5pt
	},
	solidarrow/.style={-{Stealth}, thick, solid, black},
	dashedarrow/.style={-{Stealth}, thick, dashed, red},
	]
	
	\newcommand{\lbl}[1]{{\footnotesize #1}}
	\newcommand{\sectionref}[1]{(\hyperref[#1]{\S\ref{#1}})}
	
	\node[box] (foundations)
	at (0,0)
	{\textbf{Foundations \sectionref{sec:problem_examples}} \\[2pt]
		\lbl{\cref{thm:existence} $\cdot$ existence of minimizers} \\
		\lbl{\cref{thm:KKT} $\cdot$ optimality conditions}};
	
	\node[box, right=of foundations] (barrier)
	{\textbf{Barrier \& inner loop \sectionref{sec:barrier_inner}} \\[2pt]
		\lbl{\cref{lem:central-path-Lip-onesided} $\cdot$ Lipschitz estimates} \\
		\lbl{\cref{sec:inner_complexity} $\cdot$ inner complexity}};
	
	\node[box, right=of barrier] (complexity)
	{\textbf{Total complexity \sectionref{sec:overall_complexity}} \\[2pt] 
		\lbl{\cref{cor:overall_complexity_unified}}};
	
	\node[box, below=of barrier] (outer)
	{\textbf{Interior point loop \sectionref{sec:outer_convergence}} \\[2pt]
		\lbl{\cref{prop:barrier-to-KKT} $\cdot$ convergence} \\
		\lbl{\cref{thm:outer-convergence} $\cdot$ outer complexity}};
	
	\node[box, right=of outer] (convex)
	{\textbf{Convex case} \\[2pt]
		\lbl{\cref{prop:convex_case}}};
	
	\node[box, left=of outer] (pde)
	{\textbf{PDE verification \sectionref{sec:pde_verification}} \\[2pt]
		\lbl{\cref{prop:unified_compact_stability_affine_Bf_Neumann,lem:Applicability}}};
	
	\draw[solidarrow] (foundations) -- (barrier);
	\draw[solidarrow] (barrier) -- (outer);
	\draw[solidarrow] (barrier) -- (complexity);
	\draw[solidarrow] (outer.north east) -- (complexity.south west);
	\draw[solidarrow] (outer) -- (convex);
	\draw[dashedarrow] (pde) -- (foundations);
	\draw[dashedarrow] (pde) -- (outer);
	
\end{tikzpicture}%
	\caption{Interrelations among different aspects and results presented in the following sections. Dashed arrows indicate application-specific verification.}%
	\label{fig:roadmap}%
\end{figure}
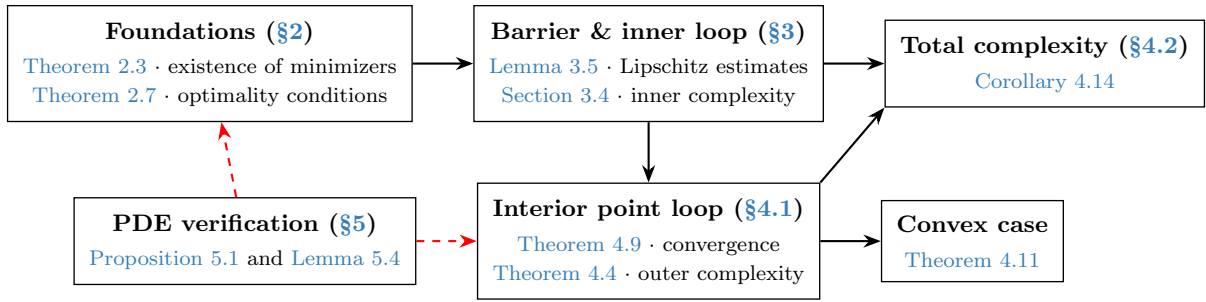

\section{Problem setting and optimality conditions}\label{sec:problem_examples}

\subsection{Problem class and examples}
We now collect the standing assumptions for the generic problem \eqref{eq:Reduced} and recall the order notation used throughout the paper.
The order relation on the Banach lattice \(Z\) is induced by its positive cone \(Z_+\): for \(z_1,z_2\in Z\), we write \(z_1\le_Z z_2\) whenever \(z_2-z_1\in Z_+\).
The negative cone is denoted by \(Z_- \coloneqq -Z_+\), so that the constraint \(H(u)\le_Z 0\) is equivalently written as \(H(u)\in Z_-\).
In particular, we write \(H(u)<_Z 0\) whenever \(H(u)\in \interior(Z_-)\).
Since \(\interior Z_-\) is open in \(Z\), strict feasibility is stable under sufficiently small perturbations in the \(Z\)-norm.

\begin{assumption}\label{ass:General}
    Throughout, the following conditions hold:
    \begin{enumerate}[label=(G\arabic*), start=0]
        \item\label{G0}%
        \(Z_-\) has nonempty interior.
        \item\label{G1}%
        \(\costp:\Us\to(-\infty,+\infty]\) is proper, convex, and lower semicontinuous (lsc).
        \item\label{G2}%
        There exists an open set \(\mathcal D\subset\Us\) with \(\dom\costp\subseteq\mathcal D\) such that $\costr:\mathcal D\to\R$ and $H:\mathcal D\to Z$ are continuously Fr\'echet differentiable on \(\mathcal D\).
    \end{enumerate}
\end{assumption}

\ref{G0} ensures that the barrier domain \(\interior Z_-\) is nonempty.
\ref{G1} covers the nonsmooth terms typically used for regularization, including sparsity-promoting penalties and indicator functions of closed convex sets.
\ref{G2} specifies the differentiability required for the smooth part of the objective and for the conic constraint mapping.
All derivatives below are evaluated only at points in $\mathcal{D}$.
The feasible and strictly feasible sets associated with \eqref{eq:Reduced} are denoted respectively by
\begin{equation}\label{eq:strict-feasible-domain}
    \Usad
    \coloneqq
    \{ u\in\dom\costp : H(u)\le_Z 0\,\}
    \quad\text{and}\quad
    \Usad^\circ\coloneqq\{u\in\dom\costp:H(u)<_Z0\}.
\end{equation}

\paragraph*{Infinite-dimensional examples}

Formulation \eqref{eq:Reduced} also applies to optimization problems governed by PDEs subject to state constraints.
A representative abstract problem is given by
\begin{subequations}\label{eq:Primal}
\begin{align}
    \minimize_{(y,u)\in Y\times U}{}&\;
    \frac12\|Cy-y_d\|_{\Hsobs}^2+\costp(u),
    \label{eq:Primal_obj} \\
    \stt{}&\quad
    Ay+N(y)=Bu+f
    &&\text{in }V^\ast,
    \label{eq:Abs_state} \\
    &\quad
    Sy\le_Z\bar y
    &&\text{in }Z .
    \label{eq:state_constraint}
\end{align}
\end{subequations}
Here, \(u\) denotes the control and \(y\) the corresponding state.
The state equation \eqref{eq:Abs_state} is understood in the weak sense, while \eqref{eq:state_constraint} imposes an order constraint on the state.
More precisely, \(A:V\to V^\ast\) denotes the operator associated with the linear part of the PDE, \(N:Y\to V^\ast\) represents its nonlinear part, and \(B:U\to V^\ast\) is the control operator.
Furthermore, \(C:Y\to\Hsobs\) denotes the observation operator, and \(S:Y\to Z\) maps the state to the quantity subject to the order constraint.
The nonsmooth term \(\costp\) can model, for example, convex control constraints through indicator functions, sparsity-promoting \(L^1\)-type penalties, or combinations of such terms.
We work with a Gelfand triple $V \hookrightarrow H \cong H^\ast \hookrightarrow V^\ast$ with dense and continuous embeddings.
The state space \(Y\) is chosen so that the state equation is well-posed, and the state constraint is meaningful in the ordered space \(Z\).
In pointwise state-constrained problems, a typical choice is \(Z=C(\overline K)\) for a bounded open set \(K \subset \R^d\) in the physical domain, equipped with the pointwise order.
Then \(Z^\ast\) can be identified with the space of finite signed Radon measures on \(\overline K\), which is one reason why multiplier boundedness and weak-$\ast$ convergence are delicate in infinite-dimensional analysis.
Assuming that the state equation is well posed, let \(\Sctl :U\to Y\) denote the control-to-state map.
If \(\Sctl\) is continuously Fr\'echet differentiable and \(U\) is identified with the Hilbert space \(\Us\) in \eqref{eq:Reduced}, then \eqref{eq:Primal} reduces to \eqref{eq:Reduced} with
\begin{equation}\label{eq:reduced_definitions}
    \costr(u)\coloneqq \tfrac12\|C \Sctl(u)-y_d\|_{\Hsobs}^2,
    \qquad
    H(u)\coloneqq S\,\Sctl(u)-\bar y .
\end{equation}
In \cref{sec:pde_verification}, we verify the abstract assumptions for a general class of state-constrained semilinear elliptic optimal control problems, and \cref{ex:pde-control} presents numerical results for a representative instance.

\begin{remark}
    The same abstract formulation may also apply to evolution equations, such as semilinear parabolic problems with suitable initial and boundary conditions.
    This would require choosing a state space in which the state constraint is meaningful as an order-cone constraint and verifying the corresponding well-posedness and regularity assumptions.
    We do not pursue this direction here.
\end{remark}

\subsection{Existence of minimizers}

We establish that problem \eqref{eq:Reduced} admits at least one solution, under conditions that cover both the finite-dimensional and infinite-dimensional Hilbert-space cases.

\begin{theorem}[Existence of minimizers]\label{thm:existence}
    Assume that the following conditions hold:
    \begin{enumerate}[label=(E\arabic*)]
        \item\label{E1} The reduced cost functional $\cost$ is bounded from below on the admissible set $\Usad$, i.e.,  $\inf_{u\in \Usad} \cost(u) > -\infty$.
        \item\label{E2} The mapping $u \mapsto \costr(u)$ is weakly lsc.
        \item\label{E3} Either $\cost$ is radially unbounded on $\Us$, i.e., $\lim_{\|u\|_{\Us}\to\infty}\cost(u)= \infty$, or the admissible set $\Usad$ is bounded in $\Us$.    
        \item\label{E4} The admissible control set $\Usad$ is nonempty and weakly sequentially closed in $\Us$.
    \end{enumerate}
    Then, \eqref{eq:Reduced} admits at least one minimizer.
\end{theorem}
\begin{proof}
    The argument is based on the direct method in the calculus of variations.
    Since $\Usad$ is nonempty, there exists a $\bar{u} \in \Usad$.
    Let $\{u_n\}_{n\in\N} \subset \barUsad \coloneqq \Usad \cap \bigl\{\, u \in \Us \;|\; \cost(u) \le  \cost(\bar{u}) \,\bigr\}$ be a minimizing sequence for \eqref{eq:Reduced}.
    Due to \ref{E2}--\ref{E4} and \ref{G1}, $\barUsad$ is sequentially weakly compact in $\Us$.  Thus, there exist a subsequence (not relabeled) and some $u^\star\in \barUsad$ such that $u_n \rightharpoonup u^\star$ weakly in $\Us$.
    Moreover, by \ref{E2} and \ref{G1}, the weak lower semicontinuity of $\cost$ implies
    \[
    \cost(u^\star) \le \liminf_{n\to\infty} \cost(u_n) = \inf_{u\in \Usad}\cost(u),
    \]
    and therefore $u^\star$ is a minimizer of \eqref{eq:Reduced}.
\end{proof}

\begin{remark}
\cref{thm:existence} is stated in a form that covers both finite- and infinite-dimensional settings:
\begin{itemize}
    \item  Finite-dimensional case (\(\Us=\R^n\)):
Weak and strong convergence coincide, and weak sequential closedness in \ref{E4} is equivalent to ordinary closedness.
Moreover, weak lower semicontinuity in \ref{E2} reduces to the usual lower semicontinuity. Hence, the theorem becomes the standard Weierstrass-type existence of solutions.

\item Infinite-dimensional Hilbert-space case:
The main issue is compactness. Bounded subsets of \(\Us\) are generally not relatively compact in the strong topology, but they are weakly sequentially relatively compact by reflexivity. Thus, the boundedness condition \ref{E3} provides the required compactness in the weak topology.  Moreover, if \(\Usad\) is closed and convex, then it is weakly sequentially closed; hence, \ref{E4} is satisfied. In PDE-constrained optimization, the remaining weak continuity properties are often verified using the weak-to-strong continuity of \(\Sctl\), and consequently of \(H\); see \cref{lem:Applicability}.
\qedhere
\end{itemize}
\end{remark}

\subsection{First-order optimality conditions}\label{sec:kkt}

We derive KKT-type first-order optimality conditions for the constrained regularized problem \eqref{eq:Reduced}, where \(\Us\) is a Hilbert space and \(Z\) is a Banach lattice as in \cref{sec:generic}.

Constraint qualifications (CQs) are local regularity assumptions at a feasible point; they ensure, in particular, the existence of Lagrange multipliers for the inequality system in~\eqref{eq:Reduced}.
Let \(Z_- \coloneqq -Z_+\) and denote by \(Z_+^\ast\subset Z^\ast\) the dual cone,
\[
Z_+^\ast \coloneqq \{\,\mu\in Z^\ast \;:\; \innerprod{\mu}{z}_{Z^\ast,Z}\ge 0\ \ \forall z\in Z_+\,\}.
\]
For the two main scenarios covered by \eqref{eq:Reduced}, the dual positive cone is given by $Z_+^\ast=\R_+^q$ in the finite-dimensional case $Z=\R^q$, endowed with the componentwise order.
In the infinite-dimensional function-space case, $Z=C(\overline{K})$ and $Z^\ast=\mathcal{M}(\overline{K})$, where $\mathcal{M}(\overline{K})$ denotes the space of finite signed regular Borel measures on the compact set $\overline{K}$.
Consequently, $Z_+^\ast=\mathcal{M}_+(\overline{K})$, the cone of finite nonnegative regular Borel measures on $\overline{K}$, and $C(\overline{K})$ is endowed with the pointwise order.

\begin{definition}\label{def:CQ_general}
    Let \(u^\star\in\mathcal D\).
    We consider the following standard CQs at \(u^\star\):
    \begin{enumerate}[label={(CQ\arabic*)},ref=(CQ\arabic*)]
        \item \emphdef{Linearized Slater condition:}\label{CQ1}
        There exists \(v\in\dom\costp\) such that $H(u^\star)+H'(u^\star)(v-u^\star)<_Z0$.
        \item \emphdef{Slater condition:}\label{CQ2}
        \(H\) is convex on \(\dom\costp\) and there exists \(\bar u\in\dom\costp\) such that $H(\bar u)<_Z0$.
    \end{enumerate}
\end{definition}

\begin{remark}\label{rem:CQs}
    When $\dom\costp=U$, \ref{CQ1} is the usual linearized Slater condition with an arbitrary direction $d=v-u^\star\in U$.
    If $\costp=\indicator_C$, then $\dom\costp=C$, so the comparison point must belong to $C$.
    Moreover, \ref{CQ2} implies \ref{CQ1} at every feasible $u^\star$; see also \cite[Props.~2.104 and~2.106]{BonnShap2000perturbation}.
\end{remark}

With existence in hand, we now derive the necessary KKT-type conditions for a local solution of \eqref{eq:Reduced}, showing that \ref{CQ1} guarantees the existence of a Lagrange multiplier satisfying dual feasibility and complementarity.

\begin{theorem}[First-order optimality conditions]\label{thm:KKT}
    Let \(u^\star\in \Usad\) be a local solution of~\eqref{eq:Reduced}.
    Assume that \ref{CQ1} holds at $u^\star$.
    Then there exists a multiplier \(\mu^\star\in Z_+^\ast\) such that
    \begin{equation}\label{eq:KKT_inclusion}
        0 \in \nabla \costr(u^\star) + \partial \costp(u^\star) + H'(u^\star)^\ast \mu^\star .
    \end{equation}
    In addition, the primal feasibility and complementarity conditions
    \begin{equation}\label{eq:KKT_comp}
        H(u^\star)\le_Z 0,\qquad
        \mu^\star\in Z_+^\ast,\qquad
        \innerprod{ \mu^\star}{H(u^\star)}_{Z^\ast,Z}=0
    \end{equation}
    are satisfied.
\end{theorem}
\begin{proof}
    We first show that \(u^\star\) solves the convex linearized problem
    \begin{equation}
    \label{eq:convex_lin_kkt}
    \min_{w\in U}\ 
    \langle\nabla\costr(u^\star),w-u^\star\rangle_U+\costp(w)
    \quad\text{subject to}\quad
    H(u^\star)+H'(u^\star)(w-u^\star)\le_Z0.
    \end{equation}
    For an arbitrary \(w\in\dom\costp\) satisfying the linearized constraint in \eqref{eq:convex_lin_kkt}, we write \(w_\theta=(1-\theta)w+\theta v\), where \(v\) is from \ref{CQ1} and \(\theta\in(0,1]\).
    By convexity of the linearized constraint, $w_\theta$ is strictly feasible for \eqref{eq:convex_lin_kkt}.
    Differentiability of \(H\) and convexity of \(\dom\costp\) imply that \(u^\star+t(w_\theta-u^\star)\) is feasible for the original problem for all sufficiently small \(t>0\).
    Using local optimality and convexity of \(\costp\), dividing by \(t\), and then sending \(t\downarrow0\), we obtain
    \[
    0\le
    \langle\nabla\costr(u^\star),w_\theta-u^\star\rangle_U
    +\costp(w_\theta)-\costp(u^\star).
    \]
    Since \(\costp(w_\theta)\le(1-\theta)\costp(w)+\theta\costp(v)\), sending \(\theta\downarrow0\) proves the claim. Since \eqref{eq:convex_lin_kkt} has the strict feasible point \(v\in\dom\costp\), using  convex separation under Slater's condition, e.g., \cite[Prop.~2.106 and Thm~3.4]{BonnShap2000perturbation}, we can infer that there exists \(\mu^\star\in Z_+^*\) with \(\langle\mu^\star,H(u^\star)\rangle_{Z^*,Z}=0\) and
    \begin{equation}
    \label{eq:kkt_subgra_lin}
            \costp(w)-\costp(u^\star) + \langle\nabla\costr(u^\star)+H'(u^\star)^*\mu^\star,w-u^\star\rangle_U \ge0.
     \end{equation}
    Since $w\in\dom\costp$ was arbitrary, \eqref{eq:kkt_subgra_lin} is precisely the subgradient inequality corresponding to \eqref{eq:KKT_inclusion}.
\end{proof}

If \(Z=\R^q\), then \(\mu^\star\in \R_+^q\) and \eqref{eq:KKT_comp} form the usual componentwise complementarity system.
If \(Z=C(\overline K)\), then \(\mu^\star\) can be identified with a nonnegative measure in \(\mathcal{M}_+(\overline{K})\), and \(\langle \mu^\star,\,H(u^\star)\rangle_{Z^\ast,Z}=0\) encodes the complementarity between the measure multiplier and the pointwise constraint residual.

\section{Barrier regularization and inner subproblem analysis}\label{sec:barrier_inner}

We treat the inequality constraint \(H(u)\le_Z 0\) by a barrier approach.
In the two settings of interest, \(Z=\R^q\) (componentwise order) and \(Z=C(\overline K)\) (pointwise order), one has
\(\interior (Z_-)\neq\emptyset\) with respect to the norm topology of \(Z\) (Euclidean norm in \(\R^q\), sup-norm in \(C(\overline K)\)).

\paragraph*{Standing notation for the order and integration}\label{rem:standing-notation}

Throughout the paper, we use the unified notation for the two model cases \(Z=\R^q\) and \(Z=C(\overline K)\).
The partial order \(\le_Z\) is induced by \(Z_+\), i.e.,
$z_1\le_Z z_2 \iff z_2-z_1\in Z_+$.
We write
\[
    \Xi=\{1,\dots,q\}\quad\text{if }Z=\R^q,
    \qquad 
    \Xi=\overline K\quad\text{if }Z=C(\overline K),
\]
and identify \(z\in Z\) with the function \(\xi\mapsto z(\xi)\) on \(\Xi\).
In particular, we can write
\[
    z\le_Z 0 \iff z(\xi)\le 0 \ \ \forall \xi\in\Xi,
    \qquad  \text{ and }  \qquad
    z<_Z 0 \iff z(\xi)<0 \ \ \forall \xi\in\Xi.
\]
We denote by \(\mathbf{1}\in Z\) the order unit, i.e.,
\[
    \mathbf{1}=(1,\dots,1)\in\R^q,
    \qquad \text{ and }  \qquad
    \mathbf{1}(x)\equiv 1 \ \text{in } C(\overline K).
\]
We use the convention
\[
    \int_{\Xi} g(\xi)\mathrm{d}\meas(\xi)
    =
    \begin{cases}
    \displaystyle \sum_{i=1}^q g(i), & \text{if}~ \Xi=\{1,\dots,q\}\ \ (\meas=\text{counting measure}),\\[1ex]
    \displaystyle \int_{\overline K} g(x)\,\mathrm{d}\meas(x), & \text{if}~ \Xi=\overline K,
    \end{cases}
\]
where, in the case $\Xi=\overline{K}$, $\meas$ is a finite positive Borel measure on $\overline{K}$ with full support; if $K\subset\mathbb{R}^d$ is a bounded domain, one may take the Lebesgue measure \(\meas=\mathcal L^d|_{\overline{K}}\).
Accordingly, we have 
\[
    \meas(\Xi)=q \quad \text{if}~ \Xi=\{1,\dots,q\},
    \qquad \text{ and }  \qquad
    \meas(\Xi)=\meas(\overline K)\in(0,\infty) \quad  \text{if}~ \Xi=\overline K.
\]

\subsection{Barrier functionals}\label{sec:barrier-banach-lattices}

\begin{definition}\label{def:barrier-functional}
    A mapping \(\barrierfun :\interior (Z_-)\to\R\) is called a \emphdef{barrier functional} if the following hold:
    \begin{enumerate}[label=(B\arabic*),ref=(B\arabic*)]
    \item \label{B1} \(\barrierfun \) is convex and twice continuously Fr\'echet differentiable on \(\interior(Z_-)\).
    \item \label{B2} \(\nabla \barrierfun (z)\in Z_+^\ast\) for every \(z\in \interior(Z_-)\).
    \end{enumerate}
\end{definition}

In contrast to \cite[Assump.~2]{demarchi2024interior}, $\barrierfun$ need not be nonnegative, thus covering the important logarithmic barrier.
Moreover, unlike classical interior-point complexity theory for smooth conic problems, which exploits self-concordance of the barrier \cite{Nesterov1994Interior}, the barriers $\barrierfun$ considered here are not necessarily self-concordant in the Banach-lattice setting.
 
We now consider barrier functionals induced by scalar kernels.
Let \(\phi:(-\infty,0)\to\R\) be a scalar barrier kernel and define
\begin{equation}\label{eq:Bphi_def}
    \barrierfun_\phi(z)\coloneqq \int_{\Xi} \phi\bigl(z(\xi)\bigr)\,\mathrm{d}\meas(\xi),
    \qquad
    z\in \interior (Z_-).
\end{equation}
For \(\phi\in C^2((-\infty,0))\), the first and second Fr\'echet derivatives of \(\barrierfun _\phi\) are given in \cref{tab:Bphi-derivatives}. In particular, if \(\phi''\ge 0\) on \((-\infty,0)\), then \(\mathrm{D}^2\barrierfun _\phi(z)[h,h]\ge 0\) for all \(h\in Z\), and hence \(\barrierfun _\phi\) is convex.
If, in addition, \(\phi'(t)>0\) for all \(t<0\), then \(\nabla \barrierfun _\phi(z)\in Z_+^\ast\), since
$
    \langle \nabla \barrierfun _\phi(z),h\rangle_{Z^\ast,Z}
    =
    \mathrm{D}\barrierfun _\phi(z)[h]
    \ge 0$
for all $h\in Z_+$.

As concrete examples of scalar kernels, we use the logarithmic and power functions listed in \cref{tab:barrier-derivatives}, which define valid barrier functionals in the sense of \cref{def:barrier-functional}.
 
\begin{table}[tbh]
    \centering%
    \caption{First and second Fr\'echet derivatives of \(\barrierfun _\phi\) in the finite- and infinite-dimensional settings.}%
    \label{tab:Bphi-derivatives}%
    \begin{tabular}{@{}lcc@{}}
        \toprule
        & \(Z=\R^q\) & \(Z=C(\overline K)\) \\
        \midrule
        \(\mathrm{D}\barrierfun _\phi(z)[h]\)
        &
        \(\displaystyle \sum_{i=1}^q \phi'(z_i)\,h_i\)
        &
        \(\displaystyle \int_{\overline K}\phi'(z(x))\,h(x)\,\mathrm{d}\meas(x)\)
        \\
        \(\mathrm{D}^2\barrierfun _\phi(z)[h,k]\)
        &
        \(\displaystyle \sum_{i=1}^q \phi''(z_i)\,h_i\,k_i\)
        &
        \(\displaystyle \int_{\overline K}\phi''(z(x))\,h(x)\,k(x)\,\mathrm{d}\meas(x)\)
        \\
        \bottomrule
    \end{tabular}
\end{table}

\begin{table}[tbh]
    \centering%
    \caption{Barrier kernels and their derivatives on \((-\infty,0)\).}%
    \label{tab:barrier-derivatives}%
    \begin{tabular}{lccc}%
        \toprule
        kernel & \(\phi(t)\) & \(\phi'(t)\) & \(\phi''(t)\) \\
        \midrule
        \(\phi_{\log}\)
        & \(-\log(-t)\)
        & \(-\dfrac{1}{t}\)
        & \(\dfrac{1}{t^{2}}\) \\[1ex]
        \(\phi_{\rm pow}\)
        & \((-t)^{-p},\ p>0\)
        & \(p(-t)^{-(p+1)}\)
        & \(p(p+1)(-t)^{-(p+2)}\) \\
        \bottomrule
    \end{tabular}
\end{table}

\begin{lemma}\label{lem:log-power-are-barriers}
    For the kernels $\phi$ in \cref{tab:barrier-derivatives}, the barrier functional \(\barrierfun_\phi\) defined in \eqref{eq:Bphi_def} satisfies \ref{B1}--\ref{B2}.
\end{lemma}
\begin{proof}
    Both kernels are $C^2$ on $(-\infty,0)$.
    Since each $z\in\interior Z_-$ is uniformly separated from zero, differentiation under the integral shows that $\barrierfun_\phi$ is twice continuously Fr\'echet differentiable.
    The inequalities $\phi''>0$ and $\phi'>0$ imply convexity and positivity of its derivative, respectively.
    Thus, \ref{B1} and \ref{B2} hold.
\end{proof}

The abstract convergence results use only \ref{B1}--\ref{B2}.
The logarithmic and power kernels enter separately only through the quantitative curvature and complementarity estimates below.

\begin{remark}\label{rem:barrier-boundary}
    For $Z=\R^q$, both the logarithmic and power barriers satisfy
    \begin{equation}
    \label{eq:Bar_Blowup}
        z_j\in \operatorname{int} Z_-,
        \quad
        z_j\to z\in\partial Z_-
        \quad\implies\quad
        \barrierfun_\phi(z_j)\to+\infty.
        \end{equation}
    For $Z=C(\overline K)$, this implication generally fails, since the limit may touch the boundary only on a set of measure zero while the barrier values remain bounded.
    Nevertheless, if $z_j\in\interior Z_-$ converges uniformly to $z$ and
    \[
        \meas\bigl(\{x\in\overline K:z(x)=0\}\bigr)>0,
    \]
    then $\barrierfun_\phi(z_j)\to+\infty$ for both kernels.
    This follows from Fatou's lemma, after shifting the logarithmic integrands by a common constant if necessary.
    Thus, bounded barrier values imply $z<0$ $\meas$-almost everywhere, but not necessarily $z\in\interior Z_-$.
    For power barriers, pointwise contact with the constraint boundary can be excluded under additional regularity.
    Following \cite[Lem.~7.1]{Schiela2009barrier}, if the relevant constraint functions are uniformly bounded in \(C^{0,\alpha}(\overline K)\) and \(\alpha p>d\), then
    $\barrierfun_{\phi_p}(z_j)\to+\infty$
    whenever \(z_j\) approaches the boundary.
    Hence, under this regularity condition, bounded power-barrier values exclude boundary contact even on sets of measure zero.
    In contrast, no analogous conclusion holds in general for the logarithmic barrier.
    
     The required regularity is natural in many PDE-constrained optimization problems, including \eqref{eq:Primal}. Indeed, for $S=I$, if \(K\subset\R^d\), \(d\le 3\), and the constraint function belongs to \(H^2(K)\) (see \cref{prop:unified_compact_stability_affine_Bf_Neumann}\ref{stAffE:I}), then the Sobolev--Morrey embedding \cite[Thm.~9.12]{Brez2011functional} yields $H^2(K)\hookrightarrow C^{0,\alpha}(\overline K)$ for every $\alpha<2-\frac d2$.
    Consequently, one can choose \(\alpha\) such that \(\alpha p>d\) whenever \(p>2\) for \(d=2\) and \(p>6\) for \(d=3\).
\end{remark}

\subsection{Inner loop: Barrier subproblem}

In this section, we investigate the barrier subproblems \eqref{eq:Reduced_barrier_problem}, which approximate \eqref{eq:Reduced} by replacing the inequality constraint with a barrier term in the objective.
The approximation is controlled by the \emph{barrier parameter} $\nu > 0$.
Rewriting subproblem \eqref{eq:Reduced_barrier_problem} as
\begin{equation}\label{eq:Reduced_barrier_problem_2}
    \minimize_{u \in \Us}\;
    \costb(u)\coloneqq \costs_\nu(u) + \costp(u)
    \qquad\text{where}\qquad
    \costs_\nu
    \coloneqq 
    u \mapsto \costr(u) + \nu \barrierfun (H(u))
\end{equation}
highlights the structured form ``smooth plus prox-friendly'' and appears amenable to methods of proximal-gradient type, as noticed in \cite{demarchi2024interior}.
Working on the strictly feasible domain $\Usad^\circ$, see \eqref{eq:strict-feasible-domain}, we assign $\costb=+\infty$ outside $\Usad^\circ$ to reject infeasible trial points.
For each subproblem \eqref{eq:Reduced_barrier_problem_2} the interior-point algorithm requires only approximate stationary points (instead of exact minimizers), which the inner proximal-gradient algorithm can obtain in finitely many steps.

\subsubsection{Boundedness of the admissible controls}
Boundedness of the inner iterates can be ensured by bounded control constraints or by a quadratic control cost \(\|u\|_{\Us}^{2}\), provided that the barrier objective values remain bounded above. In the latter case, we assume
\begin{equation}\label{eq:Cpde}\tag{C}
    \barrierfun (H(u)) \ge -\kappa\bigl(1+\|u\|_{\Us}\bigr),
    \qquad 
    \forall u \text{ with } H(u)\in \interior (Z_-),\ \|u\|_{\Us}\ge R_0,
\end{equation}
where $\kappa,R_0\ge0$. This lower bound preserves coercivity of an objective with positive quadratic lower growth. For the power barrier, \eqref{eq:Cpde} holds with $\kappa=0$.
For the logarithmic barrier, it follows from
\[
    -H(u)\le_Z
    M(1+\|u\|_{\Us})^d e^{a\|u\|_{\Us}}\mathbf1,
\]
where $M\ge1$ and $a,d\ge0$, since
\[
    -\log(-H(u))
    \ge_Z
    -\bigl[\log M+d\log(1+\|u\|_{\Us})
             +a\|u\|_{\Us}\bigr]\mathbf1.
\]
For an integral barrier, we assume that the underlying measure is finite.

In particular, for the pointwise state constraint $H(u)=\Sctl(u)-\bar y$ with bounded $\bar y$, the estimate
\begin{equation}\label{eq:ctr_esti}
    \|\Sctl(u)\|_Z\le M(1+\|u\|_{\Us})
\end{equation}
implies the above slack bound with $d=1$ and $a=0$, after increasing $M$ if necessary.

\subsubsection{Lipschitz continuity of the gradient for the smooth part}

In this section, we study the continuity of the gradient of the smooth part \(\costs_\nu\) in \eqref{eq:Reduced_barrier_problem_2} and derive bounds for the corresponding Lipschitz constant.

In many PDE-constrained optimization problems, the control-to-state map \(\Sctl\colon U\to Y\) is twice continuously Fr\'echet differentiable, and its first and second derivatives are uniformly bounded on bounded subsets of its domain.
This motivates the following standing assumption.

\begin{assumption}\label{ass:Lip}
    The following condition holds:
    \begin{enumerate}[label=(G\arabic*), start=3]
    \item \label{G3}%
    For every bounded set \(U_b \subseteq \mathcal{D}\), the mappings \(\costrD\) and \(H'\) are Lipschitz continuous on \(U_b\), with constants \(L_{\costr}(U_b)\) and \(L_{H'}(U_b)\), respectively.
    Moreover, $\sup_{u\in U_b}\|H'(u)\|_{\mathcal{L}(U,Z)} \le M_{H'}(U_b)$ for some \(M_{H'}(U_b)>0\).
    \end{enumerate}
\end{assumption}

Let \(B\subset \mathcal{D}\) be a bounded neighborhood such that \(H(B)\) lies in the strictly feasible set
\[
    \mathcal{D}_{a,b}
    \coloneqq
    \{\,z\in Z:\ -b\mathbf{1}\le_Z z \le_Z -a\mathbf{1}<_Z 0\,\},
    \qquad a,b>0.
\]
For a barrier functional \(\barrierfun_\phi\) with \(\phi\in C^2((-\infty,0))\), the gradient of $\costs_\nu$ in \eqref{eq:Reduced_barrier_problem_2} is given by
\[
    \nabla \costs_\nu(u)
    =
    \nabla \costr(u) + \nu\,H'(u)^{*}\nabla \barrierfun _\phi(H(u))
    =
    \nabla\costr(u) + \nu\,H'(u)^{*}\phi'\big(H(u)\big),
\]
where the last expression is understood via the standing integration convention (counting measure in \(Z=\R^q\), and \(\meas\) on \(\overline K\) in \(Z=C(\overline K)\)).
Since 
\[
    M_{\phi'}(a,b) \coloneqq \sup_{t\in[-b,-a]}|\phi'(t)|,
    \qquad
    L_\phi(a,b) \coloneqq \sup_{t\in[-b,-a]}|\phi''(t)|
\]
are finite, it follows from \ref{G3} that \(\nabla\costs_\nu\) is Lipschitz continuous on \(B\), with constant
\begin{equation}\label{eq:Lipschitz-smooth-nu}
    \Lsmooth(\nu; a,b)
    \coloneqq
    L_{\costr}(B)
    + \nu c_{\Xi} \Big[
    L_{H'}(B)\,M_{\phi'}(a,b) + M_{H'}(B)\,L_\phi(a,b)\,L_H(B)
    \Big],
\end{equation}
where \(L_H(B)\) is a Lipschitz constant of \(H\) on \(B\).
Here $c_\Xi=\meas(\Xi)$ for $C(\overline K)$ with the sup norm, while $c_\Xi=\max\{1,\sqrt q\}$ is a sufficient common factor for Euclidean $\R^q$.
These constants account for the norm of the integral derivative and are independent of $\nu$.

If $H$ is uniformly bounded on the regions under consideration, the upper slack bound $b$ can be chosen independently of $\nu$.
Provided that the remaining constants are also uniform, the dependence in \eqref{eq:Lipschitz-smooth-nu} is determined by the explicit factor $\nu$ and the lower slack bound $a=a(\nu)$.

In the next lemma, we derive corresponding upper bounds on the Lipschitz constants and quantify their blow-up rates.

\begin{table}[tbh]%
    \centering%
    \caption{Canonical barrier scalings as \(\nu\downarrow 0\).}%
    \label{tab:scaling}%
    \begin{tabular}{lcccc}%
        \toprule
        kernel 
        & \(a(\nu)\) 
        & \(M_{\phi'}(a,b)\) 
        & \(L_\phi(a,b)\) 
        & \(\nu\,L_\phi(a,b)\) \\
        \midrule
        \(\phi_{\log}(t)=-\log(-t)\)
        & \(a(\nu)\simeq \nu\)
        & \(\nu^{-1}\)
        & \(\nu^{-2}\)
        & \(\nu^{-1}\)
        \\[1mm]
        \(\phi_{\rm pow}(t)=(-t)^{-p}\), \(p>0\)
        & \(a(\nu)\simeq \nu^{1/(p+1)}\)
        & \(\nu^{-1}\)
        & \(\nu^{-(p+2)/(p+1)}\)
        & \(\nu^{-1/(p+1)}\)
        \\
        \bottomrule
    \end{tabular}
\end{table}

Since the inner proximal-gradient complexity depends on the Lipschitz constant of $\nabla \costs_\nu$, we now quantify how this constant grows as the barrier parameter decreases, establishing the $\nu$-dependent bounds that feed directly into the total complexity analysis.
In \cref{sec:overall_complexity}, we distinguish slack bounds at outer endpoints from the inner-trajectory bounds required to use these estimates in a complexity argument.

\begin{lemma}\label{lem:central-path-Lip-onesided}
    Assume the hypotheses leading to \eqref{eq:Lipschitz-smooth-nu} hold on a neighborhood \(B\subset U_b\), and that \(H(B)\subset\mathcal D_{a(\nu),b}\) for some \(a(\nu)>0\) and some fixed \(b>0\).
    Then, as \(\nu\downarrow0\),
    \[
    \Lsmooth(\nu;a(\nu),b)
    \le
    L_{\costr}(B)+
    \begin{cases}
    \bigO(\nu^{-1}), 
    & \text{if }\phi(t)=-\log(-t)\ \text{and } a(\nu)\ge c_1\nu,\\[1.5mm]
    \bigO\bigl(\nu^{-1/(p+1)}\bigr), 
    & \text{if }\phi(t)=(-t)^{-p},\ p>0,\ \text{and } a(\nu) \ge c_1\nu^{1/(p+1)},
    \end{cases}
    \]
    for some \(c_1>0\) and all sufficiently small \(\nu\).
\end{lemma}
\begin{proof}
    Insert the barrier-specific expressions for \(M_{\phi'}(a(\nu),b)\) and \(L_\phi(a(\nu),b)\) from \cref{tab:scaling} into \eqref{eq:Lipschitz-smooth-nu}, and use the lower bound for \(a(\nu)\).
\end{proof}

Each outer iteration of the interior-point method approximately solves a subproblem with smooth part \(\costs_\nu\) by a proximal-gradient method, whose iteration complexity typically depends on the Lipschitz constant of \(\nabla \costs_\nu\).
\cref{lem:central-path-Lip-onesided} therefore quantifies possible worst-case deterioration of the smoothness bound, but it does not assert divergence of the actual Lipschitz constant.
However, barrier curvature may grow even when the associated multipliers remain bounded.

\subsection{Inner algorithm and optimality conditions}

The stationarity condition for subproblem~\eqref{eq:Reduced_barrier_problem_2} can be expressed as
$0 \in \partial \costb(u_\nu)$.
For computational purposes, we consider an \emph{approximate stationarity condition}, given by  
\begin{equation}\label{eq:Termination_subproblem}
    \dist\bigl(0, \partial \costb(u_\nu) \bigr) \le \varepsilon_\nu
\end{equation}
for a prescribed tolerance $\varepsilon_\nu > 0$.
To find such an $\varepsilon_\nu$-stationary point $u_\nu \in \dom\costb$, proximal-gradient methods are the natural choice when the gradient of $\costs_\nu$ and the proximal operator of $\costp$ are readily available.
As a baseline scheme that allows us to derive complexity bounds, we employ the nonmonotone forward–backward splitting method studied in \cite{azmi2025nonmonotone}.
The scheme is outlined in \cref{alg:NMPG}, which incorporates a backtracking mechanism from strictly feasible initial points.
This strategy requires that the proximal-gradient update
\begin{equation}\label{eq:prox_grad_eval}
    u_{j+1} \coloneqq
            \prox_{\ssizek^{-1}\costp}
            \left(
                u_j - \ssizek^{-1} \nabla \costs_\nu(u_j)
            \right)
\end{equation}
satisfies the nonmonotone sufficient decrease condition
\begin{equation}\label{eq:ls_ineq}
    \costb(u_{j+1}) + \delta \ssizek \|u_{j+1} - u_j\|_{\Us}^2
    \leq
    \max_{0 \le \ell \le m(j)} \costb(u_{j-\ell})
\end{equation}
for some fixed $\delta \in (0,1)$.
Here, the memory function $m\colon\N_0 \to \N_0$ is defined by
\begin{equation}\label{eq:ls_memory}
    m(0) = 0, 
    \qquad
    m(j) = \min\{\, m(j-1) + 1,\, \mmax \,\}, 
    \quad j \in \N,
\end{equation}
for a given upper bound $\mmax \in \N_0$.

\begin{algorithm2e}
    \caption{Nonmonotone proximal-gradient method for \eqref{eq:Reduced_barrier_problem_2}}
    \label{alg:NMPG}
    \DontPrintSemicolon
    \KwIn{strictly feasible $u_0\in \Usad^\circ$, barrier parameter $\nu > 0$, tolerance $\varepsilon_\nu > 0$}
    \KwData{select $\delta \in (0,1)$, $\mmax \in \N_0$, $\eta > 1$, $\ssizeub \geq \ssizelb > 0$}
    \KwOut{$\varepsilon_\nu$-stationary point $u_\nu \in \Usad^\circ$}
    \For{$j=0,1,2\ldots$}{
        choose $\ssizeinit \in [\ssizelb, \ssizeub]$ and set $\ssizek\gets \ssizeinit$\;
        compute the trial point $u_j^+ \gets \prox_{\ssizek^{-1}\costp}
            \left(
                u_j - \ssizek^{-1} \nabla \costs_\nu(u_j)
            \right)$\label{step:NMPG:u_trial}\;
        \If{$u_j^+ \notin\Usad^\circ$ or \eqref{eq:ls_ineq} fails}{set $\ssizek \gets \eta\ssizek$ and go to Step~\ref{step:NMPG:u_trial}}
        set $u_{j+1} \gets u_j^+$\label{step:NMPG:u_update}\;
        \If{termination condition $\eqref{eq:Termination_subproblem_practical}$ holds}{\Return{$u_\nu\gets u_{j+1}$}}
    }
\end{algorithm2e}

Starting from an initial, bounded estimate $\ssizeinit \in [\ssizelb,\ssizeub]$, the regularization parameter $\ssizek$ is increased by a constant factor until the proximal-gradient candidate $u_{j+1}$ computed as \eqref{eq:prox_grad_eval} satisfies \eqref{eq:ls_ineq}.

\begin{lemma}[Finite termination of the inner backtracking]
    Let \(u_j\in \Usad^\circ\).
    Then there exists \(\bar\ssizek<\infty\) such that every \(\ssizek\ge\bar\ssizek\) produces a proximal-gradient point \(u_{j+1}\in \Usad^\circ\) according to \eqref{eq:prox_grad_eval} that satisfies the decrease condition \eqref{eq:ls_ineq}.
\end{lemma}
\begin{proof}
    We only sketch the proof, which follows \cite[Lem.~13]{demarchi2024interior}.
    Set $t=\ssizek^{-1}$ and $u(t) \coloneqq \prox_{t\costp}\left( u_j - t \nabla \costs_\nu(u_j)\right)$.
    By the standard convergence result for proximal mappings as the parameter tends to zero, $u(t)\to u_j$ strongly in $U$ as $t\downarrow 0$; see \cite[Prop.~12.33]{BausComb2017convex}.
    Since $u(t)\in\dom\costp\subset\mathcal{D}$ and $H:\mathcal{D}\to Z$ is continuous, while $H(u_j)<_Z 0$, we have $H(u(t)) <_Z 0$ for all sufficiently small $t$; hence $u(t)\in \Usad^\circ$.
    On a neighborhood of $u_j$ contained in the local smoothness region, the standard proximal-gradient descent estimate gives \eqref{eq:ls_ineq} for all sufficiently small $t$, equivalently for all sufficiently large $\alpha_j$.
 Thus, there exists $\bar\ssizek<\infty$ such that every $\ssizek\ge\bar\ssizek$ is accepted.
\end{proof}

Since criterion \eqref{eq:Termination_subproblem} involves computing the subdifferential $\partial \costb(u_j)$, it is not a practical termination condition in general.
However, the proximal-gradient update rule \eqref{eq:prox_grad_eval} yields the (necessary optimality) inclusion
\begin{equation}\label{eq:prox_grad_inclusion}
    0 \in \nabla \costs_{\nu}(u_j) + \partial\costp(u_{j+1}) + \ssizek (u_{j+1}-u_j)
\end{equation}
for each $j\in\N$, which leads to
the upper bound
\begin{equation*}
    \dist\bigl(0, \partial \costb(u_{j+1}) \bigr) \leq \| \nabla \costs_{\nu}(u_{j+1}) - \nabla \costs_{\nu}(u_j) - \ssizek (u_{j+1}-u_j) \|_{\Us}
\end{equation*}
on the stationarity measure.
Thus, while executing \cref{alg:NMPG}, one can adopt the practical termination condition
\begin{equation}\label{eq:Termination_subproblem_practical}
    \| \nabla \costs_{\nu}(u_{j+1}) - \nabla \costs_{\nu}(u_j) - \ssizek (u_{j+1}-u_j) \|_{\Us} \leq \varepsilon_\nu
\end{equation}
to certify $\varepsilon_\nu$-stationarity according to \eqref{eq:Termination_subproblem} and terminate with the iterate $u_{j+1} \in \Usad^\circ$.

\begin{remark}
We note that, if \eqref{eq:Bar_Blowup} holds, \cref{step:NMPG:u_trial} automatically yields a strictly feasible point $u_j^+\in\Usad^\circ$, and we proceed directly to \cref{step:NMPG:u_update} without an additional feasibility check. \cref{alg:NMPG} enforces strict feasibility of every accepted inner iterate. In the finite-dimensional constraint setting $Z=\R^q$, the logarithmic and power barriers satisfy \eqref{eq:Bar_Blowup}; hence, a finite barrier value guarantees strict feasibility. However, for integral logarithmic barriers in $C(\overline K)$, \eqref{eq:Bar_Blowup} may fail when the constraint function touches zero only on a set of zero measure. Therefore, a finite barrier value does not guarantee strict feasibility, which is explicitly checked in
\cref{alg:NMPG}. The outer convergence analysis assumes that each inner solve returns a strictly feasible point satisfying the prescribed approximate stationarity and decrease conditions.

For power barriers, a finite barrier value guarantees strict feasibility under the regularity and geometric conditions of \cref{rem:barrier-boundary}. In particular, if the constraint function belongs to $C^{0,\eta}$ and $\eta p>d$, then boundary contact would force the power-barrier integral to be infinite.

For a discretization with finitely many inequality constraints and a finite sum of barrier terms with positive weights, a finite discrete barrier value guarantees strict feasibility of all discrete inequalities.
\end{remark}

\subsection{Inner complexity bound}\label{sec:inner_complexity}

In this section, we estimate the number of proximal-gradient iterations required to obtain an $\varepsilon$-stationary point of a fixed barrier subproblem.
\begin{assumption}\label{ass:inner-hypotheses}
    For a fixed $\nu>0$, we assume that subproblem \eqref{eq:Reduced_barrier_problem_2} and the iterates $\{u_j\}$ generated by \cref{alg:NMPG} satisfy
    \begin{align*}
        \cost_\nu^*
        \coloneqq{}&
        \inf_{\Usad^\circ}\cost_\nu>-\infty,
        \\
        \|\nabla\costs_\nu(u_{j+1}) -\nabla\costs_\nu(u_j)\|_{\Us}
        \le{}&
        \Lsmooth(\nu)\|u_{j+1}-u_j\|_{\Us}
        \quad\text{for every }j,
    \end{align*}
    where $\Lsmooth(\nu)$ is independent of $j$.
\end{assumption}
Note that this uniform Lipschitz constant $\Lsmooth(\nu)$ along the iterates $\{u_j\}$, not the objective $\costs$, implies that the stepsizes are bounded uniformly from above, namely there exists $\overline\alpha_\nu\in\R$ such that $\alpha_j \leq \overline\alpha_\nu$ for all $j$.
We use the notation $\Lsmooth(\nu)=\Lsmooth(\nu;a,b)$ when this bound is supplied by \eqref{eq:Lipschitz-smooth-nu}, with any dependence of $a,b$ on $\nu$ understood.
The dependence of these bounds on $\nu$, including the line-search bound from \cite[Lem.~4]{azmi2025nonmonotone}, is discussed in \cref{sec:overall_complexity}.

Following \cite{azmi2025nonmonotone}, we define the \emphdef{gradient mapping} by
\[
\gmap_\alpha(u)
=\alpha\left[u-\prox_{\alpha^{-1}\costp}
\bigl(u-\alpha^{-1}\nabla\costs_\nu(u)\bigr)\right].
\]
Thus, the iterates satisfy $\gmap_{\alpha_j}(u_j)=\alpha_j(u_j-u_{j+1})$.
We denote the left-hand side of \eqref{eq:Termination_subproblem_practical} by $r_{j+1}^{\rm in}$ and write
\begin{equation}\label{eq:K-inner}
    c_\nu \coloneqq\delta^{-1}\left(\overline\alpha_\nu+2\Lsmooth(\nu) +\frac{\Lsmooth(\nu)^2}{\ssizelb}\right)
    \quad \text{and}\quad
    C_\nu(u_0) \coloneqq \sqrt{c_\nu\bigl(\cost_\nu(u_0)-\cost_\nu^*\bigr)}.
\end{equation}

\begin{proposition}[Inner complexity]\label{prop:inner-complexity}
    Under \cref{ass:inner-hypotheses}, \cref{alg:NMPG} reaches $r_j^{\rm in} \le\varepsilon$ after at most
    \begin{equation}\label{eq:inner_complexity}
        \Niterin(\nu,\varepsilon) \leq (\mmax+1)\left(\left\lfloor\frac{C_\nu(u_0)^2}{\varepsilon^2}\right\rfloor+1\right)
    \end{equation}
    iterations.
    If, in addition, $\cost_\nu$ is $\kappa$-strongly convex, then the subsequence with monotonic decrease (specified in the proof) satisfies
    \begin{equation}\label{eq:inner-linear-residual}
         r_{j_k}^{\rm in} \le C_\nu(u_0)\sigma_\nu^{(k-1)/2},
        \qquad
        \sigma_\nu=(1+2\kappa/c_\nu)^{-1},\quad k\ge1.
    \end{equation}
    In particular, it suffices to take
    \begin{equation}\label{eq:inner_complexity_strong_convexity}
        \Niterin(\nu,\varepsilon)
        \le(\mmax+1)\left( \left\lceil
        \frac{2\log_+(C_\nu(u_0)/\varepsilon)}{|\log\sigma_\nu|}
        \right\rceil + 1 \right),
    \end{equation}
    where $\log_+ t=\max\{0,\log t\}$ and $\log_+0=0$.
\end{proposition}
\begin{proof}
    Following \cite[Lem.~5, (L1)--(L3)]{azmi2025nonmonotone}, we let $\ell(j) \coloneqq \max \{\arg\max_{j-m(j)\le i\le j}\cost_\nu(u_i)\}$ and set $j_k\coloneqq\ell\bigl(k(\mmax+1)\bigr)$ for each $k\in\N$.
    The indices are strictly increasing, $j_0=0$, $j_k\le k(\mmax+1)$, and
    \begin{equation}\label{eq:selected-decrease}
    \cost_\nu(u_{j_k})
    \leq
    \cost_\nu(u_{j_{k-1}})
    -\frac{\delta}{\alpha_{j_k-1}}
     \|\gmap_{\alpha_{j_k-1}}(u_{j_k-1})\|_U^2.
    \end{equation}
    These conclusions use only the memory rule \eqref{eq:ls_memory} and accepted decrease \eqref{eq:ls_ineq}; thus their proof applies to the strictly feasible iterates here, without a global smooth extension of the barrier objective.
    By \cref{ass:inner-hypotheses} and the proximal optimality inclusion, we have
    \[
     r_{j+1}^{\rm in}\le\left(1+\frac{\Lsmooth(\nu)}{\alpha_j}\right)
               \|\gmap_{\alpha_j}(u_j)\|_U
     \le\left(1+\frac{\Lsmooth(\nu)}{\ssizelb}\right)
               \|\gmap_{\alpha_j}(u_j)\|_U.
    \]
    To combine the first inequality with \eqref{eq:selected-decrease}, we observe that
    \[
    \frac{\alpha_j}{\delta}
    \left(1+\frac{\Lsmooth(\nu)}{\alpha_j}\right)^2
    =\frac{1}{\delta}\left(\alpha_j+2\Lsmooth(\nu)
     +\frac{\Lsmooth(\nu)^2}{\alpha_j}\right)
    \le c_\nu.
    \]
    Thus,  we obtain $(r^{\rm in}_{j_k})^2\le c_\nu\bigl(\cost_\nu(u_{j_{k-1}})-\cost_\nu(u_{j_k})\bigr)$ and by summing as in \cite[Eq.~(21)]{azmi2025nonmonotone}, we can write
    \begin{equation*}
    \sum_{k=1}^n (r^{\rm in}_{j_k})^2\le C_\nu(u_0)^2,
    \qquad \text{ and  } \qquad
    \min_{1\le k\le n}(r^{\rm in}_{j_k})^2\le\frac{C_\nu(u_0)^2}{n}.
    \end{equation*}
    Hence the selected residuals converge to zero if the iteration is continued indefinitely.
    Taking $n=\lfloor C_\nu(u_0)^2/\varepsilon^2\rfloor+1$ and using the fact that $j_n\le n(\mmax+1)$, we obtain \eqref{eq:inner_complexity}.
    
    We now assume that  $\cost_\nu$ is $\kappa$-strongly convex.
    By the proximal optimality inclusion underlying \eqref{eq:prox_grad_inclusion}, we have
    \begin{equation*}
    \xi_k\coloneqq\nabla\costs_\nu(u_{j_k})-\nabla\costs_\nu(u_{j_k-1})
     -\alpha_{j_k-1}(u_{j_k}-u_{j_k-1}) \in\partial\cost_\nu(u_{j_k}),
    \qquad \text{with }  \|\xi_k\|_U=r^{\rm in}_{j_k}.
    \end{equation*}
    The strong-convexity subgradient inequality therefore gives, for every $v\in\Usad^\circ$,
    \[
    \cost_\nu(v)\ge\cost_\nu(u_{j_k})
     +\langle\xi_k,v-u_{j_k}\rangle_U
     +\frac{\kappa}{2}\|v-u_{j_k}\|_U^2.
    \]
    Using the Cauchy--Schwarz inequality and writing $t=\|u_{j_k}-v\|_U$, we obtain
    \[
    \begin{aligned}
    \cost_\nu(u_{j_k})-\cost_\nu(v)
    &\le r^{\rm in}_{j_k}t-\frac{\kappa}{2}t^2=\frac{(r^{\rm in}_{j_k})^2}{2\kappa}
     -\frac{\kappa}{2}\left(t-\frac{r^{\rm in}_{j_k}}{\kappa}\right)^2
    \le\frac{(r^{\rm in}_{j_k})^2}{2\kappa}.
    \end{aligned}
    \]
    Since this estimate holds for every $v\in\Usad^\circ$, taking the infimum of $\cost_\nu(v)$ yields
    \[
    \cost_\nu(u_{j_k})-\cost_\nu^*\le\frac{(r^{\rm in}_{j_k})^2}{2\kappa}.
    \]
    In particular, this argument does not require attainment of $\cost_\nu^*$.
    We write $E_k\coloneqq\cost_\nu(u_{j_k})-\cost_\nu^*\ge0$ for $k\ge0$.
    Combining the last estimate with the preceding residual-decrease bound, we obtain
    \[
    2\kappa E_k\le (r^{\rm in}_{j_k})^2\le c_\nu(E_{k-1}-E_k).
    \]
    Rearranging and applying the resulting estimate successively, we can write
    \[
    (c_\nu+2\kappa)E_k\le c_\nu E_{k-1},
    \qquad
    E_k\le\frac{c_\nu}{c_\nu+2\kappa}E_{k-1}
    =\sigma_\nu E_{k-1}\le\sigma_\nu^k E_0.
    \]
    Using again the residual-decrease bound and $E_k\ge0$, we consequently have
    \[
    (r^{\rm in}_{j_k})^2\le c_\nu(E_{k-1}-E_k)
    \le c_\nu E_{k-1}
    \le c_\nu E_0\sigma_\nu^{k-1}
    =C_\nu(u_0)^2\sigma_\nu^{k-1}.
    \]
    Taking square roots proves \eqref{eq:inner-linear-residual}.
    For $\varepsilon>0$, choosing
    \[
    k=1+\left\lceil
    \frac{2\log_+(C_\nu(u_0)/\varepsilon)}{|\log\sigma_\nu|}
    \right\rceil
    \]
    ensures $r^{\rm in}_{j_k}\le\varepsilon$.
    Together with $j_k\le k(\mmax+1)$, this proves \eqref{eq:inner_complexity_strong_convexity}.
\end{proof}

The estimates above concern a fixed barrier subproblem.
In \cref{sec:overall_complexity}, we use the warm starts in \cref{alg:IP} to bound the initial objective gaps and then sum the inner iteration bounds over the barrier parameters.

\section{Interior-point scheme: convergence and total complexity} \label{sec:outer_scheme}

In this section, we present the outer interior-point scheme used to solve \eqref{eq:Primal} and investigate its convergence and complexity.

\subsection{Outer interior-point loop}\label{sec:outer_convergence}

The numerical procedure is inspired by \cite[Alg. 1]{demarchi2024interior} and outlined in \cref{alg:IP}.
The main computational effort lies in the minimization (up to approximate stationarity) of the barrier subproblem \eqref{eq:Reduced_barrier_problem}. In iteration $k\in\N$, we execute this step by invoking \cref{alg:NMPG} with (strictly feasible) initial point $u_{k-1}$, barrier parameter $\nu_k$, and tolerance $\varepsilon_k$.
The outer loop then proceeds by successively reducing the barrier parameter $\nu_k$ and the inner tolerance $\varepsilon_k$ until the prescribed termination conditions are met. 

\begin{algorithm2e}
    \caption{Interior point method for \eqref{eq:Reduced}}%
    \label{alg:IP}%
    \DontPrintSemicolon%
    \KwIn{strictly feasible $u_0\in \Usad^\circ$, tolerance $\varepsilon>0$}
    \KwData{select $\nu_1, \varepsilon_1 > 0$, $\theta_\varepsilon, \theta_\nu \in (0,1)$}
    \KwOut{$\varepsilon$-KKT point $u^\star \in \Us$}
    \For{$k = 1,2,3\ldots$}{
        find a strictly feasible $\varepsilon_k$-stationary point $u_k\in \Usad^\circ$ for \eqref{eq:Reduced_barrier_problem} with $\nu\gets\nu_k$, starting from $u_{k-1}$
        \label{step:IP:primal}\;
        set $\mu_k \gets \nu_k \nabla \barrierfun ( H(u_k))$\label{step:IP:dual}\;
        \If{termination condition $\eqref{eq:terminationEpsKKT}$ holds}{\Return{$u^\star\gets u_k$}}
        set $\varepsilon_{k+1} \gets \max\{\varepsilon,\theta_\varepsilon \varepsilon_k\}$ and $\nu_{k+1} \gets \theta_\nu \nu_k$\label{step:IP:params}\;
	}
\end{algorithm2e}

The key step in our analysis is to relate approximate solutions of the barrier subproblem \eqref{eq:Reduced_barrier_problem} to approximate KKT points of the original problem \eqref{eq:Reduced}.
The following lemma collects the four fundamental properties maintained at every outer iteration of \cref{alg:IP}: strict feasibility of the primal iterates, dual feasibility of the approximate multipliers, a quantitative stationarity residual for the original problem, and a descent property when the barrier kernel is nonnegative, cf. \cite[Lemma 17]{demarchi2024interior}.

\begin{remark}
    Although \cref{alg:IP} uses warm starts at \cref{step:IP:primal}, the outer stationarity and KKT convergence results remain valid with other strictly feasible initial points for the inner solves, under the stated assumptions.
    Warm starts are used for the descent estimate in \cref{lem:outer_iterates_behaviour}\ref{lem:outer_iterates_behaviour:iv} and the uniform initial objective-gap bounds in \cref{sec:overall_complexity}.
    For other initializations, these bounds must be verified separately.
\end{remark}

\begin{lemma}[Basic properties of the outer iterates]\label{lem:outer_iterates_behaviour}
    Let \(\{(u_k,\mu_k)\}_{k\ge 1}\) be a sequence generated by \cref{alg:IP}.
    Then, for every \(k\ge 1\), the following hold:
    \begin{enumerate}[label=(\roman*)]
        \item \label{lem:outer_iterates_behaviour:i}%
        \(u_k\in \Us\) is strictly feasible for \eqref{eq:Reduced}, i.e., $H(u_k)<_Z 0$;
        \item \label{lem:outer_iterates_behaviour:ii}%
        \(\mu_k\in Z_+^\ast\);
        \item \label{lem:outer_iterates_behaviour:iii}%
        $\dist\left( -\nabla \costr(u_k) - H'(u_k)^\ast \mu_k, \partial \costp(u_k) \right)\le \varepsilon_k$;
        \item \label{lem:outer_iterates_behaviour:iv}%
        if the scalar kernel \(\phi\) defining \(\barrierfun_\phi\) is nonnegative on \((-\infty,0)\), then
        \[
        (\costr+\costp)(u_k)
        \;\le\;
        (\costs_{\nu_k}+\costp)(u_k)
        \;\le\;
        (\costs_{\nu_k}+\costp)(u_{k-1})
        \;\le\;
        (\costs_{\nu_{k-1}}+\costp)(u_{k-1}).
        \]
    \end{enumerate}
\end{lemma}
\begin{proof}
    \ref{lem:outer_iterates_behaviour:i}
    By \cref{alg:NMPG}, each inner solve returns $u_k\in \Usad^\circ$.
    Hence $H(u_k)\in\interior(Z_-)$, i.e., $H(u_k)<_Z 0$.
    
    \ref{lem:outer_iterates_behaviour:ii}
    By definition of the approximate multiplier, we have
    \begin{equation}\label{eq:lag_Mul}
        \mu_k \coloneqq \nu_k \nabla \barrierfun_\phi\bigl(H(u_k)\bigr),
    \end{equation}
    with \(\nu_k>0\).
    Moreover, due to \ref{B2}, we have $\nabla \barrierfun_\phi(z)\in Z_+^\ast$ for all  $z\in \interior(Z_-)$.
    Therefore, it holds \(\mu_k\in Z_+^\ast\).
    
    \ref{lem:outer_iterates_behaviour:iii}
    By the definition of \(\costs_{\nu_k}\), we can write 
    \[
    \nabla \costs_{\nu_k}(u_k)
    =
    \nabla \costr(u_k) + \nu_k H'(u_k)^\ast \nabla \barrierfun_\phi\bigl(H(u_k)\bigr)
    =
    \nabla \costr(u_k) + H'(u_k)^\ast \mu_k .
    \]
    Then, substituting the above identity in the \(\varepsilon_k\)-stationarity condition for the barrier subproblem yields
    \[
    \dist\left(
    0,\nabla \costr(u_k) + H'(u_k)^\ast \mu_k + \partial \costp(u_k)
    \right)\le \varepsilon_k,
    \]
    which is equivalent to the stated estimate. In particular, there exists \(\zeta_k\in\partial\costp(u_k)\) such that
    \begin{equation}\label{eq:approx_stationarity_xi}
    \|\nabla \costs_{\nu_k}(u_k)+\zeta_k\|\le \varepsilon_k.
    \end{equation}
    
    \ref{lem:outer_iterates_behaviour:iv}
    Assume \(\phi\ge 0\) on \((-\infty,0)\).
    Then \(\barrierfun_\phi\ge 0\) on \(\interior(Z_-)\), and hence
    $(\costr+\costp)(u_k)\le (\costs_{\nu_k}+\costp)(u_k)$.
    The second inequality follows from the descent property of \cref{alg:NMPG} at the \(k\)-th outer iteration, since the subproblem solve is initialized at \(u_{k-1}\).
    Finally, because \(\barrierfun_\phi\ge 0\) and \(0\le \nu_k\le \nu_{k-1}\), the third inequality holds too, proving the claim.
\end{proof}

All iterates $u_k$ generated by \cref{alg:IP} are (strictly) feasible by construction, while stationarity is controlled by the inner tolerance $\varepsilon_k$, as shown by \cref{lem:outer_iterates_behaviour}\ref{lem:outer_iterates_behaviour:iii}.
Therefore, as a termination criterion for \cref{alg:IP}, we monitor inner stationarity and use an approximate \emph{complementarity} condition for a given tolerance $\varepsilon \ge 0$.
The practical termination conditions for $\varepsilon$-KKT optimality read
\begin{subequations}
    \label{eq:terminationEpsKKT}
    \begin{align}
        \dist\bigl(0, \nabla \costs_{\nu_k}(u_k)  + \partial\costp(u_k)\bigr) &\leq \varepsilon, \\ \label{eq:terminComp}
        \bigl| \langle \mu_k, H(u_k) \rangle_{Z^{\ast},Z} \bigr| &\leq \varepsilon.     
    \end{align}
\end{subequations}
These conditions are sufficient for approximate KKT optimality of $u_k$ for \eqref{eq:Reduced}, since the primal feasibility $H(u_k) \le_Z 0$ and dual feasibility $\mu_k \in Z_+^{\ast}$ hold by construction; cf. \cref{lem:outer_iterates_behaviour}.

To bound the complementarity residual for the power barrier, we need a direction along which the linearized constraint variation points strictly into the interior of the cone near the constraint boundary.
The following lemma shows that such a direction exists and remains uniformly bounded along the barrier path, as a consequence of the linearized Slater condition \ref{CQ1}.

\begin{lemma}[Stability of an inward-pointing condition along a convergent barrier path]\label{lem:IP}
    Let $\{\nu_k\}_{k\in\N}\subset(0,\infty)$ satisfy \(\nu_k\downarrow 0\), and let
    \(u_k\in \Us\) be an \(\varepsilon_k\)-stationary point of \eqref{eq:Reduced_barrier_problem}, where \(\varepsilon_k\downarrow 0\).
    Assume that:
    \begin{enumerate}[label=(S\arabic*)]
        \item\label{Stability:A1}%
        \(u_k\rightharpoonup u^\star\) in \(\Us\) for some \(u^\star\in\mathcal D\), and the linearized Slater condition \ref{CQ1} holds at \(u^\star\).
        \item\label{Stability:A2}%
        The mapping $H:\mathcal D \to Z$ is sequentially weak-to-strong continuous, and the derivatives converge strongly on each fixed direction.
        More precisely, if $\{u_j\} \subset \mathcal D$ and $\bar u\in\mathcal D$, then
        \(u_j\rightharpoonup \bar u\) in \(\Us\) implies \(H(u_j)\to H(\bar u)\) in \(Z\) and
        \(H'(u_j)v \to H'( \bar u)v\) in \(Z\) for every \(v\in\Us\).
    \end{enumerate}
    Then there exist a direction \(d\in \Us\), constants \(C_d>0\), \(\kappa>0\), \(\delta_0>0\), \(M_z>0\), and \(k_0\in\N\) such that
    $\|d\|_{\Us}\le C_d$,
    and, for every \(k\ge k_0\) and every \(\delta\in(0,\delta_0]\), the linearized constraint variation
    $z_k\coloneqq H'(u_k)\,d \in Z$
    satisfies
    \begin{equation}\label{eq:IP}
    \|z_k\|_{Z} \le M_z,
    \qquad
    z_k \le_Z -\kappa\,\mathbf{1}
    \ \ \text{on the set}\ \
    \{\xi\in\Xi:\ 0<-[H(u_k)](\xi)\le \delta\}.
    \end{equation}
\end{lemma}
\begin{proof}
    By \ref{Stability:A1}, we choose $v\in\dom\costp$ from \ref{CQ1} and write $\bar d=v-u^\star$.
    Then there exists $\eta>0$ such that
    \begin{equation}\label{eq:linearized_slater_polished}
        H(u^\star)+H'(u^\star)\bar d \le_Z -\eta\,\mathbf{1},
    \end{equation}
    equivalently \(H(u^\star)+H'(u^\star)\bar d\in \interior (Z_-)\).
    We fix \(\delta_0\in(0,\eta/2)\) and set \(z^\star\coloneqq H'(u^\star)\bar d\in Z\).
    Let \(\xi\in\Xi\) be such that \(0\le -H(u^\star)(\xi)\le 2\delta_0\), i.e., \(H(u^\star)(\xi)\ge -2\delta_0\).
    Evaluating \eqref{eq:linearized_slater_polished} at \(\xi\) yields
    \[
    H(u^\star)(\xi)+z^\star(\xi)\le -\eta,
    \qquad\text{hence}\qquad
    z^\star(\xi)\le -\eta - H(u^\star)(\xi)\le -\eta + 2\delta_0.
    \]
    Defining  $\kappa \coloneqq \frac{2}{3}\,(\eta-2\delta_0)>0$, we obtain that 
    \begin{equation}\label{eq:neg_on_sublevel_polished}
        z^\star(\xi)\le -\frac{3}{2}\kappa
        \qquad
        \forall\,\xi\in\Xi\ \text{with}\ 0\le -H(u^\star)(\xi)\le 2\delta_0 .
    \end{equation}
    By \ref{Stability:A2} and \(u_k\rightharpoonup u^\star\), we have \(H(u_k)\to H(u^\star)\) in \(Z\).
    Thus there exists \(k_1\in\N\) such that
    \begin{equation}\label{eq:uniform_s_polished}
        \|H(u_k)-H(u^\star)\|_{Z}\le \delta_0
        \qquad
        \forall\,k\ge k_1.
    \end{equation}
    Consequently, for any \(k\ge k_1\) and any \(\xi\in\Xi\) with \(0<-[H(u_k)](\xi)\le \delta_0\), we obtain
    \[
    -H(u^\star)(\xi)
    \le -H(u_k)(\xi) + |H(u_k)(\xi)-H(u^\star)(\xi)|
    \le \delta_0+\delta_0=2\delta_0.
    \]
    Next, we write $d=\bar d$ and $z_k=H'(u_k)d$.
    By \ref{Stability:A2} we have \(z_k\to z^\star\) in \(Z\).
    Hence there exists \(k_2\in\N\) such that
    \begin{equation}\label{eq:uniform_z_polished}
        \|z_k-z^\star\|_{Z}\le \frac{1}{2}\kappa
        \qquad \forall\,k\ge k_2.
    \end{equation}
    Let \(k_0\coloneqq \max\{k_1,k_2\}\).
    For \(k\ge k_0\) and \(\xi\in\Xi\) with \(0<-[H(u_k)](\xi)\le \delta_0\), we have shown that \(-H(u^\star)(\xi)\le 2\delta_0\), and thus by \eqref{eq:neg_on_sublevel_polished}, \(z^\star(\xi)\le -\frac{3}{2}\kappa\).
    Combining this with \eqref{eq:uniform_z_polished} gives
    \[
    z_k(\xi)\le z^\star(\xi)+|z_k(\xi)-z^\star(\xi)|
    \le -\frac{3}{2}\kappa+\frac{1}{2}\kappa
    =-\kappa,
    \]
    which proves the inward-pointing estimate in \eqref{eq:IP} (with \(\delta\le \delta_0\)).
    Finally, the norm bound follows from \eqref{eq:uniform_z_polished}:
    \[
    \|z_k\|_{Z}\le \|z^\star\|_{Z}+\|z_k-z^\star\|_{Z}
    \le \|z^\star\|_{Z}+\frac{1}{2}\kappa \eqqcolon M_z.
    \]
    Setting \(C_d \coloneqq 1+\|d\|_U\) completes the proof.
\end{proof}

We now state the main outer-loop result, which quantifies how many barrier parameter updates are needed to drive the complementarity residual below a prescribed tolerance $\varepsilon$, distinguishing the logarithmic and power barriers.

\begin{theorem}[Outer iteration complexity]\label{thm:outer-convergence}
    Let \(\{(\nu_k,\varepsilon_k,u_k,\mu_k)\}_{k\in\N}
    \subset (0,\infty)^2\times \Us\times Z^\ast\) be a sequence generated by \cref{alg:IP}.
    Define the complementarity residual
    \[
    r_k \;\coloneqq\; \bigl|\langle \mu_k,\,H(u_k)\rangle_{Z^\ast,Z}\bigr|.
    \]
    Then \(r_k\to 0\) as \(k\to\infty\) for the barrier choices listed below.
    Moreover, the number of outer iterations needed to find an $\varepsilon$-KKT point admits the following bounds:
    \begin{enumerate}[label=(\arabic*)]
    \item \textbf{Logarithmic barrier.}
        Then \(r_k=\meas(\Xi)\,\nu_k\) and hence \(r_k\to 0\).
        In particular, it is sufficient to perform
        \begin{equation}\label{eq:complexity_outer_iterations:log_barrier}
        \Niterout(\varepsilon)
        =1+\left\lceil\max\left\{
        0,\frac{\log(\varepsilon_1/\varepsilon)}{\log(1/\theta_\varepsilon)},
        \frac{\log(\meas(\Xi)\nu_1/\varepsilon)}{\log(1/\theta_\nu)}
        \right\}\right\rceil
        \end{equation}
        outer iterations to satisfy \eqref{eq:terminationEpsKKT} for all \(k\ge \Niterout(\varepsilon)\).
    \item \textbf{Power barrier with \(p>0\).}
        Assume, in addition to \ref{Stability:A1}--\ref{Stability:A2}, that
        \begin{enumerate}[label=(S\arabic*),start=3]
            \item \label{Stability:A3}
            \(\nabla \costr:\mathcal{D} \to \Us\) is sequentially weak-to-strong continuous.
            \item \label{unifrom_bounded:A4}
            the sequence of $\zeta_k\in\partial \costp(u_k)$ in \eqref{eq:approx_stationarity_xi} is uniformly bounded (i.e., $\sup_{k}\|\zeta_k\|_U \leq C_\zeta$ for some  $C_\zeta>0$).
        \end{enumerate}
        Then, there exists a constant \(C>0\) (depending only on \(p\), \(\meas(\Xi)\), $C_\zeta$, the stability constants from \cref{lem:IP}, and the boundedness constants along the generated sequence) such that
        \[
        r_k \le C \nu_k^{\frac{1}{p+1}} \qquad \text{for all }k\ge k_0.
        \]
        Consequently, it is sufficient to perform
        \begin{equation}\label{eq:complexity_outer_iterations:power_barrier}
        \Niterout(\varepsilon)=\max\left\{k_0,\
        1+\left\lceil\max\left\{
        0,\frac{\log(\varepsilon_1/\varepsilon)}{\log(1/\theta_\varepsilon)},
        \frac{(p+1)\log(C\nu_1^{1/(p+1)}/\varepsilon)}{\log(1/\theta_\nu)}
        \right\}\right\rceil\right\}.
        \end{equation}
        outer iterations to satisfy \eqref{eq:terminationEpsKKT} for all \(k\ge \Niterout(\varepsilon)\).
    \end{enumerate}
\end{theorem}
\begin{proof}
    Let a tolerance $\varepsilon>0$ be fixed and recall that the barrier parameters and inner tolerances are decreased geometrically in \cref{alg:IP}:
    $\nu_k=\nu_1 \theta_\nu^{k-1}$,
    $\varepsilon_k=\varepsilon_1 \theta_\varepsilon^{k-1}$,
    $\theta_\nu,\theta_\varepsilon\in(0,1)$.
    The first terms in \eqref{eq:complexity_outer_iterations:log_barrier} and \eqref{eq:complexity_outer_iterations:power_barrier} count the outer iterations needed to reach $\varepsilon_k \leq \varepsilon$ and, thus, to satisfy $\varepsilon$-stationarity.
    It remains to estimate the complementarity residual \(r_k\) for the two barrier kernels.
    
    \textbf{Logarithmic barrier.}
    For \(\phi_{\log}(t)=-\log(-t)\) we have \(\phi'_{\log}(t)=-1/t\) on \((-\infty,0)\). Hence the associated multiplier along the barrier subproblem satisfies
    \[
    \mu_{k}=\nu_k\,\phi_{\log}'(H(u_k))=-\nu_k\,H(u_k)^{-1}\in Z_+^\ast ,
    \]
    interpreted componentwise if \(Z=\R^q\) and pointwise if \(Z=C(\overline K)\).
    Therefore, we have 
    \[
    r_k
    =\bigl|\langle \mu_k,\,H(u_k)\rangle_{Z^\ast,Z}\bigr|
    =\nu_k \left|\int_{\Xi} \frac{-H(u_k)(\xi)}{H(u_k)(\xi)} \mathrm{d} \meas(\xi)\right|
    =\nu_k \meas(\Xi)
    =\nu_1 \meas(\Xi) \theta_\nu^{k-1}.
    \]
    In particular, \(r_k\to 0\) as \(k\to\infty\).
    To ensure $r_k\leq\varepsilon$, we require
    $\meas(\Xi)\nu_1 \theta_{\nu}^{k-1} \leq \varepsilon$, or equivalently $\theta_\nu^{k-1} \leq\frac{\varepsilon}{\meas(\Xi)\nu_1}$.
    Since $\ln\theta_{\nu}<0$, this yields the second term in the iteration bound \eqref{eq:complexity_outer_iterations:log_barrier}, thereby completing its verification.
    
    \textbf{Power barrier.}
    Let \(p>0\) and \(\phi_{\rm pow}(t)=(-t)^{-p}\). Then we have 
    \[
    r_k
    =\bigl|\langle \mu_k,\,H(u_k)\rangle_{Z^\ast,Z}\bigr|
    =\nu_k\,p\int_{\Xi}(-H(u_k)(\xi))^{-p} \mathrm{d} \meas(\xi).
    \] 
    We next apply \cref{lem:IP}, thanks to the additional assumptions.
    We choose the direction $d \in \Us$, $k_0$, and $\delta_0$ from \cref{lem:IP}.
    Then $z_k\coloneqq H'(u_k)d\in Z$ satisfies \eqref{eq:IP} for every $\delta\in(0,\delta_0]$ and $k\geq k_0$.
    We also define the sets
    \[
    A(\delta)\coloneqq\{\xi\in\Xi:\ 0< -H(u_k)(\xi)\le \delta\},\qquad
    B(\delta)\coloneqq \Xi\setminus A(\delta).
    \]
    Let \(\costs_{\nu}(u)\coloneqq \costr(u)+\nu\,\barrierfun _{\rm pow}(H(u))\), where
    \(\barrierfun _{\rm pow}(z)=\int_{\Xi}(-z(\xi))^{-p} \mathrm{d} \meas(\xi)\).
    A direct differentiation yields
    \[
    \bigl\langle \nabla \costs_{\nu_k}(u_k),\,d\bigr\rangle_{\Us}
    =
    \bigl\langle \nabla \costr(u_k),\,d\bigr\rangle_{\Us}
    +\nu_k\,p\int_{\Xi} z_k(\xi)\,(-H(u_k)(\xi))^{-(p+1)}\mathrm{d}\meas(\xi).
    \]
    Choose \(\zeta_k\in\partial \costp(u_k)\) such that
    \(\|\nabla \costs_{\nu_k}(u_k)+\zeta_k\|_{\Us}\le \varepsilon_k\).
    Then we obtain
    \begin{equation}\label{eq:stationarity_dir}
        \bigl\langle \nabla \costs_{\nu_k}(u_k),\,d\bigr\rangle_{\Us}
        =
        \langle -\zeta_k, d\rangle_{\Us}
        +\langle \nabla \costs_{\nu_k}(u_k)+\zeta_k, d\rangle_{\Us}
        \ge
        -(\|\zeta_k\|_{\Us}+\varepsilon_k) \|d\|_{\Us}
        \ge
        -(\|\zeta_k\|_{\Us}+\varepsilon_k) C_d.
    \end{equation}
    Moreover, since \(\{u_k\}_k \subset \mathcal D\) is bounded , it follows from \ref{Stability:A3} that there exists \(G>0\) with \(\|\nabla\costr(u_k)\|_{\Us}\le G\) for all \(k\).
    Hence it holds \(\langle \nabla\costr(u_k),d\rangle_{\Us}\le G C_d\).
    We now estimate the barrier integral by splitting over \(A(\delta)\) and \(B(\delta)\).
    On \(A\), \cref{lem:IP} gives \(z_k(\xi)\le -\kappa\), and \(-H(u_k)(\xi)\le\delta\), so
    \begin{equation*}
    \begin{split}
    \int_{A(\delta)} z_k(\xi)\,(-H(u_k)(\xi))^{-(p+1)}\mathrm{d}\meas(\xi)
     &\le
    -\kappa\int_{A(\delta)}(-H(u_k)(\xi))^{-(p+1)}\mathrm{d}\meas(\xi)\\&
    \le
    -\kappa\,\delta^{-1}\int_{A(\delta)}(-H(u_k)(\xi))^{-p}\mathrm{d}\meas(\xi).
    \end{split}
    \end{equation*}
    Furthermore, on \(B(\delta)\), we have \(-H(u_k) (\xi)>\delta\).
    Hence \(( -H(u_k)(\xi))^{-(p+1)} < \delta^{-(p+1)}\), and using \(\|z_k\|_Z\le M_z\) yields
    \[
    \int_{B(\delta)} z_k(\xi)\,(-H(u_k)(\xi))^{-(p+1)}\mathrm{d}\meas(\xi)
    \le
    M_z\,\delta^{-(p+1)}\,\meas(B(\delta))
    \le
    M_z\,\delta^{-(p+1)}\,\meas(\Xi).
    \]
    Combining these estimates with \eqref{eq:stationarity_dir} gives
    \[
    -(\|\zeta_k\|_{\Us}+\varepsilon_k) C_d
    \le \bigl\langle \nabla \costs_{\nu_k}(u_k),\,d\bigr\rangle_{\Us} \le
    G C_d
    +\nu_k\,p\,M_z\,\delta^{-(p+1)}\,\meas(\Xi)
    -\nu_k\,p\,\kappa\,\delta^{-1}\int_{A(\delta)}(-H(u_k))^{-p}\mathrm{d}\meas,
    \]
    and, therefore, we obtain that 
    \begin{equation}\label{eq:intA_bound}
    \nu_k\,p\int_{A(\delta)}(-H(u_k))^{-p}\mathrm{d}\meas
    \le
    \frac{C_d (G+\varepsilon_k+\|\zeta_k\|_{\Us})}{\kappa}\,\delta
    +\frac{\nu_k\,p\,M_z}{\kappa}\,\delta^{-p}\,\meas(\Xi).
    \end{equation}
    Here \(C_d\) comes from the bound on the inward-pointing direction, \(G\) from the stationarity bound, and \(M_z\) from the uniform bound on \(H'(u_k)d\); all are independent of \(k\) and \(\delta\).
    Further, on \(B(\delta)\), we have  \((-H(u_k))^{-p} < \delta^{-p}\), and as a consequence, we can write  
    \begin{equation}\label{eq:intB_bound}
    \nu_k\,p\int_{B(\delta)}(-H(u_k))^{-p}\mathrm{d}\meas
    \le \nu_k\,p\,\delta^{-p}\,\meas(B(\delta)) \leq
    \nu_k\,p\,\delta^{-p}\,\meas(\Xi).
    \end{equation}
    Adding \eqref{eq:intA_bound} and \eqref{eq:intB_bound} yields
    \begin{equation*}
    r_k
    =\nu_k\,p\int_{\Xi}(-H(u_k))^{-p}\mathrm{d}\meas
    \le
    \frac{C_d (G+\varepsilon_k+\|\zeta_k\|_{\Us})}{\kappa}\,\delta
    +\frac{\nu_k\,p\,(M_z+\kappa)}{\kappa}\,\delta^{-p}\,\meas(\Xi).
    \end{equation*}
    Finally, using \ref{unifrom_bounded:A4}, we  obtain a bound of the form
    \[
    r_k \;\le\; C_1\,(G+\varepsilon_k+C_\zeta)\,\delta + C_2\,\nu_k\,\delta^{-p},
    \]
    with constants \(C_1,C_2>0\) independent of \(k\) and \(\delta\).
    We enlarge $k_0$, if necessary, so that $\nu_k^{1/(p+1)}\le\delta_0$ for every $k\ge k_0$. Choosing $\delta=\nu_k^{1/(p+1)}$ then balances both terms and yields, for every $k\ge k_0$,
    \begin{equation*}
    r_k \;\le\; C\,\nu_k^{\frac{1}{p+1}}
    \qquad \text{for some constant } C>0,
    \end{equation*}
    hence \(r_k\to 0\) as \(k\to\infty\).
    To obtain the outer complexity bound, we use $\nu_k=\nu_1\theta_\nu^{k-1}$.
    For $k\ge k_0$, the estimate $r_k\le C\nu_k^{1/(p+1)}$ guarantees $r_k\le\varepsilon$ whenever $\nu_1\theta_\nu^{k-1}\le (\varepsilon/C)^{p+1}$.
    Since \(\ln(\theta_\nu)<0\), this yields the second term in the claimed estimate \eqref{eq:complexity_outer_iterations:power_barrier}.
\end{proof}

\begin{remark}
    If both $U$ and $Z$ are finite-dimensional, Assumptions~\ref{Stability:A2} and~\ref{Stability:A3} follow from Assumption~\ref{G2}.
    Indeed, in finite-dimensional spaces, weak and strong convergence coincide, and all norms are equivalent. 
\end{remark}

For the following result and the subsequent KKT-limit analysis, we use the collective compactness condition:
\begin{enumerate}[label=(S\arabic*),start=5]
    \item\label{barrier-to-KKT:A5} For every bounded set $B\subset Z^\ast$, the set
    \[
    \mathcal K(B)\coloneqq
    \{\,H'(u_k)^\ast\mu:k\in\N,\ \mu\in B\,\}
    \]
    is relatively compact in $U$; that is, the family
    $\{H'(u_k)^\ast:k\in\N\}$ is collectively compact on bounded subsets of $Z^\ast$.
\end{enumerate}

\begin{remark}
    The approximate KKT estimates and outer-iteration bounds in \cref{thm:outer-convergence} do not require the collective compactness condition \ref{barrier-to-KKT:A5}.
    For the power kernel, these results retain \ref{unifrom_bounded:A4}, whereas neither \ref{unifrom_bounded:A4} nor \ref{barrier-to-KKT:A5} is required for the logarithmic kernel.
    We use \ref{barrier-to-KKT:A5} in the subsequent KKT-limit analysis for both kernels; it also allows us to verify \ref{unifrom_bounded:A4} through \cref{lem:A4-compactness}, rather than assume it separately.
\end{remark}

\begin{lemma}\label{lem:A4-compactness}
    In the setting of \cref{lem:IP}, we assume \ref{Stability:A1}--\ref{Stability:A3} and \ref{barrier-to-KKT:A5}.
    For either the logarithmic barrier kernel or the power barrier kernel with $p>0$, the multipliers $\mu_k$ in \eqref{eq:lag_Mul} and the selected subgradients $\zeta_k$ in \eqref{eq:approx_stationarity_xi} are uniformly bounded.
    In particular, \ref{unifrom_bounded:A4} holds.
\end{lemma}
\begin{proof}
We first prove that
\begin{equation}\label{eq:moving-direction-limit}
    \lim_{k\to\infty}
    \|H'(u_k)(u_k-u^\star)\|_Z=0.
\end{equation}
By the dual characterization of the norm and the definition of the adjoint, we can write
\[
    \|H'(u_k)(u_k-u^\star)\|_Z
    =
    \sup_{\|\lambda\|_{Z^\ast}\le1}
    \bigl|
    \langle H'(u_k)^\ast\lambda,u_k-u^\star\rangle_U
    \bigr|.
\]
We argue by contradiction.
Suppose that \eqref{eq:moving-direction-limit} does not hold.
Then, there exists $\delta>0$, a subsequence $\{u_{k_j}\}_j$, and $\lambda_j\in Z^\ast$ with $\|\lambda_j\|_{Z^\ast}\le1$ such that
\[
    \bigl|
    \langle H'(u_{k_j})^\ast\lambda_j,
            u_{k_j}-u^\star\rangle_U
    \bigr|\ge\delta
    \qquad\text{for all }j.
\]
By \ref{barrier-to-KKT:A5}, we can choose a further subsequence and $w\in U$ such that
\[
    \lim_{\ell\to\infty}
    \|H'(u_{k_{j_\ell}})^\ast\lambda_{j_\ell}-w\|_U=0.
\]
Since $u_{k_{j_\ell}}\rightharpoonup u^\star$ in $U$, the sequence $\{u_{k_{j_\ell}}-u^\star\}_\ell$ is bounded,
and we obtain
\[
\bigl|
 \langle H'(u_{k_{j_\ell}})^\ast\lambda_{j_\ell},
         u_{k_{j_\ell}}-u^\star\rangle_U
 \bigr|
\le
 \|H'(u_{k_{j_\ell}})^\ast\lambda_{j_\ell}-w\|_U \|u_{k_{j_\ell}}-u^\star\|_U +|\langle w,u_{k_{j_\ell}}-u^\star\rangle_U| \to 0
 \quad \text{ as } \ell\to\infty.
\]
This contradicts the lower bound $\delta$ and proves \eqref{eq:moving-direction-limit}. 

Now, applying \ref{Stability:A2} to the fixed direction \(v-u^\star\), for any \(v\in\dom\costp\), together with \eqref{eq:moving-direction-limit}, we obtain
\begin{equation}
\label{eq:boundedness_of_mul_conv}
\begin{split}
&\|H(u_k)+H'(u_k)(v-u_k)
      -H(u^\star)-H'(u^\star)(v-u^\star)\|_Z\\
&\le
 \|H(u_k)-H(u^\star)\|_Z
 +\|(H'(u_k)-H'(u^\star))(v-u^\star)\|_Z
 +\|H'(u_k)(u_k-u^\star)\|_Z
 \to0
 \quad \text{as } k\to\infty.
\end{split}
\end{equation}
By \ref{CQ1}, we can choose  \(v\in\dom\costp\) an  $\eta>0$ such that
\[
    H(u^\star)+H'(u^\star)(v-u^\star)
    \le_Z-2\eta\mathbf1.
\]
In both settings, $Z=\R^q$ with the Euclidean norm and
$Z=C(\overline K)$ with the supremum norm, we know that $ \|z\|_Z\le\eta$ implies that $-\eta\mathbf1\le_Z z\le_Z\eta\mathbf1$.
Thus, the norm convergence \eqref{eq:boundedness_of_mul_conv} implies, for all
sufficiently large $k$,
\begin{equation}
\label{eq:boundedness_of_mul_conv2}
    H(u_k)+H'(u_k)(v-u_k) \le_Z H(u^\star)+H'(u^\star)(v-u^\star)  +\eta\mathbf1\le_Z-\eta\mathbf1.
\end{equation}
Since $\costp$ is proper, convex, and lsc, it is weakly lsc on $\Usad$.
By \ref{Stability:A1}, $u_k \rightharpoonup u^\star$, so $\costp(u^\star) \leq \liminf_{k\to\infty} \costp(u_k)$.
Hence $\{\costp(u_k)\}$ is bounded below along the tail of the sequence; since the finitely many remaining values are finite, $\inf_k\costp(u_k)>-\infty$.
Therefore, from $\zeta_k\in\partial\costp(u_k)$, we obtain that 
\begin{equation}\label{eq:boundedness_of_mul_es1}
 \langle\zeta_k,v-u_k\rangle_U  \le\costp(v)-\costp(u_k)\le C
\end{equation}
with $C$ independent of $k$.
By the definition of $\mu_k$ in \eqref{eq:lag_Mul}, we can write
\begin{equation}
\label{eq:boundedness_of_mul_es2}
\begin{split}
-\langle\mu_k,H'(u_k)(v-u_k)\rangle_{Z^\ast,Z} &=-\langle H'(u_k)^\ast\mu_k,v-u_k\rangle_U \\ &=\langle\nabla\costr(u_k)+\zeta_k,v-u_k\rangle_U-\langle\nabla\costs_{\nu_k}(u_k)+\zeta_k,v-u_k\rangle_U.
\end{split}
\end{equation}
By \eqref{eq:approx_stationarity_xi} and the
Cauchy--Schwarz inequality, we obtain
\begin{equation}
\label{eq:boundedness_of_mul_es3}
-\langle\nabla\costs_{\nu_k}(u_k)+\zeta_k,v-u_k\rangle_U \le \|\nabla\costs_{\nu_k}(u_k)+\zeta_k\|_U \|v-u_k\|_U\le\varepsilon_k\|v-u_k\|_U.
\end{equation}
Further, since $\mu_k\in Z_+^\ast$, applying this positive linear functional to both sides of \eqref{eq:boundedness_of_mul_conv2}, we obtain
\begin{equation}
\label{eq:boundedness_of_mul_es4}
\langle\mu_k,H(u_k)\rangle_{Z^\ast,Z}
+\langle\mu_k,H'(u_k)(v-u_k)\rangle_{Z^\ast,Z}
\le-\eta\langle\mu_k,\mathbf1\rangle_{Z^\ast,Z}.
\end{equation}
We write $M_k=\langle\mu_k,\mathbf1\rangle_{Z^\ast,Z}$ and $r_k=-\langle\mu_k,H(u_k)\rangle_{Z^\ast,Z}$.
Using approximate stationarity \eqref{eq:approx_stationarity_xi} and \eqref{eq:boundedness_of_mul_es1}-\eqref{eq:boundedness_of_mul_es4}, we can then estimate
\begin{equation}
\label{eq:boundedness_of_mul_es5}
\begin{split}
\eta M_k
&\le r_k-\langle\mu_k,H'(u_k)(v-u_k)\rangle_{Z^\ast,Z}\\
&\le r_k+\langle\nabla\costr(u_k),v-u_k\rangle_U
 +\langle\zeta_k,v-u_k\rangle_U
 +\varepsilon_k\|v-u_k\|_U\le r_k+C.
\end{split}
\end{equation}
Here, the last inequality follows \ref{Stability:A3}, and boundedness
of $\{u_k\}_k$ and $\{\varepsilon_k\}_k$. For the logarithmic kernel, we have $r_k=\nu_k\meas(\Xi)$, so the preceding estimate immediately gives boundedness of $M_k$. For the power kernel, we set first $\rho_k\coloneqq p\nu_k(-H(u_k))^{-(p+1)}$. Then, using  \eqref{eq:lag_Mul}, we can write 
\[
\begin{aligned}
r_k
&=p\nu_k\int_\Xi(-H(u_k))^{-p}\,\mathrm d\meas =(p\nu_k)^{1/(p+1)}
  \int_\Xi\rho_k^{p/(p+1)}\,\mathrm d\meas\\
&\le(p\nu_k)^{1/(p+1)}
  \left(\int_\Xi\rho_k\,\mathrm d\meas\right)^{p/(p+1)}
  \meas(\Xi)^{1/(p+1)}=(p\nu_k)^{1/(p+1)}
  M_k^{p/(p+1)}\meas(\Xi)^{1/(p+1)},
\end{aligned}
\]
where we used H\"older's inequality with conjugate
exponents $(p+1)/p$ and $p+1$. Combining this estimate with \eqref{eq:boundedness_of_mul_es5}, we obtain
\[
    \eta M_k
    \le C+(p\nu_k\meas(\Xi))^{1/(p+1)}
           M_k^{p/(p+1)}.
\]
If $M_k\le 2C/\eta$, we already have the desired bound.
Otherwise, we have $M_k> 2C/\eta$  subtracting $C$ and dividing by
$M_k^{p/(p+1)}$ gives
\[
    \frac{\eta}{2}M_k^{1/(p+1)}
    \le(p\nu_k\meas(\Xi))^{1/(p+1)}.
\]
Consequently, since $\nu_k\le\nu_1$, we obtain
\[
    M_k\le
    \max\left\{
        \frac{2C}{\eta},
        \left(\frac{2}{\eta}\right)^{p+1}
        p\,\nu_1\,\meas(\Xi)
    \right\}
\]
for all sufficiently large $k$. Since the finitely many remaining terms are finite, the entire sequence $\{M_k\}_k$ is bounded. In either case, $\|\mu_k\|_{Z^*}\le M_k$ for positive functionals in the two model spaces, so the multipliers are bounded.
The operators $H'(u_k)$ are uniformly bounded by \ref{Stability:A2} and the uniform boundedness principle.
Finally, approximate stationarity gives
\[
\|\zeta_k\|_U
\le\varepsilon_k+\|\nabla\costr(u_k)\|_U
+\|H'(u_k)\|_{\mathcal L(U,Z)}\|\mu_k\|_{Z^*}
\le C,
\]
which proves \ref{unifrom_bounded:A4}.
\end{proof}

\begin{remark}
 The uniform boundedness of the subgradients $\zeta_k \in \partial \costp(u_k)$ in \eqref{eq:approx_stationarity_xi}, as required in~\ref{unifrom_bounded:A4}, holds under standard assumptions whenever $\{u_k\}_k$ is bounded in $\Us$.
    Indeed, if $\costp$ is proper, lsc, convex, and Lipschitz continuous on bounded subsets of $\Us$, then subgradients are uniformly bounded on bounded sets.
    Typical examples include
    \[
    \costp(u)=\alpha\|u\|_{\Us},\qquad
    \costp(u)=\alpha\|u\|_{L^1(\Omega)}\ \bigl(\Us=L^2(\Omega),\ |\Omega|<\infty\bigr),\qquad
    \costp(u)=\alpha\|u\|_{1}\ \bigl(\Us=\R^m\bigr),
    \]
    and, in the matrix-valued case,
    \[
    \costp(X)=\alpha\|X\|_*,\qquad \Us=\R^{m\times n}\ \text{with Frobenius norm } \|\cdot\|_F.
    \]
    Hence, these functionals satisfy~\ref{unifrom_bounded:A4} along bounded sequences.
    For an indicator $\costp=\indicator_C$, membership of $u_k$ in $C$ does not imply boundedness of the selected subgradients.
    Assumption~\ref{unifrom_bounded:A4} concerns the subgradients satisfying \eqref{eq:approx_stationarity_xi}, not the whole subdifferential.
    Alternatively, \cref{lem:A4-compactness} verifies this boundedness for logarithmic and power barriers under its hypotheses, without requiring Lipschitz continuity of $\costp$.
\end{remark}

In general, condition~\ref{Stability:A1} holds only along a subsequence $\{u_{k,j}\}_j \subset \{u_k\}_k$ such that $u_{k,j} \rightharpoonup  u^{\star}$ for some $u^{\star}$ satisfying~\ref{CQ1}.
In this situation, \cref{thm:outer-convergence} implies for the power-type barrier that $r_{k,j}\to 0$ as $j \to \infty$, but without a bound on the gaps between its indices, this does not give a global outer-iteration count.
Consequently, we can only conclude that
\begin{equation}
    \liminf_{k\to \infty} r_k
    = \liminf_{k \to \infty} \bigl|\langle \mu_k, H(u_k) \rangle_{Z^{\ast},Z} \bigr|
    = 0.
\end{equation}
In the following, we focus on settings in which condition~\ref{Stability:A1} holds for the entire sequence $\{u_k\}_k$, so that the complexity result \eqref{eq:complexity_outer_iterations:power_barrier} is applicable globally.

We can now state the main convergence theorem, which combines multiplier boundedness, complementarity convergence, and the collective compactness condition \ref{barrier-to-KKT:A5} to show that accumulation points of the primal-dual sequence generated by \cref{alg:IP} satisfy the KKT conditions of the original problem \eqref{eq:Reduced}.

\begin{theorem}[Barrier limit yields KKT]\label{prop:barrier-to-KKT}
    Let an infinite sequence $\{(\nu_k,\varepsilon_k,u_k,\mu_k)\}_k \subset (0,\infty)^2 \times U \times Z^{\ast}$ be generated by \cref{alg:IP}, using either logarithmic or power barriers.
    Assume \ref{Stability:A1}--\ref{Stability:A3} and \ref{barrier-to-KKT:A5}.
    Then there exist a subsequence, not relabeled, and $\mu^{\star} \in Z_+^{\ast}$ such that
    \[
    u_k \rightharpoonup u^{\star} \ \text{ in } U
    , \qquad
    H(u_k) \to H(u^{\star}) \le_Z 0 \ \text{ in } Z
    ,\qquad
    \mu_k \rightharpoonup^{\ast} \mu^{\star}
    , \qquad
    \langle \mu^{\star}, H(u^{\star}) \rangle_{Z^{\ast},Z} = 0,
    \]
    where $(u^{\star},\mu^{\star})$ fulfills the KKT inclusion
    \begin{equation}\label{eq:KKT-limit}
        0 \in \nabla\costr(u^{\star}) + \partial \costp (u^{\star}) + H'(u^{\star})^{\ast}\mu^{\star}.
    \end{equation}
\end{theorem}
\begin{proof}
    By \ref{Stability:A1} and \ref{Stability:A2}, we have $H(u_k)\to H(u^{\star})$ in $Z$.
    Since $H(u_k)\in \interior (Z_-)$ for all $k$ and $Z_-$ is closed, we get $H(u^\star)\in Z_-$.
    
    By \cref{lem:A4-compactness}, the multipliers $\{\mu_k\}_k$ are bounded in $Z^\ast$ and \ref{unifrom_bounded:A4} holds for both barrier kernels.
    If $Z=\R^q$, this gives compactness directly because the space is finite-dimensional.
    If $Z=C(K)$ with $K$ compact metric, then $C(K)$ is separable; hence bounded subsets of $C(K)^\ast$ are weak-$\ast$ metrizable on bounded sets.
    By the Banach--Alaoglu theorem \cite[Thm~3.16]{Brez2011functional}, there exist $\mu^\star\in Z^\ast$ and a further subsequence (not relabeled) such that $\mu_k\rightharpoonup^\ast\mu^\star$ in $Z^\ast$.
    Since each $\mu_{k}\in Z_+^{\ast}$ and $Z_+^{\ast}$ is weak-$\star$ closed, we have $\mu^{\star}\in Z_+^{\ast}$.
    
    Due to \ref{Stability:A2}, $H'(u_k)v\to H'(u^\star)v$ in $Z$ for every fixed $v\in U$.
    Together with the fact that $\mu_{k}\rightharpoonup^{\ast} \mu^{\star}$, we have for every $v \in U$ that 
    \begin{equation*}
        (H'(u_k)^{\ast} \mu_k, v)_U = \langle \mu_{k}, H'(u_k)v \rangle_{Z^{\ast},Z} \to  \langle \mu^{\star}, H'(u^{\star}) v \rangle_{Z^{\ast},Z} = (H'(u^{\star})^{\ast} \mu^{\star}, v)_U. 
    \end{equation*}
    Since $v$ was arbitrary, $H'(u_k)^*\mu_k\rightharpoonup H'(u^\star)^*\mu^\star$ in $U$. Assumption \ref{barrier-to-KKT:A5} makes this sequence relatively compact, and uniqueness of its weak limit therefore gives strong convergence.
    
    Since $\nabla\costr (u_k)+H'(u_k)^{\ast}\mu_{k} \to \nabla\costr (u^{\star})+H'(u^{\star})^{\ast}\mu^{\star} $, $u_k \rightharpoonup u^{\star}$, and the graph of $\partial\costp$ is sequentially closed in the weak-strong topology because $\costp$ is proper, lsc, and convex \cite[Prop. 16.26]{BausComb2017convex}, we can pass to the limit in \cref{lem:outer_iterates_behaviour}\ref{lem:outer_iterates_behaviour:iii} and, thus, \eqref{eq:KKT-limit} follows.
    In particular, $u^\star\in\dom\partial\costp\subseteq\dom\costp$, completing primal feasibility.
    
    Due to \ref{Stability:A1} and \ref{Stability:A2}, we have $H(u_k) \to H(u^{\star})$ and $\mu_{k}\rightharpoonup^{\ast} \mu^{\star}$.
    Furthermore, invoking \cref{thm:outer-convergence},  we have that
    $\langle \mu_{k},\,H(u_k)\rangle_{Z^{\ast},Z} \to 0$.
    Passing to the limit using the strong convergence of $H(u_k)$ and weak-$\star$ convergence of $\mu_{k}$ yields $\langle \mu^{\star}, H(u^{\star})\rangle_{Z^{\ast},Z} = 0$.
    Together with $H(u^{\star})\le_Z 0$ and $\mu^{\star}\in Z_+^{\ast}$ we obtain the complementarity condition, concluding the proof.
\end{proof}

\begin{remark}
    In the finite-dimensional setting, Assumption~\ref{barrier-to-KKT:A5} follows from \ref{G2}, since all norms (and hence the induced topologies) are equivalent on finite-dimensional spaces.
\end{remark}

We now consider the convex case.
Under \ref{C1}--\ref{C3}, the following argument establishes global optimality and stronger convergence conclusions without \ref{barrier-to-KKT:A5}.

\begin{theorem}[Convex case]\label{prop:convex_case}
Let $\nu_k\downarrow0$, and for each $k$, let
$u_k\in\Usad^\circ$ be an $\varepsilon_k$-stationary
point of \eqref{eq:Reduced_barrier_problem}, with
$\varepsilon_k\downarrow0$.
We consider either the logarithmic barrier or the power
barrier with exponent $p>0$.
We assume that:
\begin{enumerate}[label=(C\arabic*)]
    \item\label{C1}
    $\cost\coloneqq\costr+\costp$ is proper, lower
    semicontinuous, convex, and coercive.
    \item\label{C2}
    $H:\Us\to Z$ is order-convex and sequentially
    weak-to-strong continuous.
    \item\label{C3}
    Slater's condition \ref{CQ2} holds.
\end{enumerate}
Then the following statements hold:
\begin{enumerate}[label=(\roman*)]
    \item\label{S1}
    The sequence $\{u_k\}_k$ is bounded, and every weak
    accumulation point solves \eqref{eq:Reduced}.
    Moreover, the objective values converge to the
    minimum value of \eqref{eq:Reduced}.

    \item\label{S2}
    If \eqref{eq:Reduced} has a unique solution $u^\star$,
    then $u_k\rightharpoonup u^\star$ in $U$.

    \item\label{S3}
    If a solution $u^\star$ satisfies the global
    quadratic growth condition
    \begin{equation}\label{QG}
        \cost(u)\ge\cost(u^\star)
        +c\|u-u^\star\|_U^2
        \qquad\text{for every }u\in U_{\rm ad},
    \end{equation}
    for some $c>0$, then $u_k\to u^\star$ strongly in $U$.
\end{enumerate}
\end{theorem}

\begin{proof}
\ref{S1}
We choose a Slater point $\hat u\in\dom\costp$ and
$\sigma>0$ such that
$H(\hat u)\le_Z-\sigma\mathbf1$.
By approximate stationarity, we choose
$\xi_k\in\partial\cost_{\nu_k}(u_k)$ such that $\|\xi_k\|_U\le\varepsilon_k$. For both kernels, convexity of the full barrier objective and approximate stationarity imply that 
\begin{equation}\label{eq:convex_estimate}
    \cost_{\nu_k}(u_k)-\cost_{\nu_k}(v) \le\langle\xi_k,u_k-v\rangle_U\le\|\xi_k\|_U\,\|u_k-v\|_U\le\varepsilon_k\|u_k-v\|_U.
\end{equation}
We first show that, for either kernel, there exist
$C_0,C_1\ge0$ such that
\begin{equation}
\label{eq:convex-barrier-lower}
    \barrierfun_\phi(H(u))
    \ge-C_0-C_1\|u-\hat u\|_U
    \qquad\text{for every }u\in\Usad^\circ.
\end{equation}
For the power kernel, this holds with $C_0=C_1=0$. For the logarithmic kernel, convexity of $H$ gives
\[
    H(u)\ge_Z H(\hat u)+H'(\hat u)(u-\hat u).
\]
Consequently, for $u\in\Usad^\circ$ and $\xi\in\Xi$,
\[
    0<-H(u)(\xi)
    \le c_0+c_1\|u-\hat u\|_U,
\]
where $c_0=\|H(\hat u)\|_Z$ and
$c_1=\|H'(\hat u)\|_{\mathcal L(U,Z)}$.
Using $-\log s\ge-s$ for $s>0$, we obtain
\[
    \barrierfun_{\log}(H(u))
    \ge-\meas(\Xi)
          \bigl(c_0+c_1\|u-\hat u\|_U\bigr).
\]
Thus, for either kernel, there exist constants
$C_0,C_1\ge0$ such that \eqref{eq:convex-barrier-lower} holds. By \ref{C1} and \cite[Prop.~14.16 (iv)]{BausComb2017convex},
there exist $a>0$ and $b\in\R$ such that
$\cost(u)\ge a\|u-\hat u\|_U-b$ for every $u\in U$.
Applying \eqref{eq:convex_estimate} with $v=\hat u$
and using the preceding barrier bound, we obtain
\[
(a-\nu_kC_1-\varepsilon_k)\|u_k-\hat u\|_U\le \cost(\hat u)+b+\nu_k\bigl(\barrierfun_\phi(H(\hat u))+C_0\bigr).
\]
Since $\nu_k,\varepsilon_k\to0$, the coefficient on
the left is at least $a/2$ for all sufficiently large $k$,
while the right-hand side is bounded.
Thus, $\{u_k\}_k$ is bounded.
The barrier bound also gives a constant $C_B\ge0$ such that
\begin{equation}
\label{eq:convex_lower_bound_Bar}
    \barrierfun_\phi(H(u_k))\ge-C_B
    \qquad\text{for every }k.
\end{equation}    
For any feasible $u$ with $\cost(u)<\infty$, we write
$u^t=(1-t)u+t\hat u$, where $t\in(0,1]$.
By convexity, $u^t\in\dom\costp$ and
$H(u^t)\le_Z-t\sigma\mathbf1$.
Hence, the barrier value at $u^t$ is finite for both
kernels, \eqref{eq:convex_lower_bound_Bar} and \eqref{eq:convex_estimate} gives
\[
    \cost(u_k)
    \le\cost(u^t)
       +\nu_k\bigl(\barrierfun_\phi(H(u^t))+C_B\bigr)
       +\varepsilon_k\|u_k-u^t\|_U.
\]
We first send $k\to\infty$ and then $t\downarrow0$.
Convexity and lower semicontinuity give
$\cost(u^t)\to\cost(u)$, so
\[
    \limsup_{k\to\infty}\cost(u_k)\le\cost(u)
    \qquad\text{for every feasible }u
    \text{ with }\cost(u)<\infty.
\]
Boundedness in the Hilbert space $U$ ensures the existence
of a weak accumulation point.
For any such point $u^\star$, we choose a subsequence
$u_{k_j}\rightharpoonup u^\star$.
By \ref{C2}, we have $H(u^\star)\le_Z0$.
Using weak lower semicontinuity, we obtain
\begin{equation}\label{eq:limsup}
    \cost(u^\star)
    \le\liminf_{j\to\infty}\cost(u_{k_j})
    \le\limsup_{k\to\infty}\cost(u_k)
    \le\cost(u)
\end{equation}
for every feasible $u$ with finite objective value.
Taking $u=\hat u$ shows that $\cost(u^\star)<\infty$.
Thus, $u^\star$ is feasible and solves the original problem.
Finally, feasibility of $u_k$ gives
$\cost(u^\star)\le\cost(u_k)$.
Together with \eqref{eq:limsup} for $u=u^\star$, this yields
\begin{equation}
\label{eq:convex_convergence_obj}
    \lim_{k\to\infty}\cost(u_k)=\cost(u^\star).
\end{equation}

\ref{S2}
Suppose that the solution $u^\star$ is unique.
Every subsequence of the bounded sequence $\{u_k\}_k$
has a further weakly convergent subsequence, and
\ref{S1} identifies its limit as $u^\star$.
Consequently, the whole sequence converges weakly
to $u^\star$.

\ref{S3}
By \ref{S1}, \eqref{eq:convex_convergence_obj} holds.
Since the iterates ${u_k} \subset U_0$, \eqref{QG} gives
\[
    0\le c\|u_k-u^\star\|_U^2
    \le\cost(u_k)-\cost(u^\star)
    \to0.
\]
Therefore, $u_k\to u^\star$ strongly in $U$.
\end{proof}

\subsection{Overall complexity bound for barrier methods}\label{sec:overall_complexity}

In this subsection, we estimate the complexity of \cref{alg:IP} when combined with \cref{alg:NMPG} for the barrier subproblem solves.
We use the unified notation from \cref{rem:standing-notation}, so that the discussion covers both \(Z=\R^q\) (componentwise constraints) and \(Z=C(\overline K)\) (pointwise constraints).

\begin{lemma}[Slack and barrier-curvature estimates]\label{lem:slack-curvature}
    Let $u_k\in\Usad^\circ$ be the approximate stationary points returned by the inner solves, and write $s_k\coloneqq-H(u_k)>0$.
    For the logarithmic or power barrier, we denote the density of the multiplier $\mu_k$ in \eqref{eq:lag_Mul} by $\rho_k$, so that  $\rho_k=\nu_k\phi'(-s_k)$. In finite dimensions, $\rho_k$ denotes the components of the multiplier vector.
    We assume that, for some $M>0$ independent of $k$, it holds  $0\le\rho_k(\xi)\le M$ for $\meas$-almost every $\xi\in\Xi$.
    Then there exist constants $c,C>0$, independent of $k$, such that
    \[
        s_k\ge_Z c\nu_k^\beta\mathbf1,
        \qquad
        0 \le_Z\nu_k\phi''(-s_k) \le_Z
         C\nu_k^{-\beta}\mathbf1 ,
        \qquad
        \beta=
        \begin{cases}
            1,&\phi(t)=-\log(-t),\\
            1/(p+1),&\phi(t)=(-t)^{-p},\quad p>0.
        \end{cases}
    \]
    The slack inequality holds componentwise for $Z=\R^q$.
    For $Z=C(\overline K)$, it holds everywhere by continuity of $s_k$ and full support of $\meas$.
\end{lemma}
\begin{proof}
The multiplier identity and the assumed bound give
\[
    0<\rho_k(\xi)
    =\nu_k\phi'(-s_k(\xi))\le M
    \qquad\text{for $\meas$-almost every }\xi\in\Xi.
\]
For the logarithmic kernel, we consequently obtain
\[
    s_k(\xi)\ge\frac{\nu_k}{M},
    \qquad
    \nu_k\phi''(-s_k(\xi))
    =\frac{\nu_k}{s_k(\xi)^2}
    \le M^2\nu_k^{-1}.
\]
For the power kernel, the identity
$\rho_k(\xi)=p\nu_k s_k(\xi)^{-(p+1)}$ gives
\[
    s_k(\xi)\ge
    \left(\frac{p}{M}\right)^{1/(p+1)}
    \nu_k^{1/(p+1)},
\]
and hence
\[
    \nu_k\phi''(-s_k(\xi))
    \le p(p+1)
       \left(\frac{M}{p}\right)^{(p+2)/(p+1)}
       \nu_k^{-1/(p+1)}.
\]
In $\R^q$, these inequalities hold for every component.
In $C(\overline K)$, continuity and full support of
$\meas$ extend them to every point of $\overline K$.
Thus, in both settings, they yield the claimed
inequalities with respect to the order $\le_Z$
\end{proof}

In finite dimensions, the density bound in \cref{lem:slack-curvature} follows from uniform boundedness of the multipliers. In the continuous-function setting, it is an additional requirement.
The lemma motivates the barrier-parameter dependence assumed below, but concerns the returned approximate stationary points rather than all points visited during the inner solves.

For the complexity statement, we assume
\begin{assumption}\label{ass:uniform-inner}
    For every subproblem \eqref{eq:Reduced_barrier_problem_2}, \cref{ass:inner-hypotheses} holds with
    \begin{equation}\label{eq:uniform-inner-family}
    \begin{gathered}
    \overline\alpha_{\nu_k}\le c_A(1+\Lsmooth(\nu_k)),\quad
    \Lsmooth(\nu_k)\le C_L\nu_k^{-\beta},
    \quad
    \beta=\begin{cases}1,&\text{logarithmic kernel},\\
    1/(p+1),&\text{power kernel},\ p>0.
    \end{cases}
    \end{gathered}
    \end{equation}
\end{assumption}
The constants $c_A,C_L,\ssizelb$ are independent of $k$.
A sufficient route to the last bound is a uniformly bounded inner region, uniform non-barrier constants in \eqref{eq:Lipschitz-smooth-nu}, and the order bound
$-H(u)\ge_Z c\nu_k^\beta\mathbf1$ throughout that region.
The second bound then follows from \cref{lem:central-path-Lip-onesided}; it is a worst-case upper bound and does not assert that the actual curvature diverges.

The first bound in \eqref{eq:uniform-inner-family} is motivated by the backtracking estimate in \cite[Lem.~4(ii)]{azmi2025nonmonotone}, which applies under the corresponding domain and Lipschitz continuity assumptions.
Since these assumptions are not automatic for the barrier subproblems, we retain \eqref{eq:uniform-inner-family} as an explicit hypothesis of the complexity result below.

We now use the warm starts in \cref{alg:IP} to bound the initial objective gaps independently of $k$.
We retain the lower bound on the original objective and use a lower bound on the barrier only along the outer iterates:
\begin{equation}\label{eq:objective-lower-bounds}
J_*:=\inf_{\Usad^\circ}\cost>-\infty,
\qquad \barrierfun(H(u_k))\ge-C_B\quad(k\ge0),
\qquad C_B\ge0.
\end{equation}
For the power kernel we can take $C_B=0$.
For the logarithmic kernel, the second bound follows from \ref{Stability:A1}--\ref{Stability:A2} and is therefore not an additional condition when these assumptions hold.
Indeed, $H(u_k)\to H(u^\star)$ in $Z$, so there is $M\ge1$, including the initial point $u_0$, such that $0<-H(u_k)(\xi)\le M$ for all $k\ge0$ and $\xi\in\Xi$.
Consequently, we can write
\[
\barrierfun(H(u_k))
=-\int_\Xi\log\bigl(-H(u_k)(\xi)\bigr)\,\mathrm d\meas(\xi)
\ge-\meas(\Xi)\log M,
\]
and choose $C_B=\meas(\Xi)\log M$.
\cref{alg:NMPG} keeps every inner objective value below its initial value, since its reference value is a maximum of preceding inner values.
In particular, for $k\ge2$, we obtain
\begin{multline*}
    \cost_{\nu_k}(u_{k-1})
    =
    \cost_{\nu_{k-1}}(u_{k-1}) +(\nu_k-\nu_{k-1})\barrierfun(H(u_{k-1}))
    \\
    \le
    \cost_{\nu_{k-1}}(u_{k-2}) +C_B(\nu_{k-1}-\nu_k)
    \leq
    \ldots
    \leq
    \cost_{\nu_1}(u_0)+C_B(\nu_1-\nu_k).
\end{multline*}
To bound the infima uniformly, we recall that $\cost_{\nu_1}^*>-\infty$ is already part of \cref{ass:inner-hypotheses} for the first subproblem.
For $0<\nu\le\nu_1$ and $u\in\Usad^\circ$, we have
\[
\cost_\nu(u)
=\left(1-\frac{\nu}{\nu_1}\right)\cost(u)
 +\frac{\nu}{\nu_1}\cost_{\nu_1}(u).
\]
Taking the infimum and using \eqref{eq:objective-lower-bounds}, we therefore obtain
\[
\cost_\nu^*
\ge\left(1-\frac{\nu}{\nu_1}\right)J_*
 +\frac{\nu}{\nu_1}\cost_{\nu_1}^*
\ge\min\{J_*,\cost_{\nu_1}^*\}>-\infty.
\]
Consequently,
\begin{equation}\label{eq:uniform-initial-gaps}
0\le\cost_{\nu_k}(u_{k-1})-\cost_{\nu_k}^*
\le\cost_{\nu_1}(u_0)-\min\{J_*,\cost_{\nu_1}^*\}+C_B\nu_1 \eqqcolon C_J <\infty,
\end{equation}
uniformly in $k$.

If, in addition, $\cost$ is $\kappa$-strongly convex and $H$ is order-convex, then $\cost_{\nu_k}$ is $\kappa$-strongly convex on its convex strict domain for every $k$.
Indeed, the barrier kernels are convex and increasing, so $\barrierfun\circ H$ is convex and the same modulus $\kappa>0$ is retained.
This provides the common strong-convexity modulus used in \cref{cor:overall_complexity_unified}\ref{cor:overall_complexity_unified:stronglyconvex} below.

We count the total number of proximal-gradient iterations to reach an $\varepsilon$-KKT point:
\begin{equation}\label{eq:overall_iterations}
\Nitertot(\varepsilon)=\sum_{k=1}^{\Niterout(\varepsilon)}
      \Niterin(\nu_k,\varepsilon_k).
\end{equation}

\begin{corollary}[Overall iteration complexity]\label{cor:overall_complexity_unified}
    Suppose that Assumptions~\ref{ass:inner-hypotheses} and~\ref{ass:uniform-inner} and the objective-gap condition \eqref{eq:objective-lower-bounds} hold for the inner solves in \cref{alg:IP}.
    For the power barrier, assume in addition \ref{Stability:A1}--\ref{unifrom_bounded:A4}, so that the outer-iteration bound \eqref{eq:complexity_outer_iterations:power_barrier} applies.
    For the logarithmic barrier, the outer count is given by \eqref{eq:complexity_outer_iterations:log_barrier}.
    Let $N\coloneqq \Niterout(\varepsilon)$.
    Then, for $\nu_k=\nu_1\theta_\nu^{k-1}$, $\varepsilon_k=\varepsilon_1\theta_\varepsilon^{k-1}$ and fixed $\theta_\nu,\theta_\varepsilon\in(0,1)$, the following bounds hold.
    \begin{enumerate}[label=(\roman*)]
        \item \label{cor:overall_complexity_unified:general}
        In the general case,
        \begin{equation}\label{eq:total-general}
        \Nitertot(\varepsilon)
        =
        O\left(
        (\mmax+1)\nu_1^{-2\beta}\varepsilon_1^{-2} \frac{\Theta^N-1}{\Theta-1}\right),
        \qquad
        \Theta=\theta_\nu^{-2\beta}\theta_\varepsilon^{-2}>1.
        \end{equation}
        \item \label{cor:overall_complexity_unified:stronglyconvex}
        If $\cost_{\nu_k}$ is $\kappa$-strongly convex with the same $\kappa>0$ for every $k$, then
        \begin{equation}\label{eq:total-strong}
        \Nitertot(\varepsilon)
        =
        O\left((\mmax+1)\nu_1^{-2\beta} N\theta_\nu^{-2\beta(N-1)}\right).
        \end{equation}
    \end{enumerate}
    The implied constants are independent of $\varepsilon$ and $k$, but may depend on the initial data, the fixed geometric schedules, the constants in \eqref{eq:uniform-inner-family}, and $\kappa$ in part~(ii).
    Since $\Niterout(\varepsilon)=O(\log(1/\varepsilon))$ for the logarithmic barrier and, under \ref{Stability:A1}--\ref{unifrom_bounded:A4}, for the power barrier, \eqref{eq:total-general}--\eqref{eq:total-strong} can equivalently be written as polynomial bounds in $1/\varepsilon$.
    The displayed geometric form is retained because it shows the dependence on the barrier and tolerance schedules explicitly.
\end{corollary}
\begin{proof}
By \eqref{eq:K-inner}, \eqref{eq:uniform-inner-family} and \eqref{eq:uniform-initial-gaps}, we have
\[
c_{\nu_k}=O\bigl(1+\Lsmooth(\nu_k)^2\bigr)=O(\nu_k^{-2\beta}),
\qquad C_{\nu_k}(u_{k-1})^2\le C_J c_{\nu_k}.
\]
For part~\ref{cor:overall_complexity_unified:general}, \eqref{eq:inner_complexity} therefore gives
\[
\Niterin(\nu_k,\varepsilon_k)
=O\!\left((\mmax+1)\nu_k^{-2\beta}\varepsilon_k^{-2}\right).
\]
Summing over $k=1,\ldots,N$ yields the geometric series in \eqref{eq:total-general}.

For part~\ref{cor:overall_complexity_unified:stronglyconvex}, using $\log(1+t)\ge t/(1+t)$ for $t>0$, we obtain
\[
\frac{1}{|\log\sigma_{\nu_k}|}
=\frac{1}{\log(1+2\kappa/c_{\nu_k})}
\le 1+\frac{c_{\nu_k}}{2\kappa}
=O(\nu_k^{-2\beta}).
\]
Moreover, the preceding bound on $C_{\nu_k}(u_{k-1})$ gives
\[
\log_+\!\left(\frac{C_{\nu_k}(u_{k-1})}{\varepsilon_k}\right)
\le C+\beta\log_+(1/\nu_k)+\log_+(1/\varepsilon_k),
\]
where $\log_+(x) =\max\{ \log (x),0 \}$ and  $C$ is independent of $k$.
Consequently, \eqref{eq:inner_complexity_strong_convexity} yields
\[
\Niterin(\nu_k,\varepsilon_k)
=O\!\left((\mmax+1)\nu_k^{-2\beta}
[1+\log_+(1/\nu_k)+\log_+(1/\varepsilon_k)]\right).
\]
Along the fixed geometric schedules, the bracket is
$O(1+\log_+(1/\varepsilon_k))=O(k)$.
Thus the dependence of $C_{\nu_k}(u_{k-1})$ on $\nu_k$ does not change the resulting order, although this constant need not be uniformly bounded.
Since $\sum_{j=0}^{N-1}(j+1)\rho^j=O(N\rho^{N-1})$ for fixed $\rho>1$, \eqref{eq:total-strong} follows.
\end{proof}

With the memory convention \eqref{eq:ls_memory}, $\mmax=0$ is the monotone proximal-gradient method.

\begin{remark}[Log-like kernel]\label{rem:loglike}
    Under the same assumptions used for the logarithmic kernel, the estimates below show that the proofs of \cref{thm:outer-convergence} and \cref{cor:overall_complexity_unified} extend to the log-like kernel \(\phi_{\rm ll}(t)=\log((t-1)/t)\), \(t<0\).
    Indeed, setting \(s=-t>0\), we have
    \(\phi_{\rm ll}'(-s)=1/[s(1+s)]\le 1/s\) and
    \(\phi_{\rm ll}''(-s)=(1+2s)/[s^2(1+s)^2]\le 1/s^2\).
    Moreover,
    \(r_k=\nu_k\int_\Xi (1+s_k)^{-1}\,\mathrm d\meas
    \le \nu_k\meas(\Xi)\).
    Since this kernel is nonnegative, convex, and increasing, the multiplier analysis uses the same complementarity upper bound as for the logarithmic kernel.
\end{remark}

\section{Verification for state-constrained semilinear elliptic problems}\label{sec:pde_verification}

Let $\Omega\subset\R^d$, $d\leq 3$, be a bounded domain with
$C^{1,1}$ boundary or, if \(d=2\), a bounded convex polygon, and let
$\mathbf n$ denote the outer unit normal.
For a given control $u\in U$, we consider the affine semilinear
Neumann problem
\begin{equation}\label{eq:state_affine_Bf_Neu}
    Ay+N(\cdot,y)=Bu+f
    \quad\text{in }\Omega,
    \qquad
    \left.\partial_{\mathbf n} y\right|_{\partial\Omega}=0,
\end{equation}
where $U$ is a Hilbert space continuously embedded into
$L^2(\Omega)$, the operator
$B\in\mathcal L(U,L^2(\Omega))$ is bounded and linear, and
$f\in L^2(\Omega)$ is fixed. The linear elliptic operator $A$ is
defined by
$Ay\coloneqq-\Delta y+cy$,
with the coefficient $c$ specified in the proposition below. In this proposition, we verify the properties of the control-to-state map $\Sctl:u\mapsto y$ associated with the semilinear Neumann problem \eqref{eq:state_affine_Bf_Neu} that are required in the abstract theory: existence and uniqueness of the state, Lipschitz stability, Fr\'echet differentiability, weak-to-strong continuity, and collective compactness of the adjoints of its derivatives.
The proof is provided in \hyperref[sec:omitted_proofs]{Appendix~\ref{sec:omitted_proofs}}.

\begin{proposition}\label{prop:unified_compact_stability_affine_Bf_Neumann}
Assume the following.
\begin{enumerate}[label=\textnormal{(E)},ref=\textnormal{(E)}]
    \item\label{ass:E}
    The coefficient $c\in L^\infty(\Omega)$ satisfies $c\ge c_0>0$ a.e., and $f\in L^2(\Omega)$.
\end{enumerate}
\begin{enumerate}[label=\textnormal{(H\arabic*)},ref=\textnormal{(H\arabic*)}]
    \item\label{ass:H1}
    $N(\cdot,s)$ is measurable for every $s\in\R$, and $s\mapsto N(x,s)$ is $C^1$ for a.e. $x\in\Omega$.
    We write $N_y\coloneqq\partial_sN$.

    \item\label{ass:H2}
    For every $M>0$ there exist constants $L_M,\ell_M<\infty$ such that, for a.e. $x\in\Omega$ and all $|s|,|t|\le M$, we have 
    \[
       |N_y(x,s)|\le L_M,
       \qquad
       |N_y(x,s)-N_y(x,t)|\le\ell_M|s-t|.
    \]

    \item\label{ass:H3}
    There exists $\kappa\in[0,c_0)$ such that, for a.e. $x\in\Omega$ and all $s,t\in\R$, we have 
    \[
       \bigl(N(x,s)-N(x,t)\bigr)(s-t)
       \ge-\kappa|s-t|^2.
    \]

    \item\label{ass:H4}
    There exist $\eta\in L^2(\Omega)$, $b_1,b_2\ge0$, and an exponent $1\le p<\infty$ if $d\le2$ (respectively $1\le p\le3$ if $d=3$) such that it holds
    \[
       |N(x,s)|\le\eta(x)+b_1|s|+b_2|s|^p
       \qquad\text{for a.e. }x\in\Omega\text{ and every }s\in\R.
    \]
\end{enumerate}
Then $\Sctl:U\to H^1(\Omega)\cap C(\overline\Omega)$ has the following properties.
\begin{enumerate}[label={(\roman*)}]

\item\label{stAffE:I}
For every $u\in U$, problem \eqref{eq:state_affine_Bf_Neu} has a
unique weak solution
\[
    y=\Sctl(u)\in H^2(\Omega)
    \subset H^1(\Omega)\cap C(\overline\Omega).
\]
Moreover, there exists a constant $C_S>0$, independent of $u$, such
that
\begin{equation}\label{eq:esimate_soltution_semilinear}
    \|\Sctl(u)\|_{H^1(\Omega)}
    +\|\Sctl(u)\|_{C(\overline\Omega)}
    \leq
    C_S\bigl(1+\|u\|_U\bigr)^q,
    \qquad q=\max\{1,p\}.
\end{equation}

   \item\label{stAffE:II}
   For every bounded set $U_b\subset U$, there exists
   $L_S(U_b)>0$ such that
\[
    \|\Sctl(u_1)-\Sctl(u_2)\|_{H^1(\Omega)}
    +
    \|\Sctl(u_1)-\Sctl(u_2)\|_{C(\overline\Omega)}
    \leq
    L_S(U_b)\|u_1-u_2\|_U
\]
for all $u_1,u_2\in U_b$.

    \item\label{stAffE:III}
    The map $\Sctl:U\to H^1(\Omega)\cap C(\overline\Omega)$ is Fr\'echet differentiable.
    For $u,h\in U$, $z=\Sctl'(u)h$ is the unique weak solution of
    \begin{equation}\label{eq:lin_neu}
    \begin{aligned}
       Az+N_y(\cdot,\Sctl(u))z&=Bh &&\text{in }\Omega, \qquad
    \left.\partial_{\mathbf n} z\right|_{\partial\Omega}=0.
    \end{aligned}
    \end{equation}

    \item\label{stAffE:IV}
    For every bounded set $U_b\subset U$, there exist $M_{S'}(U_b),L_{S'}(U_b)>0$ such that
    \[
       \|\Sctl'(u)\|_{\mathcal L(U,H^1(\Omega)\cap C(\overline\Omega))}
       \le M_{S'}(U_b)
       \qquad\forall u\in U_b,
    \]
    and
    \[
       \|\Sctl'(u_1)-\Sctl'(u_2)\|_{\mathcal L(U,H^1(\Omega)\cap C(\overline\Omega))}
       \le L_{S'}(U_b)\|u_1-u_2\|_U
       \qquad\forall u_1,u_2\in U_b.
    \]

    \item\label{stAffE:V}
    If $u_k\rightharpoonup u$ in $U$, then we have 
    \[
       \Sctl(u_k)\to\Sctl(u)
       \quad\text{strongly in }H^1(\Omega)\cap C(\overline\Omega),
    \]
    and, for every $h\in U$, it holds
    \[
       \Sctl'(u_k)h\to\Sctl'(u)h
       \quad\text{strongly in }H^1(\Omega)\cap C(\overline\Omega).
    \]

    \item\label{stAffE:VI}
    If $u_k\rightharpoonup u$ in $U$, then, for every $\mu\in\mathcal M(\overline\Omega)$, we have 
    \begin{equation}\label{eq:adjoint_pointwise_conv}
       \Sctl'(u_k)^*\mu\to\Sctl'(u)^*\mu
       \quad\text{strongly in }U.
    \end{equation}
    Moreover, if $(u_k)$ is bounded in $U$ and $\barrierfun\subset\mathcal M(\overline\Omega)$ is bounded in total variation, then
    \[
       \mathcal K(\barrierfun)
       \coloneqq
       \{\,\Sctl'(u_k)^*\mu:\ k\in\N,\ \mu\in\barrierfun\,\}
    \]
    is relatively compact in $U$.
\end{enumerate}
\end{proposition}

\begin{remark}\label{rem:cubic_class}
    Assume that
    \[
        N(x,s)=a(x)|s|^{p-1}s+b(x)s+g(x),
    \]
    where \(p\in\N\), \(a,b\in L^\infty(\Omega)\), and
    \(g\in L^2(\Omega)\); the associated derivative is
    \(N_y(x,s)=p\,a(x)|s|^{p-1}+b(x)\). If, in addition, $a(x)\geq0$ for a.e. $x\in\Omega$ and
    $\kappa\coloneqq\|b^-\|_{L^\infty(\Omega)}<c_0$, where
    $b^-(x)\coloneqq\max\{-b(x),0\}$ denotes the negative part of $b$,
    then \ref{ass:H1}--\ref{ass:H3} are satisfied. The growth condition
    \ref{ass:H4} holds whenever
    \(H^1(\Omega)\hookrightarrow L^{2p}(\Omega)\). Consequently,
    \(p\in\{1,2,3\}\) is admissible for \(d=3\), whereas every finite
    \(p\in\N\) is admissible for \(d\leq2\).
\end{remark}

\begin{remark}
    Assume that the homogeneous Neumann condition in
    \eqref{eq:state_affine_Bf_Neu} is replaced by the homogeneous
    Dirichlet condition
    \(\left.y\right|_{\partial\Omega}=0\). Then all conclusions of
    \cref{prop:unified_compact_stability_affine_Bf_Neumann} remain valid,
    with \(H_0^1(\Omega)\) as the corresponding energy space. In this setting, the condition \(c\geq c_0>0\) may be relaxed to
    \(c\geq c_0\geq0\), and it is sufficient to assume
    \(\kappa\leq c_0\). Indeed, with $w=y_1-y_2\in H_0^1(\Omega)$, we have
    \[
     \langle T(y_1)-T(y_2),w\rangle
     \ge
     \|\nabla w\|_2^2+(c_0-\kappa)\|w\|_2^2
     \ge
     \|\nabla w\|_2^2.
    \]
    By Poincaré's inequality, the right-hand side is coercive with respect
    to the full \(H^1(\Omega)\)-norm. The preceding arguments therefore
    remain valid using the corresponding global Dirichlet
    \(H^2\)-regularity estimates.
\end{remark}

With \cref{prop:unified_compact_stability_affine_Bf_Neumann} in hand, we verify the following regularity and compactness assumptions for the tracking-only reduced functional in \eqref{eq:reduced_definitions}

\begin{lemma}\label{lem:Applicability}
    Let $Z \coloneqq C(\overline{\Omega})$ and let $U$, $\Hsobs$ be Hilbert spaces with $U\subseteq L^2(\Omega)$ and $Z\subseteq \Hsobs$.
    Assume the hypotheses of \cref{prop:unified_compact_stability_affine_Bf_Neumann}.
    Fix $y_d\in \Hsobs$, $\bar y\in Z$, and bounded linear operators $C \colon Z\to \Hsobs$ and $S \colon C(\overline{\Omega})\to C(\overline{\Omega})$.
    Let $\Sctl:U\to Z$ be the control-to-state map associated with \eqref{eq:state_affine_Bf_Neu} and define $\costr$ and $H$ as in \eqref{eq:reduced_definitions}.
    Then Assumptions~\ref{G2}, \ref{E2}, \ref{G3}, \ref{Stability:A2}, \ref{Stability:A3}, and \ref{barrier-to-KKT:A5} hold.
\end{lemma}
\begin{proof}
    All claims follow from \cref{prop:unified_compact_stability_affine_Bf_Neumann}.
    In particular, \ref{G2} is a consequence of \ref{stAffE:III}, while \ref{G3} (local Lipschitz continuity of $H'$ and uniform boundedness of $\|H'(u)\|_{\mathcal L(U,Z)}$ on bounded sets) follows from \ref{stAffE:IV}.
    By the chain rule, $\costr$ is Fr\'echet differentiable with
    \[
    \nabla \costr(u)=\Sctl'(u)^*\,r(u),
    \qquad
    r(u)\coloneqq C^*(C\Sctl(u)-y_d).
    \]
    Hence, for $u_1,u_2\in U$, we can write
    \[
    \nabla \costr(u_1)-\nabla \costr(u_2)
    =\bigl(\Sctl'(u_1)^*-\Sctl'(u_2)^*\bigr)r(u_1)
    +\Sctl'(u_2)^*\bigl(r(u_1)-r(u_2)\bigr).
    \]
    Using \ref{stAffE:II} and \ref{stAffE:IV}, we obtain on bounded sets $U_b\subset U$ the estimate
    \[
    \|\nabla \costr(u_1)-\nabla \costr(u_2)\|_{U}
    \le \Bigl(L_{S'}(U_b)\,R(U_b)+M_{S'}(U_b)\,\|C^{*}\|_{\mathcal{L}(\Hsobs,Z^*)}\|C\|_{\mathcal{L}(Z,\Hsobs)}\,L_S(U_b)\Bigr)\,\|u_1-u_2\|_{\Us},
    \]
    where $R(U_b) \coloneqq  \|C^{*}\|_{\mathcal{L}(\Hsobs,Z^*)}\left(C_S\|C\|_{\mathcal{L}(Z,\Hsobs)}\bigl(1+\sup_{u \in \Us_b}\|u\|_U\bigr)^q+\|y_d\|_{\Hsobs}\right)$. This yields the required local Lipschitz continuity of $\nabla\costr$.
    
    Assumptions \ref{E2} and \ref{Stability:A2} follow from \ref{stAffE:V}.
    To verify \ref{Stability:A3}, let $u_k\rightharpoonup u$ in $U$.
    By \ref{stAffE:V}, $\Sctl(u_k)\to \Sctl(u)$ strongly in $Z$, and thus, by boundedness of $C^*:\Hsobs\to Z^*$, we infer that
    \[
    r(u_k)=C^*(C\Sctl(u_k)-y_d)\to r(u)\quad\text{strongly in }Z^*,
    \qquad
    \sup_k\|r(u_k)\|_{Z^*}<\infty.
    \]
    Moreover, \ref{stAffE:IV} yields $\sup_k\|\Sctl'(u_k)^*\|_{\mathcal L(Z^*,U^*)}<\infty$.
    Therefore, 
    \[
    \nabla\costr(u_k)-\nabla\costr(u)
    =\Sctl'(u_k)^*(r(u_k)-r(u))
    +\bigl(\Sctl'(u_k)^*-\Sctl'(u)^*\bigr)r(u)\to 0 \quad \text{in }U^*,
    \]
    where the first term converges by the strong convergence of $r(u_k)$ and the uniform bound, and the second by \eqref{eq:adjoint_pointwise_conv}.
    Hence \ref{Stability:A3} holds.
    
    It remains to show \ref{barrier-to-KKT:A5}. Since $H(u)=S\,\Sctl(u)-\bar y$ with $S$ linear, we have
    \[
    H'(u_k)^*=\Sctl'(u_k)^*\,S^* \quad\text{in }\mathcal L(Z^*,U),
    \]
    and $S^*:Z^*\to Z^*$ is bounded.
    Thus, for any set $B\subset Z^*=\mathcal M(\overline{\Omega})$ bounded in total variation, also $S^*(B)$ is bounded.
    By \ref{stAffE:VI},
    \[
    \bigl\{\, \Sctl'(u_k)^* s:\ k\in\N,\ s \in S^*(B)\,\bigr\}
    \]
    is relatively compact in $\Us$, which is precisely \ref{barrier-to-KKT:A5}.
\end{proof}

\section{Numerical experiments}\label{sec:numerics}

In this section, we illustrate the behavior of \cref{alg:IP} on two examples: a finite-dimensional data science task and a PDE-constrained control problem.
We compare the performance with the two theoretically analyzed kernels, namely the logarithmic $\phi(t)\coloneqq -\ln(-t)$ (\stringlog) and inverse $\phi(t)\coloneqq -1/t$ (\stringinverse),
and with the log-like kernel $\phi(t) \coloneqq \ln((t-1)/t)$ (\stringloglike) considered as an additional numerical variant.
The log-like barrier behaves like the logarithmic one near the boundary while remaining everywhere nonnegative like the inverse barrier; see \cite{demarchi2024interior} and \cite[Table 1]{demarchi2025penalty}.
For the infinite-dimensional example, the effect of different discretization meshes is also examined.
In \cref{alg:NMPG} we set parameters $\eta=2$ and $\delta=10^{-3}$.
The stepsize estimates follow the alternating Barzilai--Borwein update rule ABBa, as given in \cite[\S 2]{azmi2025nonmonotone}, and are then clipped to $[10^{-6},10^6]$.

\subsection{Sparse dictionary learning with cone-ordered constraints}\label{ex:dict-learning}

A finite-dimensional machine-learning problem covered by \eqref{eq:Reduced} is sparse dictionary learning with nonlinear side constraints \cite{elad2010sparse}.
Let $A = [a_1,\dots,a_N]\in\R^{m\times N}$ be a data matrix whose columns are training samples (e.g., vectorized grayscale images or image patches).
Given a number of atoms $\ell\in\N$, we seek a dictionary \(D\in\R^{m\times \ell}\) and a coefficient matrix \(C\in\R^{\ell\times N}\) such that \(A\approx D C\). We consider
\begin{equation}\label{eq:dict-learning-example}
    \minimize_{D\in\R^{m\times \ell},\,C\in\R^{\ell\times N}}{}\quad
    \frac{1}{2} \|A-DC\|_F^2 + \costp(D,C)
    \quad\stt\quad
    H(D,C) \leq_{\R_+^q} 0 ,
\end{equation}
where the inequality is understood componentwise.
A common choice for the regularizer $\costp$ is
\[
    \costp(D,C)
    \coloneqq
    \sum_{j=1}^\ell \indicator_{\unitsphere}(d_{\cdot,j}) + \lambda \sum_{j=1}^\ell \|c_{j,\cdot} \|_1 ,
\]
with parameter \(\lambda>0\).
The $\ell_1$-norm penalty $\|\cdot\|_1$ applied to the rows $c_{j,\cdot}$ of $C$ promotes sparsity of the coefficient matrix, while the indicator $\indicator_{\unitsphere}$ of the unit ball
$\unitsphere \coloneqq \{ v\in\R^m \,|\, \|v\| \leq 1 \}$ enforces a normalization constraint on the columns $d_{\cdot,j}$ of $D$ \cite[\S 6.2]{themelis2018forward}.
Then, the nonlinear least-squares term \(\costr\) in \eqref{eq:dict-learning-example} is smooth and nonconvex (because of the bilinear term \(D C\)), whereas \(\costp\) is proper, lsc, convex, and has a computable proximal mapping.

We require the atoms in $D$ to be sufficiently diversified, relative to a prescribed coherence level \(\mucoh>0\), by means of a cone-valued constraint map obtained by stacking coherence constraints of the form $| \langle d_{\cdot,i},d_{\cdot,j}\rangle | \leq \mucoh$, $1\le i<j\le \ell$.
Then, the smooth constraint function $H$ is defined by
\[
H(D,C)\coloneqq \bigl(\{ \langle d_{\cdot,i},d_{\cdot,j}\rangle^2-\mucoh^2 \}_{1\le i<j\le \ell}\bigr)\in\R^q,
\qquad
q= \ell (\ell-1)/2.
\]
Thus, \eqref{eq:dict-learning-example} is a concrete nonconvex nonsmooth optimization problem with cone-ordered constraints in the finite-dimensional setting \(Z=\R^q\).

\paragraph*{Setup}
We sampled the \emph{Database of Faces} to generate a data matrix $A$,
randomly selecting two images for each subject, for a total of $N=80$ images.
These images were subsampled to reduce their size from $10304$ ($92\times 112$) to $m = 1080$ ($30\times 36$) pixels each.
We set the parameters $\ell=10$, $\mucoh = \nicefrac{1}{2}$, $\lambda = \nicefrac{1}{10}$ and obtained an instance of \eqref{eq:dict-learning-example} with $n=\ell(m+N)=11600$ variables and $q=45$ constraints.

\cref{alg:IP} is given the tolerance $\varepsilon \coloneqq 1$ and a strictly feasible initial guess, with $D=0_{m\times \ell}$ and $C\in\R^{\ell\times N}$ randomly sampled from a normal distribution.
The barrier parameter is initialized at $\nu_1\coloneqq 10^2$ with factor $\theta_\nu \coloneqq \nicefrac{9}{10}$.
The inner tolerance is set to $\varepsilon_k \coloneqq \max\{ \varepsilon, \nu_k\}$ for all $k\in\N$.

\paragraph*{Results}
The archived experiments report successful termination with each barrier function.
\cref{fig:faces_results} reports the total objective, stationarity measure and complementarity along the iterations.
\cref{fig:faces_dictionary} shows the data images along with a reconstruction and the corresponding dictionary.
In terms of efficiency, \cref{alg:IP} with logarithmic barrier requires 81 outer iterations and 48822 total inner iterations (wall-clock runtime of 91 seconds),
the inverse barrier 103 outer and 101408 inner iterations (178 seconds),
and the log-like barrier 80 outer and 45610 inner iterations (83 seconds).

On this instance, the logarithmic and log-like barriers require fewer inner iterations than the inverse barrier under the reported stopping criteria.
This observation is consistent with the corresponding behavior observed in \cite[\S 4.2]{demarchi2025penalty}, but it should be interpreted as an empirical observation for the tested instance rather than as a general performance guarantee.

\begin{figure}[tbh]
    \centering%
    \includegraphics[scale=0.75]{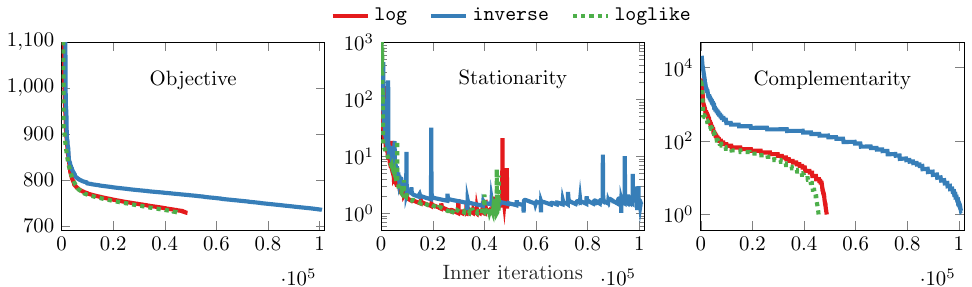}%
    \caption{Dictionary learning on faces: results in terms of objective (left), stationarity (middle), and complementarity (right) for different barrier functions (logarithmic, inverse, log-like).}%
    \label{fig:faces_results}%
\end{figure}

\begin{figure}[tbh]
    \centering%
    \includegraphics[height=0.18\textheight,width=0.47\linewidth]{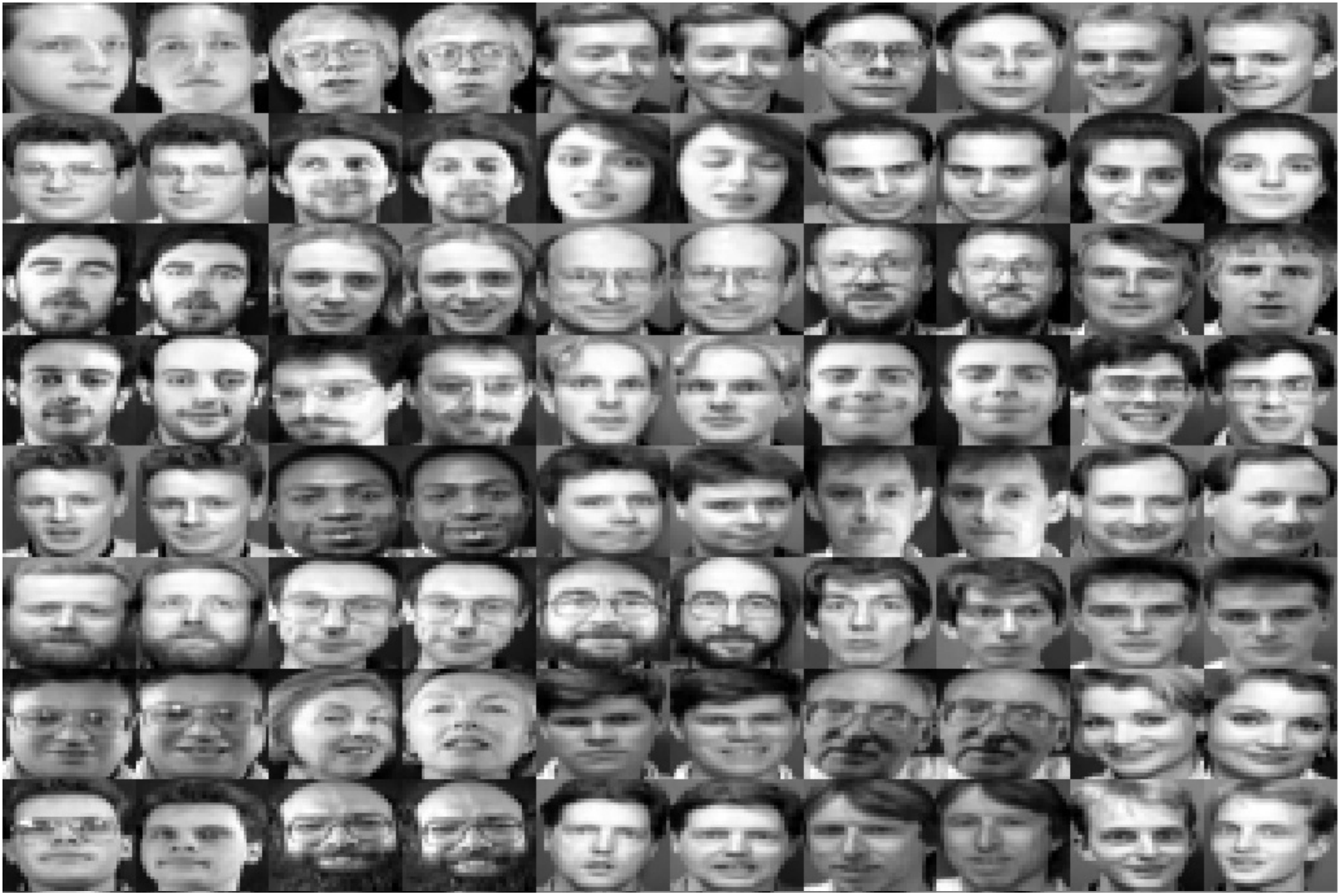}%
    \includegraphics[height=0.18\textheight,width=0.47\linewidth]{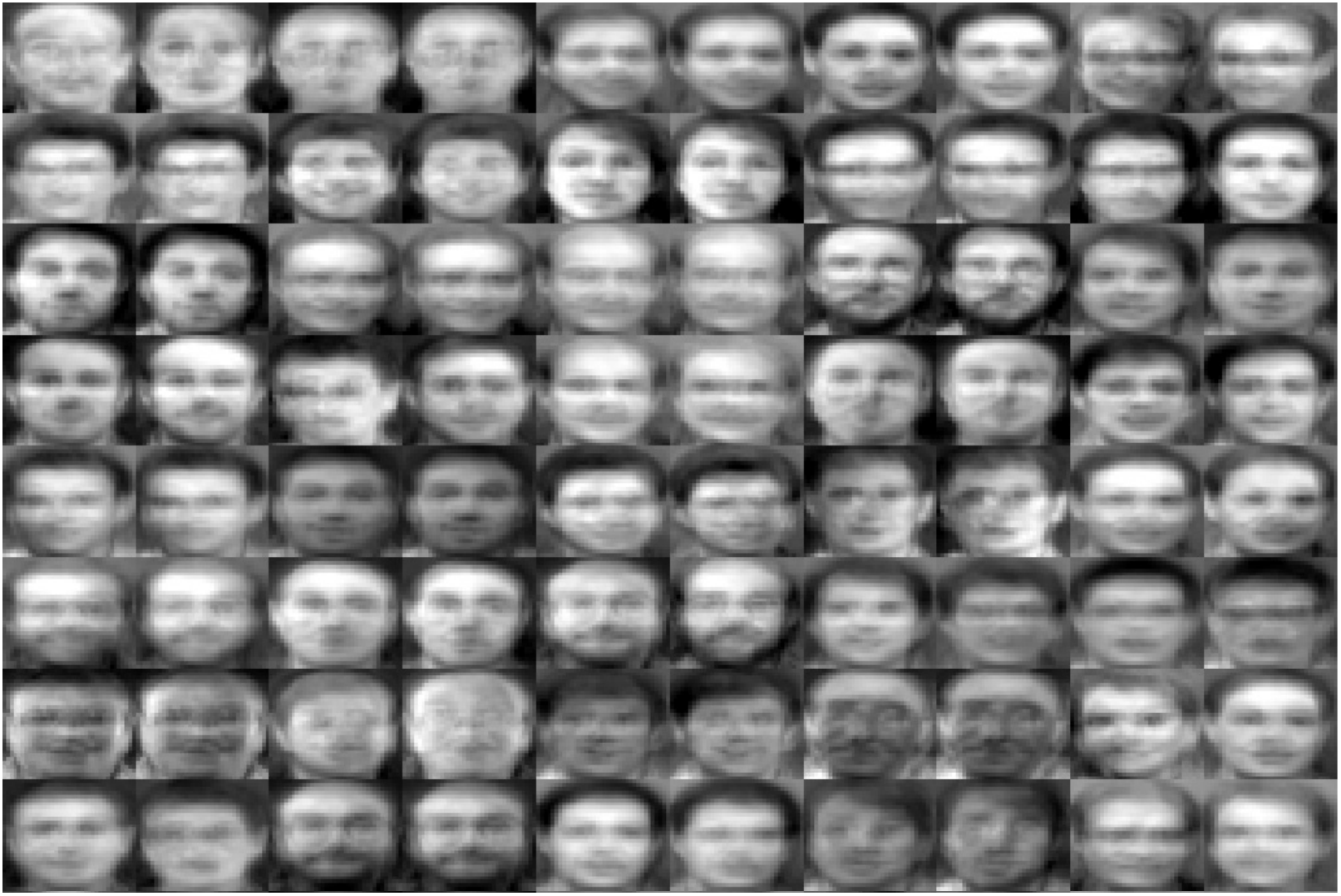}\\
    \includegraphics[width=0.47\linewidth]{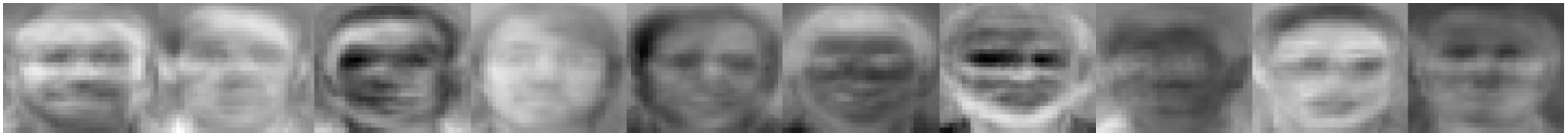}%
    \caption{Dictionary learning on faces: original images (top left) and reconstructed images (top right) with dictionary atoms (bottom) found using the logarithmic barrier.}%
    \label{fig:faces_dictionary}%
\end{figure}

\subsection{State-constrained control of an elliptic PDE}\label{ex:pde-control}

We consider the following semilinear elliptic optimal control problem
\begin{subequations}
    \label{eq:opt_problem_pde_e}
    \begin{align}
        \minimize_{y,u}{}&\quad
        \frac{1}{2}\|y-\yd\|_{L^2(\Omega)}^2 + \costp(u) \label{eq:elliptic_cost}\\
        \stt{}&\quad
        -\Delta y+y+\sigma y^3=u \ \text{in }\Omega,\qquad \partial_{\mathbf n} y|_{\partial\Omega}=0, \label{eq:elliptic_state}\\
        &\quad
        y\le \bar y \ \text{in }\Omega . \label{eq:elliptic_state_bound}
    \end{align}
\end{subequations}
Let $\Omega\subset\R^2$ be a bounded convex polygonal domain, and suppose that $\yd\in L^2(\Omega)$, $\bar y\in C(\overline\Omega)$, and $\alpha,\sigma>0$.
Equation~\eqref{eq:elliptic_state} is a special case of \eqref{eq:state_affine_Bf_Neu}, with $U=L^2(\Omega)$, $c=1$, $B=I$, $f=0$, and $N(x,s)=\sigma s^3$.
Therefore, \cref{prop:unified_compact_stability_affine_Bf_Neumann,rem:cubic_class} apply with $p=3$, $a=\sigma$, $b=g=0$, and $\kappa=0$.
In particular, the control-to-state map $\Sctl:L^2(\Omega)\to H^1(\Omega)\cap C(\overline\Omega)$ is well-defined, Fr\'echet differentiable with a locally Lipschitz continuous derivative, and sequentially weak-to-strong continuous.
Moreover, all the remaining conclusions of \cref{prop:unified_compact_stability_affine_Bf_Neumann} hold.

The functional $\costp$ in \eqref{eq:elliptic_cost} combines the quadratic control regularization with the indicator imposing pointwise control constraints.
For $u_a,u_b\in L^\infty(\Omega)$ satisfying $u_a\leq u_b$ a.e., we write
\[
    \Us_{\mathrm{ctl}}
    \coloneqq
    \bigl\{u\in L^2(\Omega):
    u_a(x)\leq u(x)\leq u_b(x)
    \ \text{for a.e. }x\in\Omega\bigr\}
\]
and let $\costp(u)\coloneqq\frac\alpha2\|u\|_{L^2(\Omega)}^2+\indicator_{\Us_{\mathrm{ctl}}}(u)$.
The set $\Us_{\mathrm{ctl}}$ is nonempty, closed, convex, and bounded in $L^2(\Omega)$.
Hence, $\costp$ is proper, convex, and lower semicontinuous, and $\dom\costp=\Us_{\mathrm{ctl}}$ is weakly sequentially closed.
Consequently, \eqref{eq:opt_problem_pde_e} admits the reduced formulation \eqref{eq:Reduced}.
With $\costr(u)=\frac12\|\Sctl(u)-\yd\|_{L^2(\Omega)}^2$, \cref{lem:Applicability} verifies the regularity and compactness properties.

If the reduced feasible set is nonempty, then the reduced objective is bounded from below and coercive, while the weak-to-strong continuity of $\Sctl$ implies that the feasible set is weakly sequentially closed.
Therefore, \cref{thm:existence} applies, and \eqref{eq:opt_problem_pde_e} admits at least one global minimizer.

\paragraph*{Setup}
We set $\Omega=(-1,1)^2$, a bounded convex polygonal domain, and consider three mesh refinements: \texttt{ref5} has 1089 nodes and 2048 elements, \texttt{ref6} has 4225 nodes and 8192 elements, and \texttt{ref7} has 16641 nodes and 32768 elements.
We choose $\sigma=2$ in \eqref{eq:elliptic_state} and $\alpha=0.1$ in \eqref{eq:elliptic_cost}.

In setting \texttt{id0}, we choose $\bar y=2$ and the pointwise control bounds $(u_a,u_b)=(-1,1)$.
In setting \texttt{id1}, we choose $\bar y=0.17$ and retain the same control bounds.
In setting \texttt{id2}, we choose $\bar y=0.17$ and $(u_a,u_b)=(0,0.6)$.

The target state is $\yd(x)=1$ if $\|x\|\le\nicefrac{1}{2}$ and $x_1\ge-\nicefrac{1}{4}$, and $\yd(x)=0$ otherwise.
We run \cref{alg:IP} with tolerance $\varepsilon=10^{-5}$.
The initial barrier parameter is $\nu_1=1$, the decrease factor is $\theta_\nu=\nicefrac{1}{2}$, and the inner tolerance is $\varepsilon_k=\max\{\varepsilon,\nu_k\}$ for each $k\in\N$.

The nonlinear state equation \eqref{eq:elliptic_state} is solved by Newton's method until the residual norm falls below \(10^{-13}\).
The effect of this inexact state solve is not included in the theoretical proximal-gradient residual analysis; the tolerance is chosen sufficiently small for the reported experiments.
Moreover, the experiments below solve finite-dimensional discretizations of this problem. We do not claim a discretization-convergence theorem here; the mesh-refinement study is included as numerical evidence of stability with respect to the discretization.

\paragraph*{Results}
In setting \texttt{id0}, both the state bound $y\le2$ and the control bounds $-1\le u\le1$ are inactive at the computed solution.
\Cref{fig:elliptic_id0_ref6} displays the computed solution and compares the three barrier functions.
As in the preceding dictionary-learning example in \cref{ex:dict-learning}, the logarithmic and log-like barriers require fewer inner iterations than the inverse barrier.
The mesh-refinement results in \cref{fig:elliptic_id0_refining} show the same behavior for all three metrics.

In setting \texttt{id1}, the tighter state bound $y\le0.17$ is active at the computed solution, whereas the control bounds $-1\le u\le1$ remain inactive.
\Cref{fig:elliptic_id1_ref6} reports the results obtained with \cref{alg:IP}.
All three barrier functions require more inner iterations than in setting \texttt{id0}, but their relative behavior remains similar.

In setting \texttt{id2}, the state bound $y\le0.17$ is combined with the tighter control bounds $0\le u\le0.6$; both are active at the computed solution.
\Cref{fig:elliptic_id2_ref6} shows that the three barrier functions require fewer inner iterations than in setting \texttt{id1}, with iteration counts comparable to those in setting \texttt{id0}.

\begin{figure}[!htbp]
    \centering%
    \includegraphics[scale=0.75]{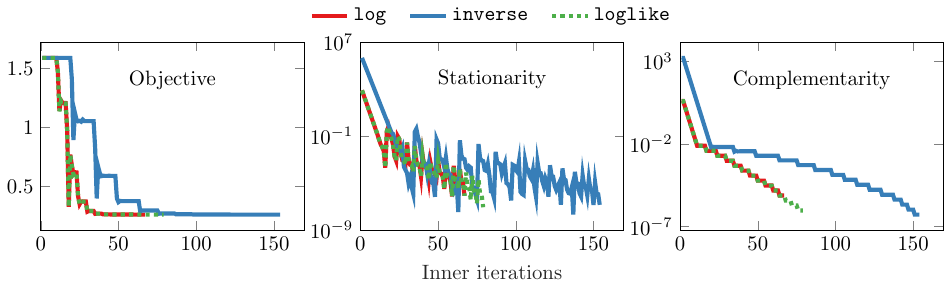}\\
    \includegraphics[scale=0.17]{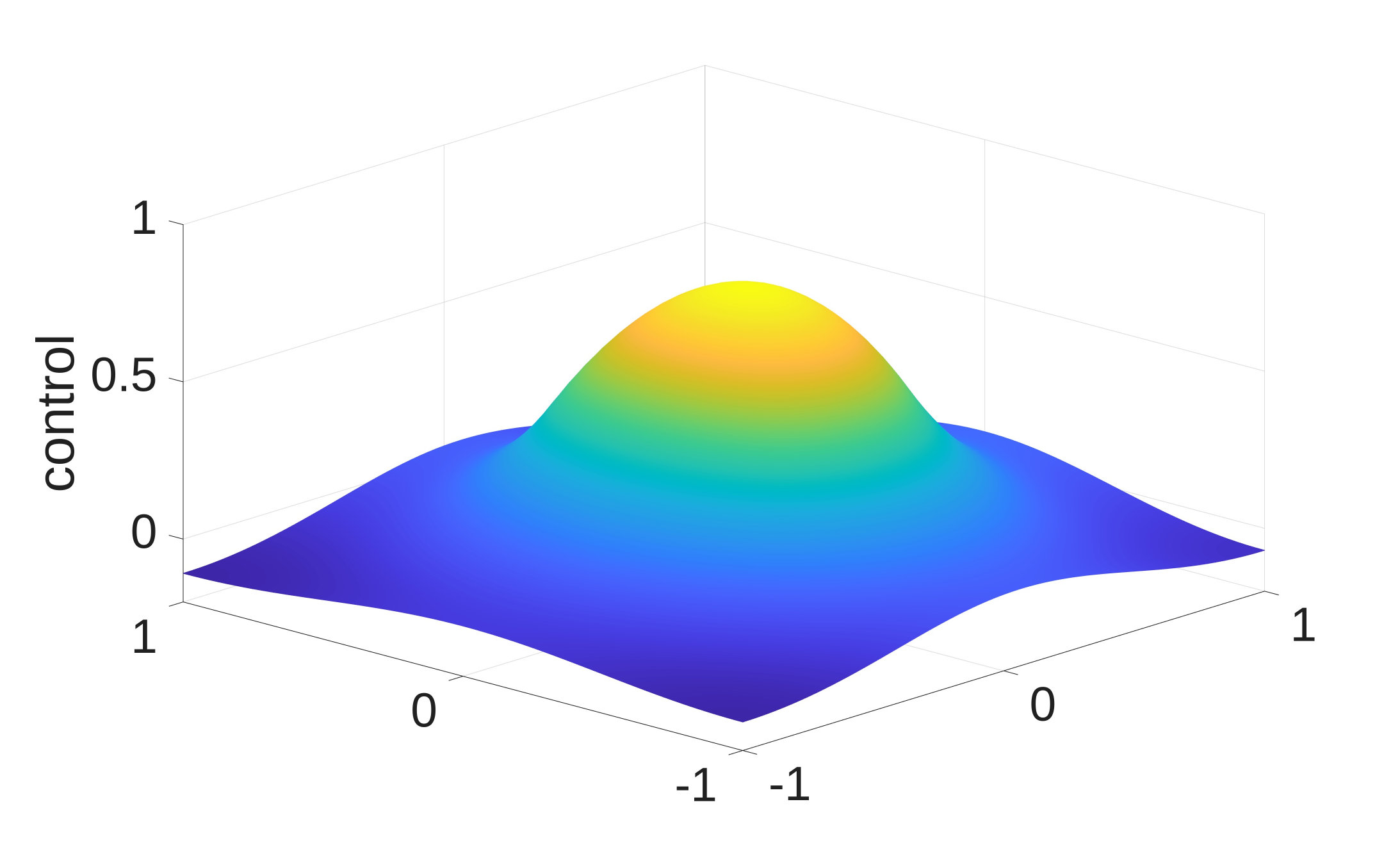}%
    \includegraphics[scale=0.17]{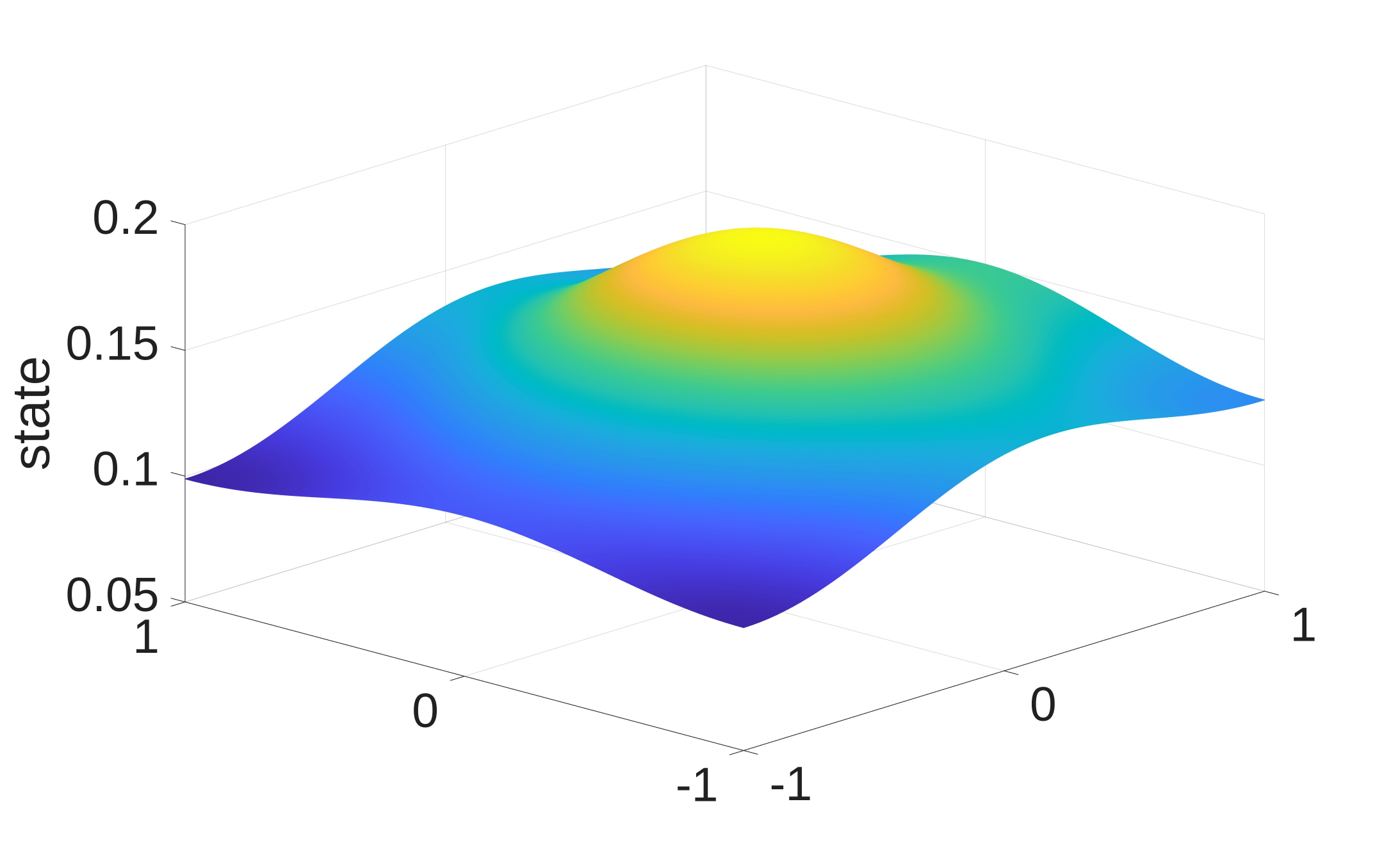}%
    \caption{Elliptic optimal control in setting \texttt{id0} on mesh \texttt{ref6}: objective, stationarity, and complementarity for the three barrier functions (top), and the computed control and state obtained with the logarithmic barrier (bottom).}%
    \label{fig:elliptic_id0_ref6}%
\end{figure}

\begin{figure}[!htbp]
    \centering%
    \includegraphics[scale=0.65]{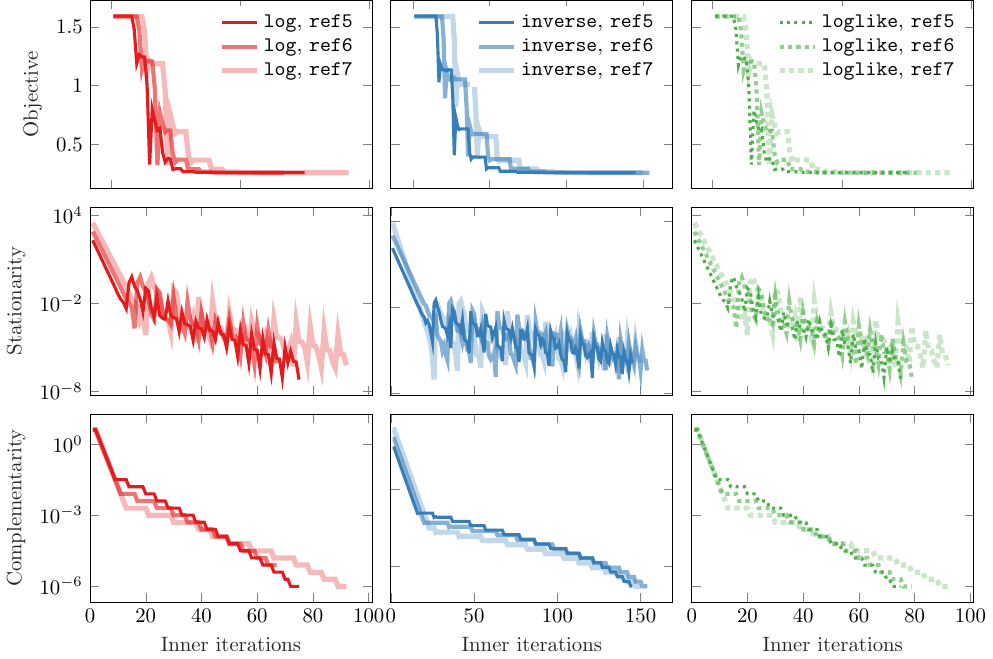}%
    \caption{Elliptic optimal control in setting \texttt{id0}: metrics along the iterations for the three barrier functions and mesh refinements.}%
    \label{fig:elliptic_id0_refining}%
\end{figure}

\begin{figure}[!htbp]
    \centering%
    \includegraphics[scale=0.75]{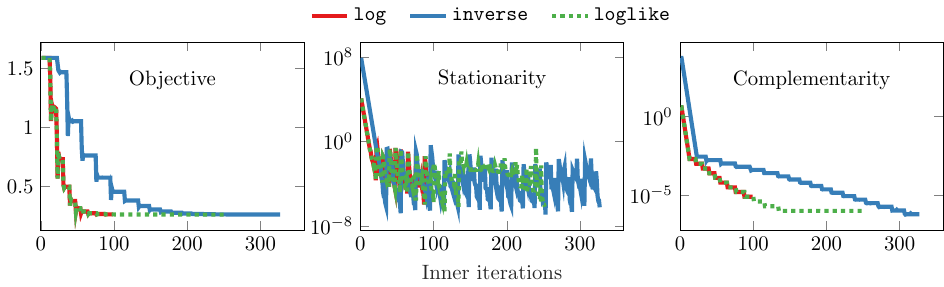}\\
    \includegraphics[scale=0.17]{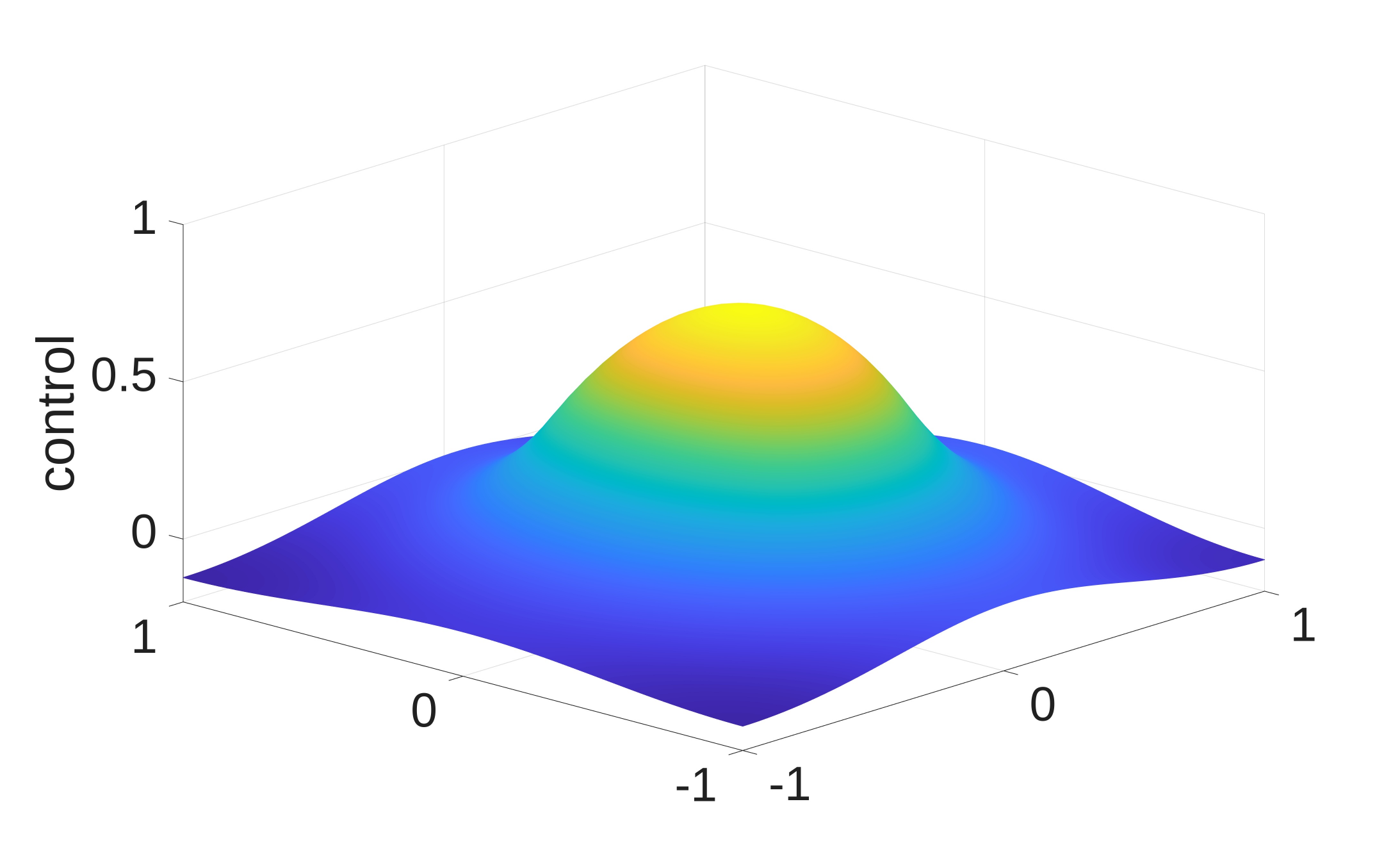}%
    \includegraphics[scale=0.17]{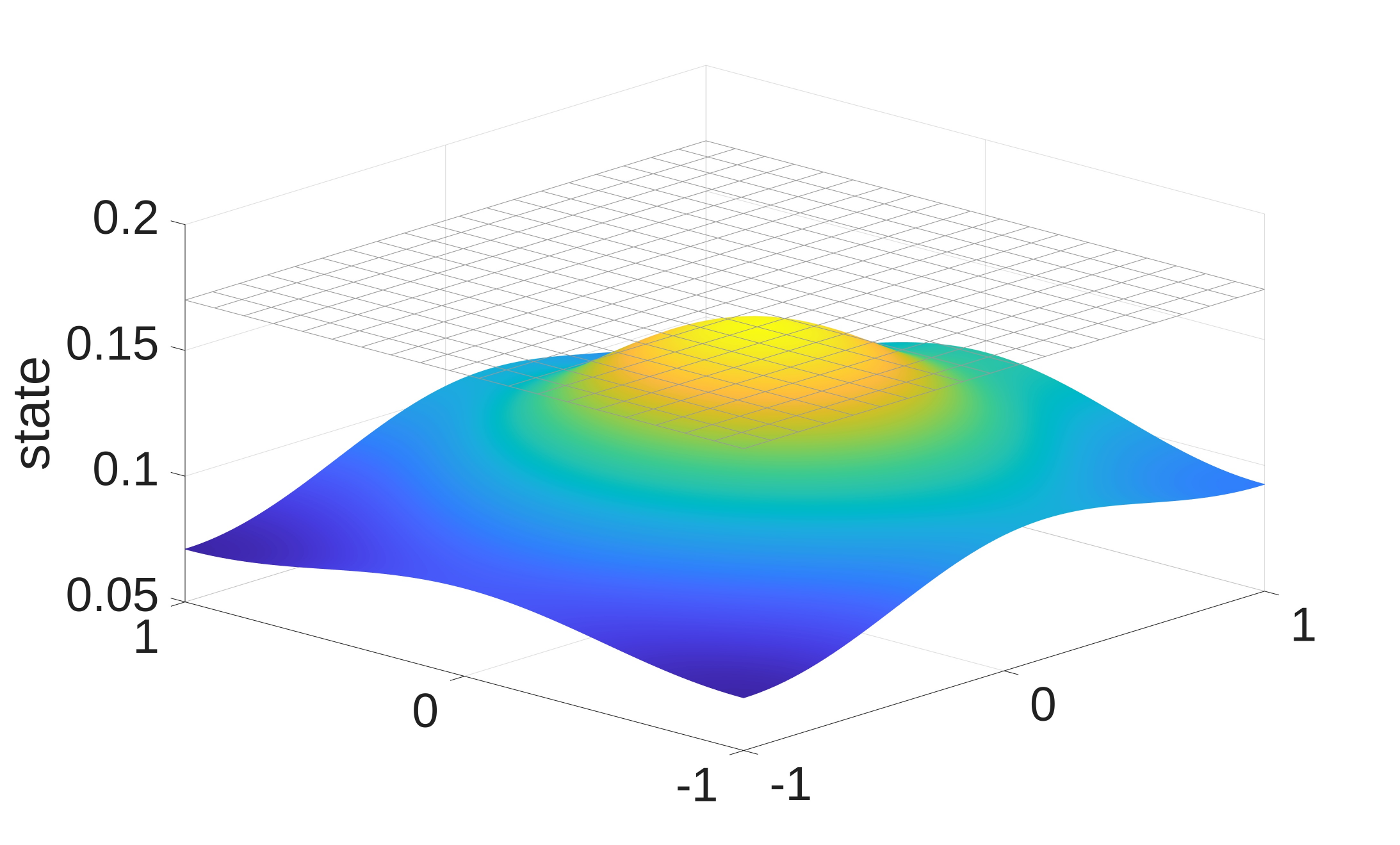}%
    \caption{Elliptic optimal control in setting \texttt{id1} on mesh \texttt{ref6}: objective, stationarity, and complementarity for the three barrier functions (top), and the computed control and state obtained with the logarithmic barrier (bottom). The gray grid in the state plot depicts the upper bound $y=0.17$.}%
    \label{fig:elliptic_id1_ref6}%
\end{figure}

\begin{figure}[!htbp]
    \centering%
    \includegraphics[scale=0.75]{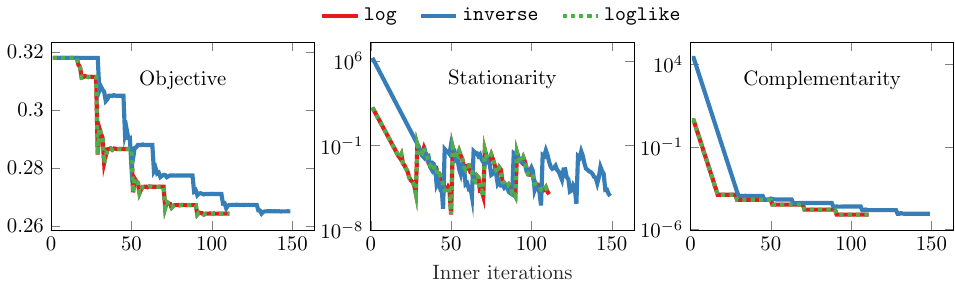}\\
    \includegraphics[scale=0.17]{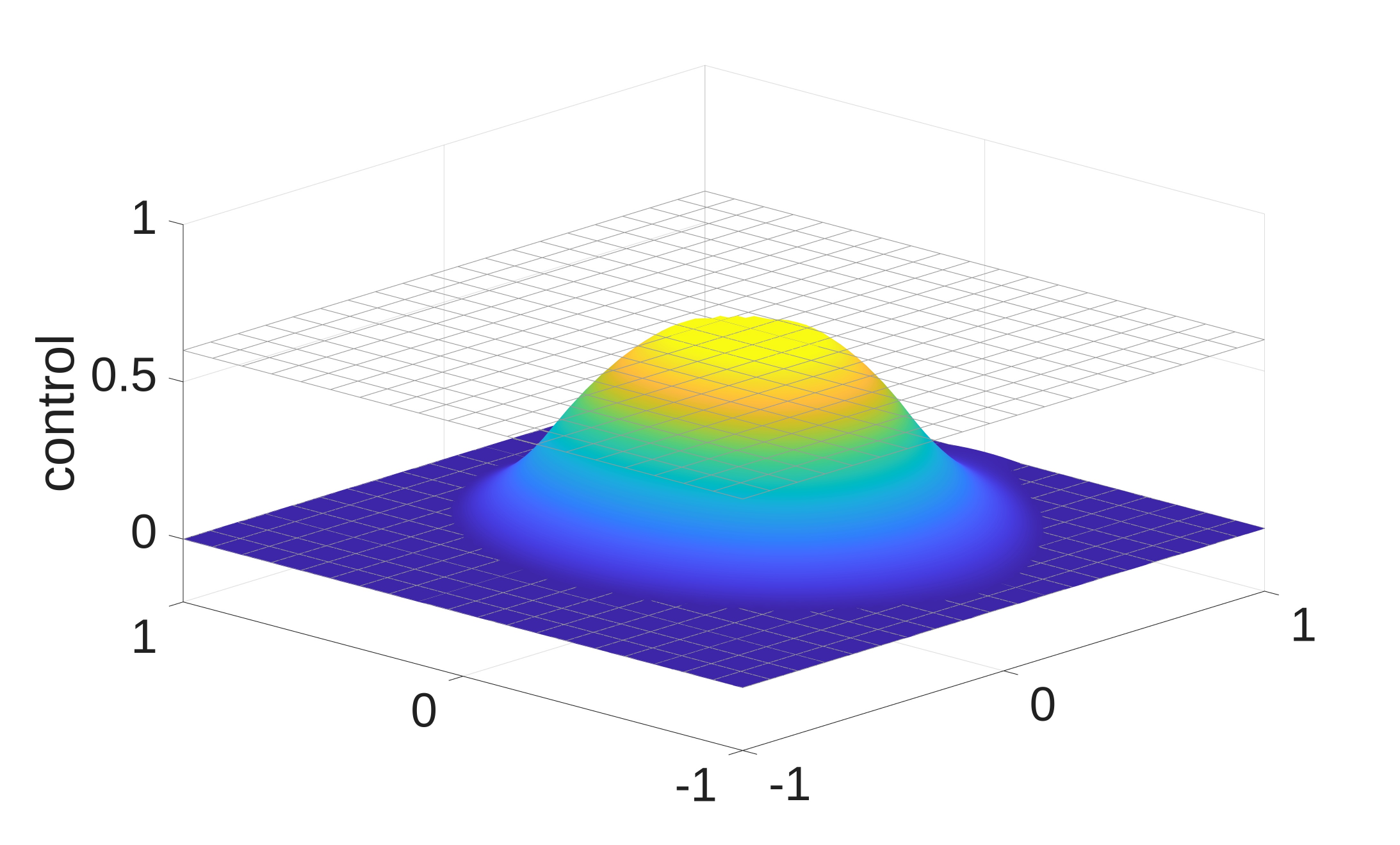}%
    \includegraphics[scale=0.17]{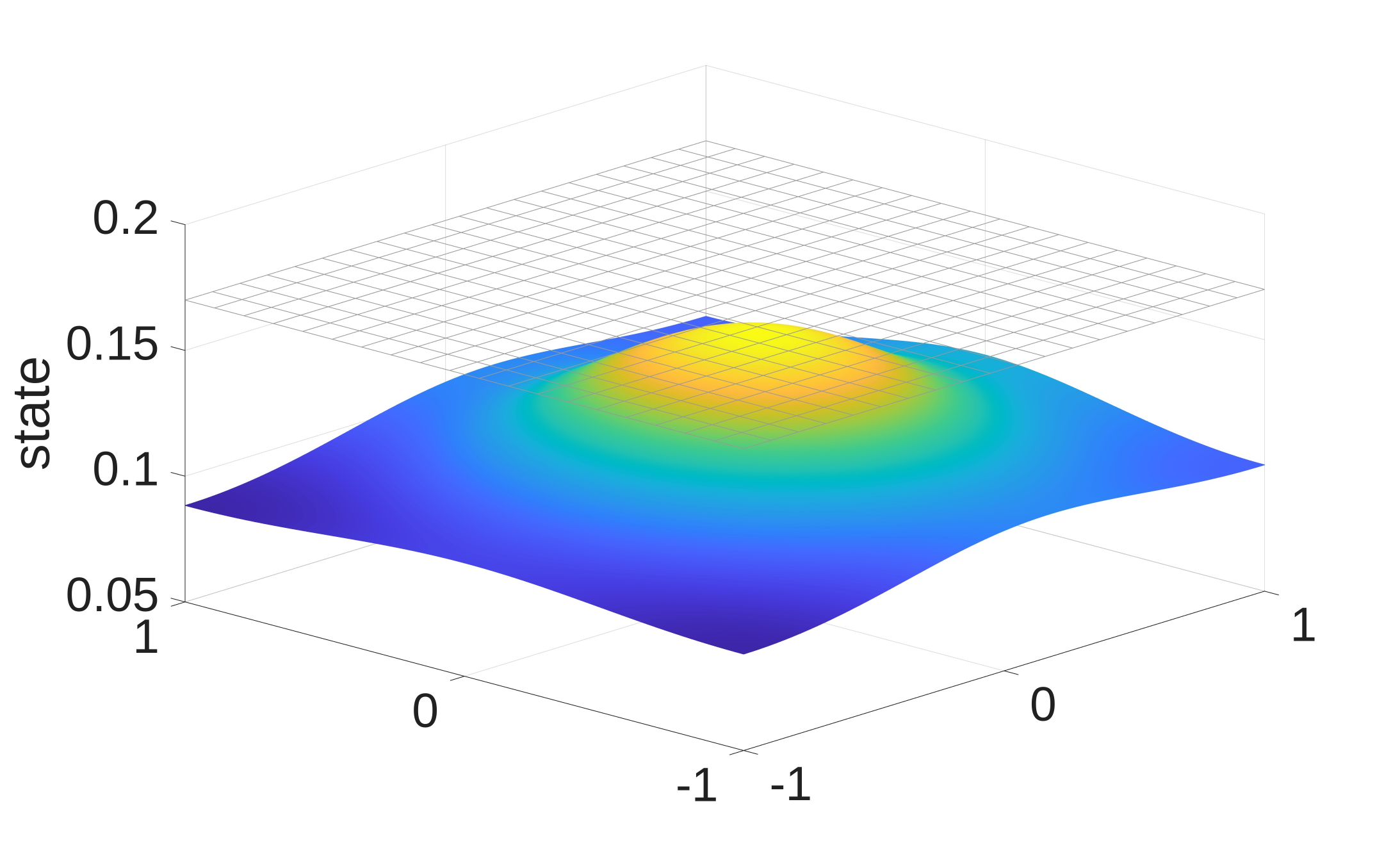}%
    \caption{Elliptic optimal control in setting \texttt{id2} on mesh \texttt{ref6}: objective, stationarity, and complementarity for the three barrier functions (top), and the computed control and state obtained with the logarithmic barrier (bottom). The gray grids depict the control bounds $u=0$ and $u=0.6$ and the state bound $y=0.17$.}%
    \label{fig:elliptic_id2_ref6}%
\end{figure}

\section{Conclusions}\label{sec:conclusion}

We developed an inexact interior-point proximal-gradient framework for nonsmooth, possibly nonconvex optimization with cone-ordered constraints in Hilbert and Banach spaces.
The analysis covers logarithmic and power-type barriers, approximate KKT conditions, weak/weak-$\ast$ multiplier limits, and barrier-dependent inner–outer complexity estimates under explicit regularity assumptions.
The main technical challenge is that the barrier can cause the Lipschitz bound for the smooth part of the barrier subproblem to deteriorate as $\nu\to0$ and as the iterates approach the cone boundary.
Consequently, for both barrier functions, the total complexity is dominated by the final outer iterations. This curvature--complexity tradeoff is a key structural feature of the analysis, and the slower curvature growth of the power barrier translates directly into a milder complexity exponent.
The PDE verification shows that the framework applies to semilinear elliptic state-constrained optimal control with nonsmooth control regularization.
Numerical experiments illustrate the behavior of the proposed method on finite-dimensional and PDE-constrained problems.
For the PDE-constrained problem, the reported relative behavior of the three barrier variants is consistent across the tested mesh-refinement levels.

An important trade-off arises in the function-space setting \(Z=C(\overline K)\). 
For the logarithmic barrier, the required blow-up property at the boundary is not guaranteed in general, so strict feasibility of the inner iterates cannot be inferred from bounded barrier values alone.
For power barriers, this property can be recovered by choosing \(p\) sufficiently large relative to the spatial dimension \(d\) and the H\"older exponent \(\alpha\) of the constraint function, namely such that \(\alpha p>d\).
The outer convergence analysis, however, requires additional control of the barrier subdifferentials, either through a uniform boundedness assumption or further regularity.
In the PDE-constrained examples considered here, the required H\"older regularity follows from the available \(H^2\)-regularity and the Sobolev--Morrey embedding.

\medskip

Several directions are open for future work.
Extending the framework to evolution equations with state constraints requires choosing a state space in which the state constraint is meaningful as an order-cone constraint and verifying the corresponding compactness and differentiability properties of the control-to-state map.
A second direction is the treatment of nonsmooth regularization terms composed with linear operators (e.g., a spatial gradient or finite-difference matrix); this requires extending the proximal-gradient inner loop to handle the composition.
Finally, stochastic and finite-sum variants relevant for large-scale machine learning problems would require stochastic inner-complexity bounds and a corresponding analysis of how variance in the gradient estimates propagates through the outer barrier loop.

\bigskip

\small
\noindent\textbf{Data and code availability.}
Code and experimental results associated with this work are archived on Zenodo and publicly available at \textsc{doi} \href{https://doi.org/10.5281/zenodo.21641166}{10.5281/zenodo.21641166}.
The \emph{Database of Faces} is available online at \url{https://cam-orl.co.uk/facedatabase.html}.

\phantomsection
\addcontentsline{toc}{section}{References}%
\bibliographystyle{habbrv}
\bibliography{biblio}
\normalsize

\appendix

\section{Omitted proofs}\label{sec:omitted_proofs}

\subsection*{Proof of \cref{prop:unified_compact_stability_affine_Bf_Neumann}}\label{proof:verification}

\begin{proof}
We write \(V\coloneqq H^1(\Omega)\), \(\alpha\coloneqq c_0-\kappa>0\), \(L^r\coloneqq L^r(\Omega)\), and \(\|\cdot\|_r\coloneqq\|\cdot\|_{L^r(\Omega)}\).
The notation \(X\hookrightarrow Y\) and \(X\Subset Y\) indicates, respectively, a continuous and a compact embedding.
Throughout the proof, \(C>0\) denotes a generic constant that may change from line to line and depends only on the fixed data and the constants in \ref{ass:E}--\ref{ass:H4}.
On a bounded set \(U_b\subset U\), it may additionally depend on \(\sup_{u\in U_b}\|u\|_U\), but not on the particular \(u\in U_b\).
Moreover, \ref{ass:H1} and \ref{ass:H3} imply
\begin{equation}\label{eq:Ny_lower_bound}
    N_y(x,s)\ge-\kappa
    \qquad
    \text{for a.e. }x\in\Omega\text{ and every }s\in\R.
\end{equation}

\paragraph*{Proof of \ref{stAffE:I}}
We consider the operator $T:V\to V^*$ given by
\[
 \langle T(y),v\rangle
 =
 \int_\Omega \nabla y\cdot\nabla v\,dx
 +\int_\Omega cyv\,dx
 +\int_\Omega N(x,y)v\,dx.
\]
By \ref{ass:H4} and the Sobolev embedding $V\hookrightarrow L^{2p}$ \cite[Cor~9.14]{Brez2011functional} in the stated range of $p$,  we have $N(\cdot,y)\in L^2$.
The Carath\'eodory property in \ref{ass:H1}, the growth condition \ref{ass:H4}, and the standard continuity theorem for Nemytskii operators show that $T$ is continuous. 

For $y_1,y_2\in V$, with $\delta y=y_1-y_2$, \ref{ass:E} and \ref{ass:H3} give
\begin{equation}\label{eq:strong_mono_neu}
    \langle T(y_1)-T(y_2),\delta y\rangle
    \ge
    \|\nabla\delta y\|_2^2
    +\alpha\|\delta y\|_2^2.
\end{equation}
Thus $T$ is strongly monotone and coercive.
The Minty--Browder theorem \cite[Thm~5.16]{Brez2011functional} gives a unique weak solution of $T(y)=Bu+f$.
Taking $t=0$ in \ref{ass:H3} yields $N(x,s)s\ge N(x,0)s-\kappa s^2$.
Testing the state equation with $y$, using $|N(\cdot,0)|\le\eta$, and applying Young's inequality therefore gives
\begin{equation}\label{eq:H1_est1}
   \|y\|_V
   \le C\bigl(1+\|Bu+f\|_2\bigr).
\end{equation}
Furthermore, we have 
\[
    \|N(\cdot,y)\|_2
    \le
    \|\eta\|_2+b_1\|y\|_2+b_2\|y\|_{2p}^p
    \le C\bigl(1+\|Bu+f\|_2\bigr)^q,
\]
where $q = \max\{1,p\}$.
Together with the fact that $Ay=Bu+f-N(\cdot,y)\in L^2$, and using global Neumann regularity \cite[Thms~2.4.2.7 and~3.2.1.3]{Grisvard2011elliptic}, we can write
\begin{equation}\label{eq:H2_reg_neu}
   \|y\|_{H^2(\Omega)}
   \le C\bigl(\|Bu+f-N(\cdot,y)\|_2+\|y\|_2\bigr)
   \le C\bigl(1+\|Bu+f\|_2\bigr)^q.
\end{equation}
Furthermore, since $d\le3$, the Sobolev embedding $H^2(\Omega)\hookrightarrow C(\overline\Omega)$ \cite[Cor~9.15]{Brez2011functional} holds.
Thus using \eqref{eq:H2_reg_neu} and the fact $\|Bu+f\|_2\le\|B\|\|u\|_U+\|f\|_2$, we obtain \eqref{eq:esimate_soltution_semilinear}.

\paragraph*{Proof of \ref{stAffE:II}}
Let $u_1,u_2\in U_b$ be given.
We write $y_i \coloneqq \Sctl(u_i)$, $\delta u \coloneqq u_1-u_2$, and $\delta y \coloneqq  y_1-y_2$.
Subtracting the corresponding state equations, we obtain
\begin{equation}\label{eq:difference_state_neu}
    A\delta y+N(\cdot,y_1)-N(\cdot,y_2)
    =B\delta u
    \quad\text{in }\Omega,
    \qquad
    \left.\partial_{\mathbf n}\delta y\right|_{\partial\Omega}=0.
\end{equation}
Testing \eqref{eq:difference_state_neu} with \(\delta y\) and using \ref{ass:E} and \ref{ass:H3}, we obtain
\[
    \|\nabla\delta y\|_2^2
    +\alpha\|\delta y\|_2^2
    \le
    \|B\delta u\|_2\|\delta y\|_2.
\]
Consequently, we have 
\begin{equation}\label{eq:H1_Lip_neu}
   \|\delta y\|_V
   \le C\|\delta u\|_U,
   \qquad
   \|\delta y\|_2
   \le\alpha^{-1}\|B\|\|\delta u\|_U.
\end{equation}
By \ref{stAffE:I}, the states corresponding to $U_b$ are uniformly bounded in $L^\infty$.
For a corresponding bound $M$, let
\[
   a_{12}(x)
   \coloneqq
   \int_0^1N_y\bigl(x,y_2+t\delta y\bigr)\,dt.
\]
Then using the integral mean-value formula, we obtain
\[
   N(\cdot,y_1)-N(\cdot,y_2)=a_{12}\delta y,
   \qquad \text{ with } \qquad
   \|a_{12}\|_\infty\le L_M.
\]
Thus, we can write  $A\delta y=B\delta u-a_{12}\delta y$.
The global $H^2$-estimate and \eqref{eq:H1_Lip_neu} imply
\begin{equation}\label{eq:H2_Lip_neu}
   \|\delta y\|_{H^2(\Omega)}
   \le C\|\delta u\|_U.
\end{equation}
Combining \eqref{eq:H1_Lip_neu},
\eqref{eq:H2_Lip_neu}, and the embedding
$H^2(\Omega)\hookrightarrow C(\overline\Omega)$ gives
\[
    \|\delta y\|_V+\|\delta y\|_{C(\overline\Omega)}
    \leq L_S(U_b)\|\delta u\|_U,
\]
which proves \ref{stAffE:II}.

\paragraph*{Proof of \ref{stAffE:III}}
First, we fix $u\in U$ and write $y \coloneqq \Sctl(u)$.
For $h\in U$, let $z_h$ solve \eqref{eq:lin_neu}.
By \eqref{eq:Ny_lower_bound}, the bilinear form of the linearized equation satisfies
\[
 \int_\Omega |\nabla v|^2\,dx
 +\int_\Omega\bigl(c+N_y(x,y)\bigr)v^2\,dx
 \ge
 \|\nabla v\|_2^2+\alpha\|v\|_2^2.
\]
Lax--Milgram and the global $H^2$-estimate therefore give
\begin{equation}\label{eq:linearized_H2_bound}
   \|z_h\|_{H^2(\Omega)}
   \le C\|h\|_U.
\end{equation}
Hence $h\mapsto z_h$ is a bounded linear map from $U$ to $V\cap C(\overline\Omega)$.
For $v\in U$ tending to zero, we write
\[
   \delta y \coloneqq \Sctl(u+v)-\Sctl(u),
   \qquad
   r_v \coloneqq \delta y-z_v.
\]
Taylor's formula then gives
\[
 \rho_v
 \coloneqq
 N(\cdot,y+\delta y)-N(\cdot,y)-N_y(\cdot,y)\delta y
 =
 \int_0^1
 \bigl[N_y(\cdot,y+t\delta y)-N_y(\cdot,y)\bigr]\delta y\,dt.
\]
For $v$ sufficiently small, the states remain in a fixed bounded interval.
Therefore \ref{ass:H2} and \eqref{eq:H1_Lip_neu}--\eqref{eq:H2_Lip_neu} yield
\[
 \|\rho_v\|_2
 \le
 \frac{\ell_M}{2}\|\delta y\|_\infty\|\delta y\|_2
 \le C\|v\|_U^2.
\]
The remainder solves
\[
   Ar_v+N_y(\cdot,y)r_v=-\rho_v \quad\text{in }\Omega,
   \qquad
  \left.\partial_{\mathbf n} r_v\right|_{\partial\Omega}=0.
\]
Thus, applying the coercivity estimate and the global \(H^2\)-regularity argument used to obtain \eqref{eq:linearized_H2_bound}, now with \(-\rho_v\) in place of \(Bh\), and using \(H^2(\Omega)\hookrightarrow C(\overline\Omega)\), we obtain
\[
    \|r_v\|_V+\|r_v\|_{C(\overline\Omega)}
    \le C\|\rho_v\|_2
    =o(\|v\|_U),
\]
which proves Fr\'echet differentiability and the formula in \eqref{eq:lin_neu}.

\paragraph*{Proof of \ref{stAffE:IV}}
For $u\in U_b$, we write $y\coloneqq\Sctl(u)$.
By \ref{stAffE:I}, there exists $M=M(U_b)>0$ such that
\[
    \|y\|_\infty\le M
    \qquad\forall u\in U_b.
\]
The coercivity estimate and \eqref{eq:linearized_H2_bound}, with constants uniform on $U_b$, give
\[
   \|\Sctl'(u)h\|_V
   +\|\Sctl'(u)h\|_{C(\overline\Omega)}
   \le C\|h\|_U.
\]
Now, let $u_1,u_2\in U_b$ be given.
We write $y_i\coloneqq\Sctl(u_i)$, $z_i\coloneqq \Sctl'(u_i)h$, and $w \coloneqq z_1-z_2$.
Then
\begin{equation}\label{eq:Prop3_w_equation}
    Aw+N_y(\cdot,y_2)w
    =
    \bigl[N_y(\cdot,y_2)-N_y(\cdot,y_1)\bigr]z_1  \quad \text{ in } \Omega,   \qquad
    \left.\partial_{\mathbf n} w\right|_{\partial\Omega}=0.
\end{equation}
Using \ref{ass:H2}, \ref{stAffE:II}, and the uniform bound for $z_1$, we obtain
\[
    \begin{aligned}
    \bigl\|
    \bigl[N_y(\cdot,y_2)-N_y(\cdot,y_1)\bigr]z_1
    \bigr\|_2
    &\le
    \ell_M\|y_2-y_1\|_\infty\|z_1\|_2\le
    C\|u_2-u_1\|_U\|h\|_U.
    \end{aligned}
\]
Coercivity and elliptic regularity applied to \eqref{eq:Prop3_w_equation} give
\[
 \|w\|_V+\|w\|_{C(\overline\Omega)}
 \le C\|u_2-u_1\|_U\|h\|_U.
\]
Taking the supremum over $\|h\|_U=1$ proves \ref{stAffE:IV}.

\paragraph*{Proof of \ref{stAffE:V}}
Let $u_k\rightharpoonup u$ in $U$.
We write $y_k \coloneqq \Sctl(u_k)$ and $y \coloneqq \Sctl(u)$.
Then, the sequence $(y_k)$ is bounded in $H^2(\Omega)$ by \ref{stAffE:I}.
Hence, after passing to a subsequence, there exists \(\bar y\in H^2(\Omega)\) such that $y_k\rightharpoonup\bar y$ weakly in $H^2(\Omega)$.
Since $H^2(\Omega)\Subset V\cap C(\overline\Omega)$ for $d\le 3$ due to \cite[Cor~9.15 and Thm~9.16]{Brez2011functional}, the same subsequence converges strongly in $V\cap C(\overline\Omega)$ to $\bar y$.
On such a subsequence, \ref{ass:H2} gives
\[
   \|N(\cdot,y_k)-N(\cdot,\bar y)\|_2
   \le L_M|\Omega|^{1/2}\|y_k-\bar y\|_\infty
   \to0.
\]
Passing to the limit in the weak state equation shows that $\bar y$ solves the equation corresponding to $u$.
Uniqueness gives $\bar y=y$. Since the same argument applies to
every subsequence of \((y_k)\), the subsequence principle yields $y_k\to y$ strongly in $V\cap C(\overline\Omega)$ for the full sequence.

Next, we fix $h\in U$ and write $z_k \coloneqq \Sctl'(u_k)h$, $z \coloneqq \Sctl'(u)h$, and $w_k \coloneqq z_k-z$.
The uniform state bound and \ref{ass:H2} imply that
\begin{equation}\label{eq:Ny_uniform_convergence}
   \|N_y(\cdot,y_k)-N_y(\cdot,y)\|_\infty
   \le\ell_M\|y_k-y\|_\infty
   \to0.
\end{equation}
Moreover, for every $k\in \N_0$, $w_k$ solves
\[
    Aw_k+N_y(\cdot,y)w_k
    =
    \bigl[N_y(\cdot,y)-N_y(\cdot,y_k)\bigr]z_k\quad\text{in }\Omega,
    \qquad
    \left.\partial_{\mathbf n} w_k\right|_{\partial\Omega}=0.
\]
Since $(z_k)$ is uniformly bounded in $H^2(\Omega)$, we infer that 
\[
 \bigl\|
   \bigl[N_y(\cdot,y)-N_y(\cdot,y_k)\bigr]z_k
 \bigr\|_2
 \le
 \|N_y(\cdot,y)-N_y(\cdot,y_k)\|_\infty\|z_k\|_2
 \to0.
\]
Coercivity and elliptic regularity yield $w_k\to0$ in $H^2(\Omega)$, and hence in $V\cap C(\overline\Omega)$.

\paragraph*{Proof of \ref{stAffE:VI}}
Let $u_k\rightharpoonup u$ in $U$.
We write $y_k\coloneqq\Sctl(u_k)$, $y\coloneqq\Sctl(u)$, $q_k\coloneqq c+N_y(\cdot,y_k)$, and $q\coloneqq c+N_y(\cdot,y)$.
Since $(u_k)$ is bounded in $U$, \ref{stAffE:I} and \ref{ass:H2} imply that $(q_k)$ and $q$ are uniformly bounded in $L^\infty$.
By \eqref{eq:Ny_lower_bound}, we have also that \(q_k,q\geq\alpha>0\).

For \(d\geq2\), we choose \(r\) such that \(2d/(d+2)<r<d/(d-1)\); for \(d=1\), any suitable \(r>1\) may be chosen.
By the classical Stampacchia theory for elliptic equations with measure data \cite[Section~3.1, in particular Remark~3.1 and Thm~3.1]{Droniou2000}, we denote by
\(p_k,p\in W^{1,r}(\Omega)\) the corresponding Stampacchia solutions, characterized by
\begin{equation}\label{eq:adjoint_weak_text}
    \int_\Omega \nabla p_k\cdot\nabla v\,dx
    +\int_\Omega q_kp_kv\,dx
    =
    \langle\mu,v\rangle_
    {\mathcal M(\overline\Omega),C(\overline\Omega)}
    \qquad
    \forall v\in C^\infty(\overline\Omega),
\end{equation}
and by the analogous identity with \(q_k,p_k\) replaced by \(q,p\).
The corresponding measure-data estimate gives
\begin{equation}\label{eq:W1r_bound_pk}
    \|p_k\|_{W^{1,r}(\Omega)}
    +\|p\|_{W^{1,r}(\Omega)}
    \leq
    C\|\mu\|_{\mathcal M(\overline\Omega)},
\end{equation}
where \(C\) is independent of \(k\).      

Next, for \(h\in U\), we write \(z_k \coloneqq \Sctl'(u_k)h\).
The duality relation characterizing \(p_k\), together with the linearized state equation, gives
\[
    \langle\mu,z_k\rangle_
    {\mathcal M(\overline\Omega),C(\overline\Omega)}
    =
    \int_\Omega p_kBh\,dx.
\]
Consequently, we can write that 
\begin{equation}\label{eq:adjoint_representation}
    \Sctl'(u_k)^*\mu=B^*p_k,
    \qquad
    \Sctl'(u)^*\mu=B^*p.
\end{equation}
By \ref{stAffE:V}, \(y_k\to y\) in \(C(\overline\Omega)\), and hence, we can conclude that $\|q_k-q\|_\infty\to 0$.
Estimate \eqref{eq:W1r_bound_pk} and the compact embedding $W^{1,r}(\Omega)\Subset L^2$ imply that every subsequence of \((p_k)\) has a further subsequence
converging strongly in \(L^2\).
Passing to the limit in \eqref{eq:adjoint_weak_text} identifies every such limit with \(p\).
By uniqueness and the subsequence principle, we conclude that $ p_k\to p$ strongly in $L^2$.
Together with \eqref{eq:adjoint_representation}, this proves \eqref{eq:adjoint_pointwise_conv}.

Finally, suppose that \(\{u_k\}\) is bounded in \(U\) and that \(\barrierfun\subset\mathcal M(\overline\Omega)\) is bounded in total variation.
By \eqref{eq:W1r_bound_pk}, the family  $\{\,p_{k,\mu}:k\in\N,\ \mu\in\barrierfun\,\}$ is bounded in \(W^{1,r}(\Omega)\) and therefore relatively compact in
\(L^2\).
Since \(\Sctl'(u_k)^*\mu=B^*p_{k,\mu}\) and \(B^*:L^2\to U\) is bounded, \(\mathcal K(\barrierfun)\) is relatively compact in \(U\), which proves the collective compactness.
\end{proof}

\end{document}